\documentclass{amsart}

\usepackage{amssymb, amsmath,mathtools}
\usepackage{mathrsfs}
\usepackage{amscd}
\usepackage{verbatim}
\usepackage{stmaryrd}
\usepackage[dvipsnames]{xcolor}
\usepackage{url}

\usepackage{enumerate}

\usepackage[colorlinks,linkcolor={blue},citecolor={blue},urlcolor={blue},]{hyperref}

\usepackage[colorinlistoftodos,prependcaption,textsize=tiny]{todonotes}

\usepackage{forest}

\usepackage{tabularx}
\newcolumntype{L}{>{\arraybackslash}X}
\usepackage{multirow}

\theoremstyle{plain}
\newtheorem*{maintheorem}{Main Result}
\newtheorem{theorem}{Theorem}[section]
\theoremstyle{remark}
\newtheorem{remark}[theorem]{Remark}

\theoremstyle{plain}

\newtheorem{lemma}[theorem]{Lemma}
\newtheorem{proposition}[theorem]{Proposition}
\newtheorem{definition}[theorem]{Definition}

\newtheorem{assumption}[theorem]{Assumption}

\numberwithin{equation}{section}

\def\N{{\mathbb N}}

\def\R{{\mathbb R}}

\newcommand{\one}{{{\bf 1}}}

\newcommand{\E}{{\mathsf E}}
\renewcommand{\P}{{\mathsf P}}
\newcommand{\F}{{\mathscr F}}

\newcommand{\om}{\omega}
\renewcommand{\O}{\Omega}

\newcommand{\n}{{\rm N}}

\newcommand{\Dom}{\mathcal{O}}

\renewcommand{\emptyset}{\varnothing}

\newcommand{\Tor}{\mathbb{T}}

\newcommand{\A}{{\mathcal A}}

\newcommand{\loc}{\mathrm{loc}}

\newcommand{\calL}{{\mathscr L}}

\newcommand{\h}{{\rm H}}

\newcommand{\Ls}{\mathbb{L}}

\newcommand{\Hs}{\mathbb{H}}

\newcommand{\e}{\mathcal{E}}

\newcommand{\wt}{\widetilde}

\newcommand{\Do}{\mathsf{D}}

\usepackage{stmaryrd}

\newcommand{\forcetwo}{\mathcal{E}}
\newcommand{\force}{\mathcal{F}}
\renewcommand{\div}{\mathrm{div}}
\newcommand{\w}{w}

\newcommand{\p}{\mathbb{P}}
\newcommand{\q}{\mathbb{Q}}
\newcommand{\embed}{\hookrightarrow}

\renewcommand{\l}{\langle}
\renewcommand{\r}{\rangle}
\newcommand{\Progress}{\mathscr{P}}

\newcommand{\Temp}{\theta}
\newcommand{\T}{\Temp}

\newcommand{\op}{\mathcal{J}}
\newcommand{\Ttwo}{\theta^{\tau}}
\newcommand{\Tthree}{\varphi}
\newcommand{\ellip}{\nu}

\newcommand{\vtau}{v^{\tau}}
\newcommand{\Hr}{H_{{\rm R}}}
\newcommand{\kone}{\kappa}

\newcommand{\Borel}{\mathscr{B}}
\newcommand{\Dr}{\Delta_{{\rm R}}^{{\rm w}}}

\newcommand{\norm}{\mathcal{N}}

\newcommand{\pr}{\mathbb{P}_{\h}}
\newcommand{\qr}{\mathbb{Q}_{\h}}
\newcommand{\ft}{F_{\T}}
\newcommand{\fvt}{f}
\newcommand{\gvtn}{g_n}

\newcommand{\fv}{F_{v}}

\newcommand{\Lv}{\mathcal{L}_v}
\newcommand{\LTphi}{\mathcal{L}_{\psi}}
\newcommand{\Lvphi}{\mathcal{L}_{\phi}}
\newcommand{\Lvp}{\mathcal{P}_{\phi}}
\newcommand{\Lp}{\mathcal{P}_{\hp}}
\newcommand{\Lpp}{\mathcal{P}_{\hp,\phi}}
\newcommand{\Lpg}{\mathcal{P}_{\hp,G}}
\newcommand{\hp}{\gamma}
\newcommand{\LT}{\mathcal{L}_{\T}}
\newcommand{\LTw}{\mathcal{L}_{\T}^{{\rm w}}}

\newcommand{\Br}{\mathrm{B}}

\newcommand{\gtn}{G_{\T,n}}
\newcommand{\gvn}{G_{v,n}}
\newcommand{\Ff}{{\rm F}}
\newcommand{\M}{\mathrm{M}_m}
\newcommand{\vz}{v_3}
\newcommand{\y}{\Xi}

\newcommand{\Ib}{\overline{I}}

\allowdisplaybreaks

\begin{document}


\date\today

\title[Primitive equations with transport noise and turbulent pressure]{The stochastic primitive equations with transport noise and turbulent pressure}




\author[Agresti]{Antonio Agresti}
\address{Department of Mathematics,
	TU Kaiserslautern, Paul-Ehrlich-Stra{\ss}e 31,
	67663 Kaiserslautern, Germany}
\email{antonio.agresti92@gmail.com}
\curraddr{Institute of Science and Technology Austria (ISTA), Am Campus 1, 3400 Klosterneuburg, Austria}

\author[Hieber]{Matthias Hieber} 
\address{Department of Mathematics,
	TU Darmstadt, Schlossgartenstr. 7, 64289 Darmstadt, Germany}
\email{hieber@mathematik.tu-darmstadt.de}

\author[Hussein]{Amru Hussein}
\address{Department of Mathematics,
	TU Kaiserslautern, Paul-Ehrlich-Stra{\ss}e 31,
	67663 Kaiserslautern, Germany}
\email{hussein@mathematik.uni-kl.de}

\author[Saal]{Martin Saal}
\address{Scuola Normale Superiore, Piazza dei Cavalieri 7, 56126 Pisa, Italy}
\email{msaal@mathematik.tu-darmstadt.de}
\curraddr{TU Darmstadt, Schlossgartenstr. 7, 64289 Darmstadt, Germany}

\thanks{The first author has  been supported partially by the Nachwuchsring – Network for the promotion of young scientists – at TU Kaiserslautern. The first and the third authors have been  supported by MathApp – Mathematics Applied to Real-World Problems - part of the Research Initiative of the Federal State of Rhineland-Palatinate, Germany.
The fourth author gratefully acknowledges the financial support of the Deutsche Forschungsgemeinschaft (DFG) through the research fellowship SA 3887/1-1.
}

\begin{abstract}
In this paper we consider the stochastic primitive equation for geophysical flows  subject to transport noise and turbulent pressure.
Admitting very rough noise terms, the global existence and uniqueness of solutions to this stochastic partial differential equation are proven using stochastic maximal $L^2$-regularity, the theory of critical spaces for stochastic evolution equations, and global a priori bounds. 
Compared to other results in this direction, we do not need any smallness assumption on the transport noise which acts directly on the velocity field and we also allow rougher noise terms.
The adaptation to Stratonovich type noise  and, more generally, to variable viscosity and/or conductivity are discussed as well.   
\end{abstract}

\maketitle

\addtocontents{toc}{\protect\setcounter{tocdepth}{1}}
\tableofcontents

\section{Introduction}\label{s:intro}
In this paper we study the stochastic primitive equation with transport noise and turbulent pressure. 
The primitive equations are one of the fundamental models 
for geophysical flows  used to describe
oceanic and atmospheric dynamics. They are derived from the
Navier-Stokes equations in domains where the vertical scale
is much smaller than the horizontal scale by the small aspect ratio limit.
Detailed information on the geophysical background for the various versions of the deterministic
primitive equations are given e.g.\ in \cite{Ped, Vallis06}. 
The introduction of additive and multiplicative noise
into models for geophysical flows  can  
 be used on the one hand 
 to account for numerical
and empirical uncertainties and errors
and on the other hand as 
subgrid-scale parameterizations for data assimilation, and ensemble prediction 
as described in the review articles \cite{Delsole04, Franzke14, Palmer19}.

In this paper we are mainly concerned with stochastic perturbations of transport type. In the study of turbulent flows,  transport noise has been introduced by R.H.~Kraichanan in 
\cite{K68,K94} and has been widely studied in the context of stochastic Navier-Stokes equations, see \cite{MR01,MR04} for a physical justification and also  \cite{BCF91,BCF92,MR05,F_intro,AV21_NS} and the references therein for related mathematical results.
The aim of this paper is to give a  systematic and detailed treatment of transport noise in the context of the primitive equations.

Here, we consider the following stochastic primitive equations on  the cylindrical domain $\Dom=\Tor^2\times (-h,0)$, where $h>0$, $\Tor^2$ is the two-dimensional torus, and we denote the coordinates by $(x_{\h},x_3)\in\Tor^2\times (-h,0)$ with the horizontal part $x_{\h}=(x_1,x_2)\in\Tor^2$, $\div_{\h}=\partial_1+\partial_2$ and $\nabla_{\h}=(\partial_1,\partial_2)$, i.e., the system
\begin{equation}
	\label{eq:primitive_full}
	\begin{cases}
		\displaystyle{d v -\Delta v\, dt=\Big[-\nabla_{\h} P  -(v\cdot \nabla_{\h})v- w\partial_3 v + \fv
		+\partial_{\hp} \wt{P} \Big]dt }\\
		\qquad \qquad \ \ \ \ \ \qquad   \ \ \   
		\displaystyle{+\sum_{n\geq 1}\Big[(\phi_{n}\cdot\nabla) v-\nabla_{\h}\wt{P}_n  +\gvn\Big] d\beta_t^n}, 
		&\text{on }\Dom,\\
		\displaystyle{d \T -\Delta \T\, dt=\Big[ -(v\cdot \nabla_{\h})\T- w\partial_3 \T+ \ft \Big]dt} \\
		\qquad \qquad\qquad\qquad\ \ \ \ \ \qquad
		\displaystyle{+\sum_{n\geq 1}\Big[(\psi_{n}\cdot\nabla) \T+\gtn \Big] d\beta_t^n}, 
		&\text{on }\Dom,\\
		\partial_{3} P+ \kone \T=0, &\text{on }\Dom,\\
		\partial_{3} \wt{P}_n=0, &\text{on }\Dom,\\
		\div_{\h} v+\partial_3 w=0,&\text{on }\Dom, \\ 
		v(\cdot,0)=v_0,\ \ \T(\cdot,0)=\T_0, &\text{on }\Dom.
	\end{cases}
\end{equation}
In the system \eqref{eq:primitive_full} the unknowns are: the velocity field $u=(v,w)\colon [0,\infty)\times \O\times \Dom\to \R^3$ where $v=(v^k)_{k=1}^2: [0,\infty)\times \O\times \Dom\to \R^2$ is the horizontal part of the velocity, the pressure $P\colon[0,\infty)\times \O\times  \Dom\to \R$, the components of the turbulent pressure  $\wt{P}_n\colon[0,\infty)\times \O\times  \Tor^2 \to \R$ and the temperature $\T\colon[0,\infty)\times \O\times \Dom\to \R$. 
The driving processes $(\beta_t^n\colon t\geq 0)_{n\geq 1}$ are given by a sequence of independent standard Brownian motions on some filtered probability space $(\O,\mathcal{A},(\F_t)_{t\geq 0},\P)$. 
Moreover, $\phi_{n}=(\phi^j_n)_{j=1}^3,\psi_{n}=(\psi_n^j)_{j=1}^3\colon[0,\infty)\times \O\times \Dom\to \R^3$, $\kone\colon [0,\infty)\times \O\times \Dom\to \R$  and $\hp_{n}=(\hp_{n}^{\ell,m})_{\ell,m=1}^2 \colon [0,\infty)\times \O\times \Dom\to \R^{2\times2}$ are given functions and 
 $$
 \partial_{\hp} \wt{P}:=\Big(\sum_{n\geq 1}\sum_{m=1}^2 \hp^{\ell,m}_n \partial_m \wt{P}_n\Big)_{\ell=1}^2
 $$ 
 describe the deterministic effect of the turbulent pressure.
Finally, $\fv,\ft,\gvn$ and $\gtn$ are given maps depending on $v,\T,\nabla v$ and $\nabla \T$ describing deterministic and stochastic forces also taking into account lower order effects like the Coriolis force. 
For a physical justification of the occurrence of the turbulent pressure $\wt{P}_n$, i.e., a pressure term within the stochastic integral and for the related deterministic contribution $\partial_{\hp} \wt{P}$, we refer to \cite{MR01} and \cite[Example 1]{MR04}.

The system \eqref{eq:primitive_full} is supplemented with  Neumann boundary conditions for the horizontal velocity $v$ and mixed Neumann-Robin boundary conditions for the temperature $\T$ on the top $\Tor^2\times\{0\}$ and the bottom $\Tor^2\times\{-h\}$ of the domain $\Dom$, i.e.,
\begin{equation}
	\label{eq:boundary_conditions}
	\begin{aligned}
		\partial_3 v (\cdot,-h)=\partial_3 v(\cdot,0)=0 \ \  \text{ on }\Tor^2,&\\
		\partial_3 \T(\cdot,-h)= \partial_3 \T(\cdot,0)+\alpha \T(\cdot,0)=0\ \  \text{ on }\Tor^2,&
	\end{aligned}
\end{equation}
where the parameter $\alpha\in \R$ is a given constant, and for the vertical velocity 
\begin{equation}
	\label{eq:boundary_conditions_w}
	\begin{aligned}
		w(\cdot,-h) =w(\cdot,0)=0\ \ \text{ on }\Tor^2.
	\end{aligned}
\end{equation}
In the horizontal directions periodicity is assumed. The results of the current paper also hold in case \eqref{eq:boundary_conditions} are replaced by periodic boundary conditions, see Remark \ref{r:global_primitive}\eqref{it:global_primitive_periodic}.
When modeling the ocean, the system \eqref{eq:primitive_full} is expanded by an equation for the salinity. This leads to terms resembling the ones for the temperature, and this does not lead to additional mathematical difficulties or more restrictive assumptions. Therefore we omit this coupling here to concentrate on the main features.

The mathematical analysis of the {\em deterministic} primitive equations (i.e., $\beta^n_t=0$ in \eqref{eq:primitive_full}) has been pioneered by J.L.~Lions, R.~Teman and S.~Wang  in a series of articles \cite{LiTeWa1, LiTeWa2, LiTeWa3}. There 
the existence of a global, weak solution to the primitive equations is proven for initial data in $v_0\in L^2(\Dom)$ and $\T_0\in L^2(\Dom)$.
The uniqueness of these weak solutions remains an open problem until today, and  only under additional regularity assumptions in the vertical direction they are known to be unique (see e.g.\ \cite{Ju17}). 

A landmark result on the global strong well-posedness of the {\em deterministic} primitive equations subject to homogeneous Neumann conditions on top and bottom for initial data in $H^1(\Dom)$ was shown first by C.~Cao and E.S.~Titi in \cite{CT07}, and independently by R.M.~Kobelkov \cite{Ko07}, via $L^\infty(0,T;H^1(\Dom))$ \textit{a priori} energy estimates. For mixed Dirchlet-Neumann conditions see also \cite{Kukavica_2007}.    
A different approach to the deterministic  primitive equations based on  evolution equations has been introduced in \cite{HK16,GGHHK20}. This approach is based on the analysis of the 
hydrostatic Stokes operator and the corresponding hydrostatic Stokes semigroup. 
For a survey on results concerning the deterministic primitive equations using  energy estimates, we refer to \cite{Li2016} and for a survey concerning the approach based on evolution equations to \cite{HH20_fluids_pressure}.

\emph{Stochastic} versions of the primitive equations have been studied by 
several authors. 
A global well-posedness result for pathwise strong solutions 
is established for multiplicative white noise in time by A.~Debussche, N.~Glatt-Holtz and R.~Temam in \cite{DEBUSSCHE20111123} and the same authors with M.~Ziane in \cite{Debussche_2012}.  There, a  Galerkin approach is used to first show the existence of 
martingale solutions, and then a pathwise uniqueness result is deduced which leads together with a Yamada-Watanabe type result to the 
existence of a local pathwise solutions. The global existences of solutions is then shown by energy estimates where the noise is seen as a perturbation of the linear system. There, one of the difficulties is the handling of the pressure when proving $L^p$-estimates for $p>2$. This is overcome by considering the corresponding Stokes problem with the noise term and then proving estimates for the difference of the  solution of the full non-linear problem and the solution of that Stokes problem. This difference solves a random partial differential equation where analytic tools can be used to estimate the pressure term. A disadvantage of this approach is that it requires the solution of the Stokes problem to be rather smooth and transport noise cannot be included for this reason.

In the recent work  \cite{BS21} by Z.~Brze\'{z}niak and J.~Slav\'{i}k a similar approach is used for the local existence, but instead of considering the Stokes problem, they impose conditions on the noise such that it does not act directly on the pressure when turning to the question of global existence. Hence, by using a hydrostatic version of the Helmholtz projection, they can apply deterministic estimates to the pressure. Transport noise acting on the full velocity field is therefore not included, only the vertical average of $v$ can be transported by the noise. 
By our approach we can overcome both drawbacks at once, we can handle full transport noise acting directly on the pressure.

Let us mention some further results on the stochastic primitive equations.
For additive noise there is a transformation such that the probabilistic 
dependence turns into a parameter for a deterministic system. For this case also the existence 
of a random pull-back attractor is known (see e.g.\ \cite{GH09}). Logarithmic moment bounds in $H^2(\Dom)$ are obtained 
in \cite{GHKVZ14} and used to prove the existence of ergodic invariant measures  
in $H^1(\Dom)$. A construction of weak-martingale solutions, that means martingale 
solutions the regularity of which in space and time is the one of a weak solution, by an implicit Euler scheme is given in 
\cite{GHTW17}. Large deviation principles are 
known for small multiplicative noise (see e.g.\ \cite{DONG2017}) and small times (see e.g.\ \cite{DongZhang18}), for 
an extension to  transport noise and moderate 
deviation principles see \cite{Slavik21}. The existence 
of a Markov selection is proven in \cite{DZ17} for additive noise.
For results in  two dimensions we refer to \cite{GHT11} and the references therein.

Aiming for noise as rough as possible we will first consider a strong-weak setting when investigating the system \eqref{eq:primitive_full}, meaning that the equations for $v$ hold in the strong PDE sense and the one for $\T$ in a weak sense, for the precise definitions of strong-weak solutions see Definition \ref{def:sol_strong_weak}.  Probabilistically, we are concerned with strong solutions. The reason for not considering both equations in the weak sense is that already in the deterministic case the uniqueness issue  for the weak velocity equation is unsolved. We also investigate the strong-strong setting, see Definition \ref{def:sol_strong_strong} for this notion of solution, since this setting is the one for which most deterministic results have been proven.
The main result is the global existence of solutions, see Theorem \ref{t:global_primitive} for the strong-weak setting  
and Theorem \ref{t:global_primitive_strong_strong} for the strong-strong setting. 
For the readers convenience we state here a simplified version. 
We write $\phi^j:=(\phi^j_n)_{n\geq 1}$,  $\psi^j:=(\psi^j_n)_{n\geq 1}$,  $\hp^{\ell,m}:=(\hp_{n}^{\ell,m})_{n\geq 1}$ and $\R_+=(0,\infty)$.

\begin{maintheorem}[Simplified version]
Let 
$\kone$ be constant, $\gvn^k=\gtn=0$,  $\ft=0$, and $\fv= k_0(v^2,-v^1)$ for $k_0\in\R$ be the Coriolis force.
For all $n\geq 1$ let the maps 
\begin{align*}
\phi_n,\psi_n\colon \R_+\times \O\times \Dom\to \R^3 \quad \hbox{and} \quad \hp_n\colon\R_+\times \O\times \Tor^2\to \R^{2\times 2}
\end{align*}
 be $\Progress\otimes \Borel$-measurable, and let for some $\delta>0$ and all $j\in \{1,2,3\}$, $\ell,m\in \{1,2\}$ be
\begin{align*}
\phi^j\in L^\infty(\R_+\times \O;H^{1,3+\delta}(\Dom;\ell^2)),\ \ \ \
  \psi^j\in L^\infty(\R_+\times \O\times \Dom;\ell^2),\\
\text{ and }\  \  \  \ 
\hp^{\ell,m}\in L^{\infty}(\R_+\times \O;L^{3+\delta}(\Tor^2;\ell^2)),&
\end{align*}	
where $\phi^1_n$ and $\phi^2_n$ are assumed to be independent of $x_3$.
Furthermore, assume that there exists $\ellip\in (0,2)$ such that, almost surely (a.s.) for all $t\in \R_+$, $x\in \Dom$ and $\xi\in \R^3$ the parabolicity conditions
		\begin{align*}
			\sum_{n\geq 1} \Big(\sum_{j=1}^3 \phi^j_n(t,x) \xi_j\Big)^2\leq \ellip |\xi|^2
			\ \ \text{ and }\ \ 
			\sum_{n\geq 1} \Big(\sum_{j=1}^3 \psi^j_n(t,x) \xi_j\Big)^2 \leq \ellip |\xi|^2
		\end{align*}
hold. Then for each
	$$
	v_0\in L^0_{\F_0}(\O;\Hs^1(\Dom)) \quad  \text{and}\quad \T_0\in L^0_{\F_0}(\O;L^2(\Dom))
	$$
	there exists a \emph{unique global} strong-weak solution $(v,\T)$ to \eqref{eq:primitive_full}-\eqref{eq:boundary_conditions_w},  in particular
	$$
	(v,\T)\in L^2_{\loc}([0,\infty);\Hs_{\n}^2(\Dom)\times H^1(\Dom))\cap C([0,\infty);\Hs^1(\Dom)\times L^2(\Dom)) 
	\text{ a.s.\ }
	$$
\end{maintheorem}
For the definition of  $\Progress\otimes \Borel$-measurable, $ L^0_{\F_0}(\O;X)$, and the notation for the function spaces see Section~\ref{sec:pre}. 
In the above, we have not specified the unknowns $w$, $P$ and $\wt{P}_n$  as they are uniquely determined by $v$ and $\T$ due to the divergence free condition and the hydrostatic Helmholtz projection. Moreover,
replacing the regularity assumption on $\psi^j$ by $\psi^j\in L^\infty(\R_+\times \O;H^{1,3+\delta}(\Dom;\ell^2))$ and considering $\T_0\in L^0_{\F_0}(\O;H^1(\Dom))$ we obtain the analogous result in the strong-strong setting.

Let us compare our result with the above mentioned Kraichanan's turbulence theory. 
There one usually assumes that, for some (typically small) $\gamma>0$,
\begin{equation}
\label{eq:Kraichnan_assumption}
\phi^j_n \in H^{3/2+\gamma}(\Dom) \  \text{ for all }j\in \{1,2,3\} \text{ and }n\geq 1,
\end{equation} 
cf.\ e.g.\ \cite[Equation (1.3)]{MR05}. Since $H^{3/2+\gamma}(\Dom)\hookrightarrow H^{1,3+\delta}(\Dom)$ 
for some $\delta>0$ by Sobolev embedding, our noise is consistent with the regularity of the reproducing kernel Hilbert space of the Kraichanan noise. However, taking into account the summability in $n\geq 1$ required in our main results,  we can cover only the case of \emph{regular} Krainchan noise. We refer to \cite[Section 5]{GY21_stabilization} and the references therein for the terminology.
For the relevance of Kraichnan's noise in the context of geophysics we refer to \cite{Delsole04}.

To prove our main results we take another point of view on the problem than  in \cite{BS21} and \cite{Debussche_2012}. Here, we interpret the transport part of the noise as a part of the linearized system, and we only need to impose conditions guaranteeing that this linearization is parabolic. Compared to \cite{BS21,Debussche_2012}, this makes it possible to consider noise that transports the full velocity field, and moreover this leads even to weaker assumptions than in \cite{BS21} and \cite{Debussche_2012} in the setting where the noise is such that their results  apply. The only smallness condition in our result is the parabolicity condition, which is optimal in the sense that when dropping it the system is not parabolic any more and thus loosens its smoothing properties. This condition origins already in the local existence theory and is by far weaker than the smallness conditions in \cite{BS21,Debussche_2012}, where the noise is handled as a nonlinear perturbation of the deterministic system. Also, to deduce the global existence of solutions in our case, no additional smallness has to be assumed compared to the local existence.

The proof of the local existence in Theorem \ref{t:local_primitive} is based on the theory of critical spaces for stochastic evolution equations developed by the M.C.~Veraar and the first author  in \cite{AV19_QSEE_1,AV19_QSEE_2}. To apply these results,  
we need to study the \emph{stochastic maximal $L^2$-regularity} estimates for the linearized problem 
elaborated in Section \ref{s:estimates_linearized_strong_weak}. The global existence result
Theorem  \ref{t:global_primitive} is then obtained from the blow-up criteria of Theorem \ref{t:local_primitive}\eqref{it:blow_up_criterium_strong_weak} and suitable \emph{energy estimates} obtained in the spirit of C.~Cao and E.S.~Titi \cite{CT07}. Here we actually follow the approach taken by the  T.~Kashiwabara and the second author in \cite{HK16} (see also \cite{HH20_fluids_pressure}) where the  $L^6$-estimates proven in \cite{CT07} are replaced by the (apparently) weaker $L^4$-estimates. The deterministic estimates are proven by splitting the velocity field into its vertical average $\overline{v}$ and the remainder $\wt{ v}=v- \overline{v}$. A crucial observation is that the deterministic part of the turbulent pressure $\partial_{\hp}\wt{P}$ does not appear in the equations for $\wt{v}$ since it is $x_3$-independent in case that $\hp_{n}^{\ell,m}$ is also $x_3$-independent. Otherwise, the $L^4$-estimate for $\wt{v}$ could not be shown in this way since $\partial_{\hp} \wt{P}$ is a \emph{non-local} operator in $v$.

Our noise and the corresponding stochastic integrals are in It\^{o}-form, but in fluid mechanics, and in particular for geophysical flows, also the Stratonovich formulation is relevant, and it is seen as a more realistic model see e.g.\ \cite{BiFla20, HL84,Franzke14,W_thesis} and the references therein,  and in \cite{FlaPa21,DP22_two_scale} transport noise of Stratonovich type in fluid dynamical models has been rigorously justified from additive noise by multiscale arguments. Also the modelling in \cite{MR01} and \cite{MR04} is based on a Stratonovich type of noise to describe the turbulent part of the velocity field which is then translated into an It\^{o} formulation. To include such types of noise directly, we will consider the primitive equations with Stratonovich noise, see system \eqref{eq:primitive_Stratonovich}. In a preparatory step we first extend our result on the It\^{o} system \eqref{eq:primitive_full} to the case of non-homogeneous viscosity and conductivity in Theorems \ref{t:local_primitive_strong_weak_2} and \ref{t:global_primitive_strong_weak_2}. Based on a Stratonovich to It\^{o} transformation we then can use these results to infer the local and global existence of the primitive equations with Stratonovich type noise in Theorems \ref{t:local_primitive_Stratonovich_strong_weak} and \ref{t:global_primitive_strong_strong_2}, respectively. 

\subsection*{Overview}
This paper is organized as follows. In Section~\ref{sec:pre} the notation is fixed and a reformulated version of the stochastic primitive equations is given. 
In Section \ref{s:strong_weak} we give the precise notion of solution in the strong-weak setting and present the main result for this case. In Section \ref{s:estimates_linearized_strong_weak} we consider a linearized system for the turbulent hydrostatic Stokes system with temperature, and show that it admits stochastic maximal $L^2$-regularity. The proofs of the theorems from Section \ref{s:strong_weak} are carried out in Section \ref{s:proofs_strong_weak}. The strong-strong setting is investigated in Section \ref{s:strong_strong}. In Section \ref{s:variable_viscosity} we generalize our results to the case of non-homogeneous viscosity and conductivity. Finally, in Section \ref{s:Stratonovich} we show how our results imply also  the well-posedness of the primitive equations with Stratonovich noise.

\subsection*{Acknowledgements}
The authors thank the anonymous referees for their helpful comments and suggestions.

\section{Preliminaries}\label{sec:pre}
\subsection{Notation and deterministic function spaces}
\label{ss:set_up}
Here we collect the main notation which will be used through the paper. We often use universal constants $C$, and we write $\lesssim_C$ or just $\lesssim$ for $\leq C$.
For any integer $k\geq 1$, and $p\in(1,\infty)$, $L^p(\Dom;\R^k)=(L^p(\Dom))^k$ denotes the usual Lebesgue space and    $H^{s,p}(\Dom;\R^k)$ the corresponding Sobolev space for $s\in (0,\infty)$. In the paper we also use the common abbreviation $H^{s}(\Dom;\R^k):=H^{s,2}(\Dom;\R^k)$. 

Since $\Dom=\Tor^2\times (-h,0)$, we employ the natural splitting $x\mapsto (x_{\h},x_3)$ where $x_{\h}=(x_1,x_2)\in \Tor^2$, $x_3\in (-h,0)$ and the subscript $\h$ stands for \emph{horizontal}. Similarly, we define 
$$
\div_{\h}:=\partial_{1} +\partial_{2} , \qquad \nabla_{\h}:=(\partial_{1},\partial_{2}), \qquad
\Delta_{\h}:=\div_{\h}\nabla_{\h}.
$$
We also use the standard notations
\begin{align*}
	(v\cdot\nabla_{\h}) v &:=\Big(\sum_{j=1}^2  v^j \partial_j v^k\Big)_{k=1}^2,  & 
	(\phi_{n}\cdot\nabla) v&:=\Big(\sum_{j=1}^3 \phi_n^j \partial_j v^k\Big)_{k=1}^2,\\
	(v\cdot\nabla_{\h}) \T &:=\sum_{j=1}^2 v^j \partial_j \T , &
	(\psi_{n}\cdot\nabla) \T&:=\sum_{j=1}^3 \psi_n^j \partial_j \T.
\end{align*}

Next we introduce the function spaces 
for the velocity field. As a first step we introduce the two-dimensional \emph{Helmholtz projection} 
denoted by $\pr$ acting on the horizontal variables $x_{\h}\in \Tor^2$. Let $f\in L^2(\Dom;\R^2)$ and 
set $\qr f:=\nabla_{\h} \Psi_f\in L^2(\Tor^2;\R^2)$ where $\Psi_f\in H^{1}(\Tor^2)$ is the unique solution to the 
problem
\begin{align*}
\Delta_{\h} \Psi_f=\div_{\h} f \quad \text{ on }\Tor^2 \quad \hbox{with} \quad \int_{\Tor^2} \Psi_f \,dx=0.
\end{align*}
Then the \emph{Helmholtz projection} is given by
$$
\pr f:= f-\qr f  \ \ \  \text{ for } \ \ \ f\in L^2(\Tor^2;\R^2).
$$
The \emph{hydrostatic Helmholtz projection} $\p:L^2(\Dom;\R^2)\to L^2(\Dom;\R^2)$ is defined as
\begin{equation}
\label{eq:Helmholtz_hydrostatic}
\p f:=f -\qr \Big[\frac{1}{h}\int_{-h}^0 f (\cdot,\zeta)\, d\zeta \Big]\quad \text{for all  }f\in L^2(\Dom;\R^2),
\end{equation}
and its complementary projection is given by 
$\q f:=\qr \Big[\frac{1}{h}\int_{-h}^0 f (\cdot,\zeta)\, d\zeta \Big]$.
One can check that $\p$ is an orthonormal projection on $L^2(\Dom;\R^2)$, and by construction, $\div_{\h} \int_{-h}^0 (\p f(\cdot,z))dz=0$ 
holds in the distributional sense for all $f\in L^2(\Dom;\R^2)$. 
Let 
\begin{align*}
\Ls^2(\Dom):=\Big\{f\in L^2(\Dom;\R^2) \,:\, \div_{\h} \Big(\int_{-h}^0 f(\cdot,z)dz\Big)=0\text{ on }\Tor^2 \Big\},
\end{align*}
be endowed with the norm $\|f\|_{\Ls^2(\Dom)}:=\|f\|_{L^2(\Dom;\R^2)}$ 
and for all $k\geq 1$ we set
$$
\Hs^{k}(\Dom):=H^{k}(\Dom;\R^2)\cap \Ls^{2}(\Dom), \qquad \|f\|_{\Hs^k(\Dom)}:=\|f\|_{H^k(\Dom;\R^2)}.
$$
If no confusion seems likely, we write simply $L^2$, $H^{k}$, $\Hs^{k}$, $ L^2(\ell^2)$ and $ H^k(\ell^2)$ instead of $L^2(\Dom;\R^m)$, $H^{k}(\Dom;\R^m)$, $\Hs^{k}(\Dom)$, $L^2(\Dom;\ell^2(\N;\R^m))$ and $ H^k(\Dom;\ell^2(\N;\R^m))$, 
where, we use the short hand notation $\ell^2$ for $\ell^2(\N;\R^m)$ or $l^2(\N)$. The dual space of $H^1(\Dom)$ is denoted by $(H^1(\Dom))^*$.

\subsection{Probabilistic notation and function spaces}\label{subsec:pre_prop}
Here we collect the main probabilistic notation. Throughout the paper we fix a filtered probability space 
\begin{align*}
(\O,\A,(\F_t)_{t\geq 0},\P) \quad \hbox{and we set} \quad \E[\cdot]:=\int_{\O}\cdot \, d\P.
\end{align*}
Moreover, $(\beta^n)_{n\geq 1}=(\beta^n_t\,:\,t\geq 0)_{n\geq 1}$ denotes a sequence of  standard independent Brownian motions on the above mentioned probability space. We will denote by $\Br_{\ell^2}$ the $\ell^2$-cylindrical Brownian motion uniquely induced by $(\beta_n)_{n\geq 1}$ 
via 
\begin{align}\label{eq:Bl2}
\Br_{\ell^2}(f):=\sum_{n\geq 1} \int_{\R_+} \langle f(t), e_n\rangle\, d\beta^n_t, \ \hbox{ with } \ e_n=(\delta_{jn})_{j\geq 1}, \,  f\in L^2(\R_+;\ell^2),
\end{align}
and Kronecker's $\delta_{jn}$, cf.\ e.g.\ \cite[Example 2.12]{AV19_QSEE_1}.
For a stopping time $\tau$, we set
$$
[0,\tau]\times \O:=\{(t,\om)\,:\, 0\leq \tau(\om)\leq t\}
$$
and use analogous definitions for $[0,\tau)\times \O$ etc. 

By $\Progress$ and $\Borel$ we denote the progressive and the Borel $\sigma$-algebra, respectively. Moreover, we say that a map $\Phi: \R_+\times \O\times \R^m\to \R$ is $\Progress\otimes \Borel$-measurable if $\Phi$ is $\Progress\otimes \Borel(\Dom)\otimes \Borel(\R^m)$-measurable, where $m\geq 1$ is an integer. By $L^0_{\F_0}(\O)$ we denote the space of $\F_0$-measureable functions and
by $L^2_{\Progress}$ the $L^2$-space with respect to the progressive $\sigma$-algebra. 

\subsection{Reformulation of the primitive equations}
\label{ss:reformulation}
As it is well-known the primitive equations can be formulated equivalently in terms of the unknown $v=(v^k)_{k=1}^2 :[0,\infty)\times \O\times \Dom\to \R^2$ which contains only the first two components of the unknown velocity field $u$. Indeed, the divergence-free condition and \eqref{eq:boundary_conditions_w} are equivalent to setting $w=w(v)$ where 
\begin{equation}
	\label{eq:def_w_v}
	\big(w(v)\big)(t,x):=-\int_{-h}^{x_3}\div_{\h} v(t,x_{\h},\zeta)\,d\zeta,
\end{equation}
a.s.\ 
for all $t\in \R_+$ and $x=(x_{\h},x_3)\in\Tor^2\times (-h,0) =\Dom$ and imposing
$$
\int_{-h}^{0}\div_{\h} v (t,x_{\h},\zeta)\,d\zeta=0,
$$
a.s.\ for all $t\in \R_+$ and $x_{\h}\in \Tor^2$. 
Moreover, we get by integrating the third equation in \eqref{eq:primitive_full}, a.s.\ for all $t\in \R_+$ and $x=(x_{\h},x_3)\in \Dom$,
\begin{align*}
	P(t,x)= p(t,x_{\h}) -\int_{-h}^{x_3} \kone(t,x_{\h},\zeta)\T(t,x_{\h},\zeta) d\zeta .
\end{align*}
Thus the pressure depends linearly on the temperature $\T$. In the physical literature $p$ is usually called the \emph{surface pressure}. Hence, \eqref{eq:primitive_full}-\eqref{eq:boundary_conditions_w} turns into
\begin{equation}
	\label{eq:primitive_intro_v}
	\begin{cases}
		\displaystyle{d v -\Delta v\, dt=\Big[  -(v\cdot \nabla_{\h})v- w(v)\partial_3 v -\nabla_{\h} p
		+\partial_{\hp} \wt{P}}\\
		\qquad \qquad\qquad \quad
		\displaystyle{+\nabla_{\h}\int_{-h}^{\cdot}  (\kone(\cdot,\zeta)\T(\cdot,\zeta))\,d\zeta+ \fv(\cdot,v,\T) \Big]dt }\\
		\qquad \qquad \ \ \ \ \ \ \    
		\displaystyle{+\sum_{n\geq 1}\Big[(\phi_{n}\cdot\nabla) v-\nabla_{\h} \wt{P}_n  +\gvn(\cdot,v,\T)\Big] d\beta_t^n}, 
		&\text{on }\Dom,\\
		\displaystyle{d \T -\Delta \T\, dt=\Big[ -(v\cdot \nabla_{\h})\T- w(v)\partial_3 \theta+ \ft(\cdot,v,\T ) \Big]dt} \\
		\qquad \qquad\qquad\qquad\quad\quad\
		\displaystyle{+\sum_{n\geq 1}\Big[(\psi_{n}\cdot\nabla) \T+\gtn(\cdot,v,\Temp)\Big] d\beta_t^n}, 
		&\text{on }\Dom,\\
		\partial_{3} p=\partial_{3} \wt{P}_n=0, &\text{on }\Dom,\\
		\displaystyle{\int_{-h}^{0}\div_{\h} v (\cdot,\zeta)\,d\zeta=0,}
		&\text{on }\Tor^2,
		\\
		v(\cdot,0)=v_0,\ \ \T(\cdot,0)=\T_0, &\text{on }\Dom,
	\end{cases}
\end{equation}
where $w(v)$ is given by \eqref{eq:def_w_v} and
complemented with the boundary conditions \eqref{eq:boundary_conditions}. \\

\section{Local and global existence in the strong-weak setting}
\label{s:strong_weak}
In this section we study 
the stochastic primitive equations in the \emph{strong-weak} setting, i.e.\ in case the equation for $v$ is understood in the \emph{strong} setting and the one for $\T$ in the \emph{weak} one. The latter means that equation for $\T$ will be formulated in its natural weak (analytic) form. In Section \ref{s:strong_strong}, we also consider the case where both equations are understood in the strong setting (referred here as the \emph{strong-strong} setting). 
Compared to the strong-strong setting, the choice made in this section has two basic advantage. Firstly, the energy estimates needed in our main global existence result are simpler, and secondly, we can allow a rougher noise in the equation for the temperature $\T$. 

We begin by reformulating the problem. Applying the hydrostatic Helmholtz projection $\p$ to the first equation in \eqref{eq:primitive_intro_v} it is, at least formally, equivalent to
\begin{equation}
\label{eq:primitive_weak_strong}
\begin{cases}
\displaystyle{d v -\Delta v\, dt=\p\Big[ -(v\cdot \nabla_{\h}) v- \w(v)\partial_3 v+\Lp (\cdot,v)} \\
\qquad\qquad \qquad \qquad\ \ \ \quad \displaystyle{
+\nabla_{\h}\int_{-h}^{x_3}  (\kone(\cdot,\zeta)\T(\cdot,\zeta))\,d\zeta + \fv (\cdot,v,\T,\nabla v) \Big]dt }\\ 
\ \ \qquad \qquad\qquad \qquad \qquad\ \ \ \qquad 
\displaystyle{+\sum_{n\geq 1}\p \Big[(\phi_{n}\cdot\nabla) v  +\gvn(\cdot,v)\Big] d\beta_t^n}, \\
\displaystyle{d \Temp -\Delta \Temp\, dt=\Big[ -(v\cdot\nabla_{\h}) \T -\w(v) \partial_3 \T+ \ft(\cdot,v,\T,\nabla v ) \Big]dt}
\\
\ \ \qquad \qquad\qquad \qquad \qquad\ \ 
\displaystyle{+\sum_{n\geq 1}\Big[(\psi_n\cdot \nabla) \Temp+\gtn(\cdot,v,\T,\nabla v)\Big] d\beta_t^n}, 
\\
v(\cdot,0)=v_0,\ \ \T(\cdot,0)=\T_0, 
\end{cases}
\end{equation}
on $\Dom=\Tor^2\times(-h,0)$ complemented with the following boundary conditions 
\begin{equation}
\label{eq:boundary_conditions_strong_weak}
\begin{aligned}
\partial_3 v (\cdot,-h)=\partial_3 v(\cdot,0)=0 \ \  \text{ on }\Tor^2,&\\
\partial_3 \T(\cdot,-h)= \partial_3 \T(\cdot,0)+\alpha \T(\cdot,0)=0\ \  \text{ on }\Tor^2.&
\end{aligned}
\end{equation}
Here $\alpha\in \R$ is given, $w(v)$ is as in \eqref{eq:def_w_v} and a.s.\ for all $t\in \R_+$,
\begin{align}
\label{eq:def_P_gamma}
\Lp (t,v)
&:= \Big(\sum_{n\geq 1} \sum_{m=1}^2 \hp_n^{\ell,m} (t,x)
\big(\q [(\phi_n\cdot\nabla) v + \gvn(\cdot,v)]\big)^m\Big)_{\ell=1}^2,
\end{align}
where $\big(\q [\cdot]\big)^m$ denotes the $m$-th component of the vector $\q [f]$. 
To see that $\Lp$ in \eqref{eq:def_P_gamma} coincides with $\partial_{\hp} \wt{P}$ in \eqref{eq:primitive_intro_v} it is enough to recall that, by the hydrostatic Helmholtz decomposition in \eqref{eq:Helmholtz_hydrostatic}, it follows that
$$
\nabla_{\h} \wt{P}_n=\q[(\phi_n\cdot\nabla) v + \gvn(\cdot,v)].
$$
Finally, let us note that in the stochastic part of the equation for the velocity field $v$, (in general) the operator $\p$ cannot be  removed since it may happen that $\div_{\h}\int_{-h}^0[(\phi_{n}\cdot\nabla) v  +\gvn(v)]\,d\zeta\neq 0$. For instance this is the case if $\phi_n$ is $x_3$-dependent and $\gvn\equiv 0$.

\subsection{Main assumptions and definitions}
\label{ss:assumptions_strong_weak}
We begin by listing the main assumptions which 
are in force in this section.

\begin{assumption} There exist $M,\delta>0$ for which the following hold.
\label{ass:well_posedness_primitive}
\begin{enumerate}[{\rm(1)}]
\item\label{it:well_posedness_measurability} For all $n\geq 1$ and $j\in \{1,2,3\}$, the maps $$\phi_n^j,\psi_n^j,\kone: \R_+\times \O\times \Dom\to \R$$ are $\Progress\otimes \Borel$-measurable;
\item\label{it:well_posedness_primitive_phi_smoothness}  a.s.\ for all $t\in \R_+$, $j,k\in \{1,2,3\}$ and $\ell,m\in \{1,2\}$,
\begin{align*}
\Big\|\Big(\sum_{n\geq 1}| \phi^j_n(t,\cdot)|^2\Big)^{1/2} \Big\|_{L^{3+\delta}(\Dom)}+
\Big\|\Big(\sum_{n\geq 1}|\partial_k \phi^j_n(t,\cdot)|^2\Big)^{1/2} \Big\|_{L^{3+\delta}(\Dom)} &\leq M,\\
\Big\|\Big(\sum_{n\geq 1}|\hp^{\ell,m}_n(t,\cdot)|^2\Big)^{1/2}\Big\|_{L^{3+\delta}(\Dom)}&\leq M;
\end{align*}
\item\label{it:well_posedness_primitive_L_infty_bound}  
a.s.\ for all $t\in \R_+$, $x\in \Dom$ and $j\in \{1,2,3\}$,
\begin{align*}
\Big(\sum_{n\geq 1} | \psi^j_n(t,x) |^2\Big)^{1/2} \leq M;
\end{align*}
\item\label{it:well_posedness_primitive_kone_smoothness}
a.s.\ for all $t\in \R_+$, $x_{\h}\in \Tor^2$, $j\in \{1,2,3\}$ and $i\in \{1,2\}$,
$$
\| \kone(t,x_{\h},\cdot) \|_{L^2(-h,0)} +\|\partial_i \kone(t,\cdot) \|_{L^{2+\delta}(\Tor^2;L^2(-h,0))} \leq M;
$$
\item\label{it:well_posedness_primitive_parabolicity} there exist $\ellip\in (0,2)$ such that, a.s.\ for all $t\in \R_+$, $x\in \Dom$ and $\xi\in \R^3$,
\begin{align*}
\sum_{n\geq 1} \Big(\sum_{j=1}^3 \phi^j_n(t,x) \xi_j\Big)^2\leq \ellip |\xi|^2,
\ \ \text{ and }\ \ 
\sum_{n\geq 1} \Big(\sum_{j=1}^3 \psi^j_n(t,x) \xi_j\Big)^2 \leq \ellip |\xi|^2;
\end{align*}
\item\label{it:nonlinearities_measurability}
for all $n\geq 1$, the maps 
\begin{align*}
&\fv\colon \R_+\times \O\times \R^2\times \R^6\times \R \to \R^2, \quad
\ft: \R_+\times \O\times \R^2\times \R^6\times \R \to \R,\\
&\gvn\colon\R_+\times \O\times \R\to \R^2, \quad\hbox{and} \quad
\gtn\colon\R_+\times \O\times \R^2\times \R^6\times \R\to \R
\end{align*}
are $\Progress\otimes \Borel$-measurable;

\item\label{it:nonlinearities_strong_weak} 
for all $T\in (0,\infty)$ and $i\in \{1,2\}$, 
\begin{align*}
\fv^i (\cdot,0),\ft(\cdot,0)&\in L^2((0,T)\times\O\times \Dom), \\
(\gvn^i(\cdot,0))_{n\geq 1}&\in  L^2((0,T)\times \O;H^1(\Dom; \ell^2)) \hbox{ and } \\
 (\gtn(\cdot,0))_{n\geq 1}&\in L^2((0,T)\times \O\times \Dom; \ell^2).
\end{align*}
%
Moreover, for all $n\geq 1$, $t\in \R_+$,  $x\in \Dom$, $y,y'\in \R^2$, $Y,Y'\in \R^6$ and $z,z'\in \R$,
\begin{align*}
&|\fv(t,x,y,z,Y)-\fv(t,x,y',z',Y')|+
|\ft(t,x,y,z,Y)-\ft(t,x,y',z',Y')|\\
&
+
\|(\gtn(t,x,y,z,Y)-\gtn(t,x,y',z',Y'))_{n\geq 1}\|_{\ell^2}\\
&\qquad\qquad
\lesssim (1+|y|^4+ |y'|^4)|y-y'|+
(1+|z|^{2/3}+|z'|^{2/3})|z-z'|\\
&\qquad\qquad
+(1+|Y|^{2/3}+|Y'|^{2/3})|Y-Y'|.
\end{align*}
Finally, a.s.\ for all $t\in\R_+$, $\Dom\times \R^2\ni (x,y)\mapsto \gvn(t,x,y)$ is continuously differentiable and for all $k\in\{0,1\}$, $j\in \{1,2,3\}$, $x\in \Dom$, and $y,y'\in \Dom$ a.s.
\begin{align*}
&\|(\partial_{x_j}^{k}\gvn(t,x,y)-\partial_{x_j}^{k}\gvn(t,x,y'))_{n\geq 1}\|_{\ell^2}
\lesssim (1+|y|^4+ |y'|^4)|y-y'|,\\
&\|(\partial_{y}\gvn(t,x,y)-\partial_{y}\gvn(t,x,y'))_{n\geq 1}\|_{\ell^2}\lesssim (1+|y|^2+|y'|^2)|y-y'|.
\end{align*}
\end{enumerate}
\end{assumption}

\begin{remark}\
\label{r:assump_local_existence_strong_weak}
\begin{enumerate}[{\rm(a)}]
\item\label{it:Holder_continuity_phi}
 In \eqref{it:well_posedness_primitive_phi_smoothness} the derivatives are taken in the distributional sense. 
The Sobolev embedding $H^{1,3+\delta}(\Dom;\ell^2)\embed C^{\alpha}(\Dom;\ell^2)$ where 
$\alpha=\frac{\delta}{3+\delta} \in (0,1)$ and \eqref{it:well_posedness_primitive_phi_smoothness} yield
\begin{equation*}
\|(\phi^j_n(t,\cdot))_{n\geq 1}\|_{C^{\alpha}(\Dom;\ell^2)} \lesssim_{\delta} M, \quad
\text{ a.s.\ for all }t\in \R_+.
\end{equation*}
\item Since $H^{3/2+\gamma}(\Dom)\embed H^{1,3+\delta}(\Dom)$ where $\delta=\frac{3\gamma}{1-\gamma}>0$ for all $\gamma\in (0,1)$, \eqref{it:well_posedness_primitive_phi_smoothness} fits the scaling of the Kraichnan's noise discussed in the introduction, cf.\ \eqref{eq:Kraichnan_assumption}.
\item\label{it:remark_parabolicity}
\eqref{it:well_posedness_primitive_parabolicity} is equivalent to the \emph{stochastic parabolicity}:  for all $t\in \R_+$, $x\in \Dom$, $\xi\in \R^3$ and a.s.
\begin{equation*}
 |\xi|^2-\frac{1}{2}\sum_{i,j=1}^3\sum_{n\geq 1}\phi_n^j(t,x) \phi_n^i(t,x) \xi_i\xi_j\geq \Big(1-\frac{\ellip}{2}\Big) |\xi|^2.
\end{equation*}
A similar reformulation holds for the condition on $\psi$. In particular, \eqref{it:well_posedness_primitive_parabolicity} is optimal in the parabolic setting.
\item \eqref{it:nonlinearities_strong_weak} contains the optimal growth assumptions on the nonlinearities which ensure existence and uniqueness of (local) solutions for data $(v_0,\T_0)\in \Hs^1(\Dom)\times L^2(\Dom)$, cf.\ the proof of Theorem \ref{t:local_primitive} in Subsection~\ref{ss:proof_local_strong_weak}.
\end{enumerate}
\end{remark}

To formulate \eqref{eq:primitive_weak_strong}-\eqref{eq:boundary_conditions_strong_weak} in the strong-weak setting, we regard the equation for $\T$ in its natural weak analytic formulation on the dual space $(H^{1}(\Dom))^*$.  To this end, the basic observation is given by the following formal integration by parts
\begin{equation}
\label{eq:formal_integration_by_parts}
\int_{\Dom}\Big( (v\cdot\nabla_{\h} v) \T +w(v)\partial_3 \T \Big) \varphi\,dx= -\int_{\Dom} 
\Big( v \T  \cdot \nabla_{\h} \varphi+ w(v)\T\cdot \partial_3 \varphi \Big)\,dx
\end{equation}
for all $\varphi\in H^1(\Dom)$. Note that the volume and boundary integrals disappear since  $\div_{\h} v+ \partial_3 w=0$ and  $w(\cdot,0)=w(\cdot,-h)=0$ on $\Tor^2$, respectively. The right hand side in \eqref{eq:formal_integration_by_parts} naturally defines an element in $(H^{1}(\Dom))^*$ by setting
\begin{equation}
\label{eq:T_map_definition}
H^1(\Dom)\ni \varphi \mapsto -\int_{\Dom} 
\Big( \T v  \cdot \nabla_{\h} \varphi+ \T w(v)\partial_3 \varphi \Big)\,dx=: \mathcal{T}_{\T}(\varphi).
\end{equation}
Below, we will use the more suggestive notation $\div_{\h}(v\T)+\partial_3 (w(v)\T)=\mathcal{T}_{\T}$. To complete the reformulation of the equation for $\T$ it remains to replace the Laplace operator $\Delta$ by its weak formulation $\Dr$ in case of Robin boundary conditions, i.e.
\begin{equation}
\label{eq:weak_Laplace_operator_robin}
\begin{aligned}
\Dr&: H^{1}(\Dom)\subseteq \big( H^{1}(\Dom))^* \to \big( H^{1}(\Dom))^*, \\
\l \varphi, \Dr \T\r& := 
-\int_{\Dom}\nabla \varphi\cdot \nabla \T\,dx - \alpha\int_{\Tor^2} \varphi(\cdot,0) \T(\cdot,0) \,dx_{\h},
\end{aligned}
\end{equation}
where $\l \cdot,\cdot\r$ denotes the duality pairing for  $H^1(\Dom)$ and $(H^1(\Dom))^*$. Note that the above definition is consistent with a formal integration by parts using the Robin boundary conditions for $\T$ in \eqref{eq:boundary_conditions_strong_weak}. Since the trace operator $f \mapsto f|_{\Tor^2 \times \{0\}}$ is bounded on $H^1(\Dom)$ with values in $L^2(\Tor^2)$, the previous definition in \eqref{eq:weak_Laplace_operator_robin}
 makes sense.

With these preparations, we are now in the position to define solutions to \eqref{eq:primitive_weak_strong}-\eqref{eq:boundary_conditions_strong_weak} in the strong-weak setting.  
 For notational convenience, we set
\begin{equation}
\label{eq:def_H_2_N_strong_weak}
\Hs_{\n}^2(\Dom):=\big\{v \in \Hs^{2}(\Dom)\,:\, \partial_3 v(\cdot,-h)=\partial_3v(\cdot,0)=0 \text{ on }\Tor^2\big\}.
\end{equation}
Recall that the embedding 
$L^2(\Dom)\embed (H^1(\Dom))^*$ is given by $\l \varphi,f\r:= \int_{\Dom} f \varphi \,dx$ where $\varphi\in H^1(\Dom)$ and $\Br_{\ell^2}$ is the $\ell^2$-cylindrical Brownian motion induced by the sequence $(\beta^n)_{n\geq 1}$, compare equation \eqref{eq:Bl2}.

\begin{definition}[$L^2$-strong-weak solutions] 
\label{def:sol_strong_weak}
Let Assumption \ref{ass:well_posedness_primitive} be satisfied.
\begin{enumerate}[{\rm(1)}]
\item Let $\tau$ be a stopping time, $v:[0,\tau)\times \O \to \Hs_{\n}^2(\Dom)$ and $\T:[0,\tau)\times \O \to H^1(\Dom)$
be stochastic processes. We say that $((v,\T),\tau)$ is 
an \emph{$L^2$-local strong-weak solution} to \eqref{eq:primitive_weak_strong}-\eqref{eq:boundary_conditions_strong_weak} if there exists a sequence of stopping times $(\tau_k)_{k\geq 1}$ for which the following hold:
\begin{itemize}
\item $\tau_k\leq \tau$ a.s.\ for all $k\geq 1$ and $\lim_{k\to \infty}\tau_k=\tau$ a.s.;
\item a.s.\ we have $(v,\T)\in L^2(0,\tau_k;\Hs_{\n}^2(\Dom)\times H^1(\Dom))$ and 
\begin{equation}
\begin{aligned}
\label{eq:integrability_strong_weak}
(v\cdot \nabla_{\h}) v+ \w(v)\partial_3 v +\fv (v,\T,\nabla v)+\Lp(\cdot,v)&\in L^2(0,\tau_k;L^2(\Dom;\R^2)),\\
-\div_{\h}(v \T) -\partial_{3}(\w(v) \T)&\in L^2(0,\tau_k;(H^1(\Dom))^*),\\
 \ft(v,\T,\nabla v )\ &\in L^2(0,\tau_k;L^2(\Dom)),\\
(\gvn(v))_{n\geq 1 }&\in L^2(0,\tau_k;H^1(\Dom;\ell^2(\N;\R^2))),\\
(\gtn(v,\T,\nabla v))_{n\geq 1 }&\in L^2(0,\tau_k;L^2(\Dom;\ell^2));
\end{aligned}
\end{equation}
\item a.s.\ for all $k\geq 1$ the following equality holds for all $t\in [0,\tau_k]$:
\begin{align*} 
v(t)-v_0
&=\int_0^t\Big( \Delta v(s)+ \p\Big[\nabla_{\h}\int_{-h}^{x_3} (\kone(\cdot,\zeta)\T(\cdot,\zeta))\,d\zeta\\
&\qquad
- (v\cdot \nabla_{\h}) v- \w(v)\partial_3 v + \fv (v,\T,\nabla v) +\Lp(\cdot,v)\Big]\Big)\,ds\\
&\qquad
+\int_0^t \Big(\one_{[0,\tau_k]}\p[ (\phi_{n}\cdot\nabla) v   +\gvn(v)] \Big)_{n\geq 1}\, d\Br_{\ell^2}(s),\\
\T(t)-\T_0
&=
\int_0^t  \Big[\Dr \T-\div_{\h}(v \T) -\partial_3 (\w(v) \T)+ \ft(v,\T,\nabla v )\Big]\,ds\\
&+
\int_0^t \Big(\one_{[0,\tau_k]}[ (\psi_{n}\cdot\nabla) \T   +\gtn(v,\T,\nabla v)] \Big)_{n\geq 1}\, d\Br_{\ell^2}(s);
\end{align*}
\end{itemize}
\item an $L^2$-local strong-weak solution $((v,\T),\tau)$ to \eqref{eq:primitive_weak_strong}-\eqref{eq:boundary_conditions_strong_weak} is said to be an \emph{$L^2$-maximal strong-weak solution to} \eqref{eq:primitive_weak_strong}-\eqref{eq:boundary_conditions_strong_weak}
if for any other local solution $((v',\T'),\tau')$ we have 
\begin{equation*}
\tau'\leq \tau \  \text{ a.s.\ \  and } \ \ (v,\T)=(v',\T') \ \text{ a.e.\ on }[0,\tau')\times \O.
\end{equation*}
\end{enumerate}
\end{definition}

Note that $L^2$-maximal strong-weak solution are \emph{unique} in the class of $L^2$-local strong-weak solutions by definition. By \eqref{eq:integrability_strong_weak}, the deterministic integrals and the 
stochastic integrals in the above definition are well-defined as Bochner and $L^2$-valued It\^{o} integrals, respectively.

\subsection{Statement of the main results}
\label{ss:main_results_strong_weak}
We begin this subsection by stating a local existence result for \eqref{eq:primitive_weak_strong}-\eqref{eq:boundary_conditions_strong_weak}. 
To economize the notation,
for $k\geq 0$, $m\geq 1$, $f:[0,t)\to H^{k+1}(\Dom;\R^m)$, we set
\begin{equation}
\label{eq:def_norm}
\norm_k(t;f):= \sup_{s\in [0,t)}\|f(s)\|_{H^{k}(\Dom;\R^m)}^2+\int_0^t \|f(s)\|_{H^{k+1}(\Dom;\R^m)}^2\,ds .
\end{equation}

\begin{theorem}[Local existence]
\label{t:local_primitive}
Let Assumption \ref{ass:well_posedness_primitive} be satisfied. Then for each 
$$
v_0\in L^0_{\F_0}(\O;\Hs^1(\Dom)), \ \ \text{ and }\ \ \T_0\in L^0_{\F_0}(\O;L^2(\Dom)),
$$
there exists an \emph{$L^2$-maximal strong-weak solution} $((v,\theta) ,\tau)$ to \eqref{eq:primitive_weak_strong}-\eqref{eq:boundary_conditions_strong_weak} where $\tau>0$ a.s. Moreover, we have
\begin{enumerate}[{\rm(1)}]
\item\label{it:path_regularity_strong_weak}{\em (Pathwise regularity)} 
 there exists a sequence of stopping times $(\tau_k)_{k\geq 1}$ such that a.s.\  for all $k\geq 1$ one has $0\leq \tau_k\leq \tau$ , $\lim_{k\to \infty}\tau_k=\tau$ and 
 \begin{align*}
(v,\T)\in L^2(0,\tau_k;\Hs_{\n}^2(\Dom)\times H^1(\Dom))\cap C([0,\tau_k];\Hs^1(\Dom)\times L^2(\Dom));
\end{align*}  
\item\label{it:blow_up_criterium_strong_weak}{\em (Blow-up criterion)} for all $T\in (0,\infty)$
$$
\P\Big(\tau<T,\,  \norm_{1}(\tau;v)+\norm_{0}(\tau;\T) <\infty\Big)=0.
$$
\end{enumerate}
\end{theorem}


Next we state our main result of this section concerning \emph{global existence} of solutions to \eqref{eq:primitive_weak_strong}-\eqref{eq:boundary_conditions_strong_weak}. To this end, we also need the following assumptions.

\begin{assumption} Let Assumption \ref{ass:well_posedness_primitive} be satisfied and assume the following
\label{ass:global_primitive}
\begin{enumerate}[{\rm(1)}]
\item\label{it:independence_z_variable} For all $n\geq 1$, $x=(x_{\h},x_3)\in \Tor^2\times (-h,0)=\Dom$, $t\in \R_+$, $j,k\in \{1,2\}$ and a.s.
\begin{align*}
\phi_n^j(t,x)\text{ and }\hp^{j,k}_n(t,x) \text{ are independent of $x_3$};
\end{align*}
%
\item\label{it:sublinearity_Gforce}
there exist $C>0$ and $\y\in L^0_{\Progress}(\O;L^2_{\loc}([0,\infty);L^2(\Dom)))$ such that, a.s.\ for all $t\in \R_+$, $j\in \{1,2,3\}$, $x\in \Dom$, $y\in \R^2$, $z\in \R$ and $Y\in \R^{6}$, 
\begin{align*}
|\fv(t,x,y,z,Y)|
&\leq C(\y(t,x)+|y|+|z|+|Y|),\\
|\ft(t,x,y,z,Y)|
&\leq C(\y(t,x)+|y|+|z|+|Y|),\\
\|(\gvn (t,x,y))_{n\geq 1}\|_{\ell^2}
+\|(\partial_{x_j}\gvn (t,x,y))_{n\geq 1}\|_{\ell^2}&\leq C(\y(t,x) + |y|), \\
\|(\partial_{y}\gvn(t,x,y))_{n\geq 1}\|_{\ell^2}
&\leq C, \\
\|(\gtn(t,x,y,z,Y))_{n\geq 1}\|_{\ell^2}
&\leq C(\y(t,x)+|y|+|z|+|Y|).
\end{align*}
\end{enumerate}
\end{assumption}


\begin{theorem}[Global existence]
\label{t:global_primitive}
Let Assumption 
\ref{ass:global_primitive} 
be satisfied, and let  
$$
v_0\in L^0_{\F_0}(\O;\Hs^1(\Dom)), \ \ \text{ and }\ \ \T_0\in L^0_{\F_0}(\O;L^2(\Dom)).
$$
Then the $L^2$-maximal strong-weak solution $((v,\T),\tau)$ to \eqref{eq:primitive_weak_strong}-\eqref{eq:boundary_conditions_strong_weak} provided by Theorem \ref{t:local_primitive} is \emph{global in time}, i.e.\ $\tau=\infty$ a.s. In particular
$$
(v,\T)\in L^2_{\loc}([0,\infty);\Hs_{\n}^2(\Dom)\times H^1(\Dom))\cap C([0,\infty);\Hs^1(\Dom)\times L^2(\Dom)) 
\text{ a.s.\ }
$$
\end{theorem}

The proof of Theorems \ref{t:local_primitive} and \ref{t:global_primitive} will be given in Subsections \ref{ss:proof_local_strong_weak} and \ref{ss:proof_global_strong_weak}, respectively.
As a key tool in the proofs, we need suitable estimates for the linearized problem of \eqref{eq:primitive_weak_strong}-\eqref{eq:boundary_conditions_strong_weak} which will be proven in Section \ref{s:estimates_linearized_strong_weak}.
Before giving the proofs, below we collect some comments on Assumption \ref{ass:global_primitive}.

\begin{remark}\label{r:global_primitive} \
\begin{enumerate}[{\rm(a)}]
\item Assumption \ref{ass:global_primitive}\eqref{it:independence_z_variable} contains additional assumptions on $\phi_n^1,\phi^2_n$ but \emph{not} on $\phi^3_n$ as compared to Assumption~\ref{ass:well_posedness_primitive}. Roughly speaking, this means that the transport noise can be very rough in the vertical direction while the horizontal part is two-dimensional.
\item \label{it:ellipticity_overline_v} Taking $\xi=(\xi_1,\xi_2,0)$ in Assumption~\ref{ass:well_posedness_primitive}\eqref{it:well_posedness_primitive_parabolicity} we also have that there exists $\ellip\in (0,2)$ such that, a.s. for all $x\in \Dom$, $t\in \R_+$ and $\xi\in \R^2$,
\begin{equation*}
\sum_{n\geq 1} \Big(\sum_{j=1}^2 \phi^j_n(t,x) \xi_j\Big)^2\leq \ellip|\xi|^2.
\end{equation*}
This implies that we also have parabolicity for the subsystem \eqref{eq:primitive_bar} below obtained from the first equation in \eqref{eq:primitive_weak_strong} after averaging in the $x_3$-variable.
\item\label{it:global_primitive_periodic}
(\emph{Periodic boundary conditions}).
The results of Theorems \ref{t:local_primitive} and \ref{t:global_primitive} also hold in case the boundary conditions \eqref{eq:boundary_conditions_strong_weak} are replaced by the periodic ones. To see this it is enough to ignore the boundary terms appearing in the integration by part arguments in the proofs below. The same applies to the results of Sections \ref{s:strong_strong}-\ref{s:Stratonovich}. For brevity, we do not repeat this observation in the following.

\end{enumerate}
\end{remark}

\section{$L^2$--estimates for the linearized problem}
\label{s:estimates_linearized_strong_weak}
\subsection{$L^2$-stochastic maximal regularity}
In this section we deduces an $L^2$--estimate for the linear part of the problem \eqref{eq:primitive_weak_strong}, which is central to our approach. Here we consider the following \emph{turbulent hydrostatic Stokes system with temperature}
\begin{equation}
\label{eq:SMR_v_T_strong_weak}
\begin{cases}
\vspace{0.1cm}
\displaystyle{dv -\Big[\Delta v+\p [\Lpp v +\op_{\kone} \T]\Big]dt =f_v dt 
+ \sum_{n\geq 1} \Big[\p [(\phi_n\cdot \nabla )v]+ g_{n,v}\Big] d\beta_t^n} &\text{on }\Dom,\\
\displaystyle{d\Temp -\Dr \Temp dt =f_{\T} dt + \sum_{n\geq 1} 
\Big[(\psi_n\cdot \nabla )\Temp+ g_{n,\T}\Big]d\beta_t^n}
&\text{on }\Dom,\\
\partial_3 v(\cdot,0)=\partial_3 v(\cdot,-h)=0 &\text{on }\Tor^2,\\
v(0)=0,\qquad \Temp(0)=0& \text{on }\Dom,
\end{cases}
\end{equation}
where 
$\Dr$ is the weak Laplacian with Robin boundary conditions (see \eqref{eq:weak_Laplace_operator_robin}), and for all $t\in \R_+$, $x=(x_{\h},x_3)\in \Dom$ and $\theta\in H^1(\Dom)$,
\begin{align}
\label{eq:Lpp_operator}
(\Lpp v)(t,x)&:=
\Big(\sum_{n\geq 1} \sum_{m=1}^2 \hp_n^{\ell,m} (t,x)
\big(\q[(\phi_n\cdot\nabla) v]\big)^m\Big)_{\ell=1}^2,\\
\label{eq:integral_operator_temperature}
(\op_{\kone} \T)(t,x) &:= \nabla_{\h} \int_{-h}^{x_3}   ( \kone(t,x_{\h},\zeta)\T(x_{\h},\zeta))d\zeta. 
\end{align}
Let $\tau\colon \Omega \rightarrow [0,T]$ be a stopping time and let 
\begin{equation}
\label{eq:f_v_etc_integrability_smr}
\begin{aligned}
(f_v,f_{\T})&\in L^2_{\Progress}((0,T)\times \Omega;\Ls^2\times (H^1)^*),\\
((g_{n,v})_{n\geq 1}, (g_{n,\T})_{n\geq 1})&\in 
L^2_{\Progress}((0,T)\times \Omega;\Hs^1(\ell^2)\times L^2(\ell^2)).
\end{aligned}
\end{equation}
Recall that $\Hs_{\n}^2$ is defined in \eqref{eq:def_H_2_N_strong_weak} and $\Br_{\ell^2}$ in \eqref{eq:Bl2}.
We say that $$(v,\T)\in L^2_{\Progress}((0,\tau)\times \O;\Hs_{\n}^2\times H^1)$$ is an \emph{$L^2$-strong-weak solution} to \eqref{eq:SMR_v_T_strong_weak} on $[0,\tau]\times \O$ if a.s.\ for all $t\in [0,\tau]$
\begin{align*} 
v(t) 
&=\int_0^t \Big[\Delta v(s)+ \p\big[ \Lpp v+\op_{\kone} \T] +f_v  \Big]\,ds\\
&\qquad +
\int_0^t \Big(\one_{[0,\tau]} \Big[\p[ (\phi_{n}\cdot\nabla) v]   +g_{v,n}\Big]\Big)_{n\geq 1}\, d\Br_{\ell^2}(s),\\
\T(t)
&=
\int_0^t  \Big[\Dr \T+ f_{\T}\Big]\,ds+
\int_0^t \Big(\one_{[0,\tau]}[(\psi_{n}\cdot\nabla) \T +g_{n,\T}] \Big)_{n\geq 1}\, d\Br_{\ell^2}(s).
\end{align*}


\begin{proposition}[Stochastic maximal $L^2$-regularity]
\label{prop:SMR_2}
Let $T\in (0,\infty)$ and Assumption \ref{ass:well_posedness_primitive}\eqref{it:well_posedness_measurability}--\eqref{it:well_posedness_primitive_parabolicity} be satisfied.
Assume that $f_v,f_{\T},g_{v,n},g_{\T,n}$ satisfy \eqref{eq:f_v_etc_integrability_smr}. Then for any stopping time $\tau:\O\to [0,T]$ there exists a unique $L^2$-strong-weak solution to \eqref{eq:SMR_v_T_strong_weak} on $[0,\tau]\times \O$ such that 
$$
(v,\T)
\in L^2((0,\tau)\times \O;\Hs_{\n}^2\times H^1)\cap  L^2(\O;C([0,\tau];\Hs^1\times L^2)),
$$ 
and moreover for any $L^2$-strong-weak solution $(v,\T)$ to \eqref{eq:SMR_v_T_strong_weak} on $[0,\tau]\times \O$ we have
\begin{equation}
\label{eq:SMR_2_estimate_strong_weak}
\begin{aligned}
&\|(v,\T)\|_{L^2((0,\tau)\times \O;\Hs^2\times H^1)}+
\|(v,\T)\|_{L^2(\O;C([0,\tau];\Hs^1\times L^2))}\\
&\qquad \qquad \qquad
\lesssim \|(f_{v},f_{\T})\|_{L^2((0, \tau)\times \O;\Ls^2\times (H^1)^*)}\\
&\qquad \qquad \qquad
+\|((g_{n,v})_{n\geq 1},(g_{n,\T})_{n\geq 1})\|_{L^2((0, \tau)\times \O;\Hs^1(\ell^2)\times L^2(\ell^2))}
\end{aligned}
\end{equation}
where the implicit constant is independent of $f_{v},f_{\T}, (g_{n,v})_{n\geq 1},(g_{n,\T})_{n\geq 1}$ and $\tau$. 
\end{proposition}

The proof of Proposition \ref{prop:SMR_2} will be given in Section \ref{ss:proof_SMR_2} below. 

\begin{remark}\
\label{r:SMR_2}
\begin{enumerate}[{\rm(a)}]
\item\label{it:choice_AB_strong_weak}
For all $U=(v,\T)\in \Hs^2_{\n}\times H^1$, we set 
\begin{equation*}
A (\cdot)U:=
\begin{bmatrix}
-\Delta v - \p \big[\Lpp v+\op_{\kone} \T \big]\\
-\Dr \T
\end{bmatrix},
\ \ 
\text{ and } \ \ 
B_{n}(\cdot) U:=
\begin{bmatrix}
\p\big[(\phi_n(\cdot)\cdot \nabla )v\big]\\
(\psi_n(\cdot) \cdot \nabla )\T
\end{bmatrix},
\end{equation*} 
a.e. on $\R_+\times\Omega$. Then using the notation introduced in \cite[Section 3 in particular Definition 3.5]{AV19_QSEE_1}, Proposition \ref{prop:SMR_2} shows that  $(A,(B_n)_{n\geq 1})\in \mathcal{SMR}^{\bullet}_{2}(T)$.
\item\label{it:SMR_random_initial_time} By \cite[Proposition 3.9 and 3.12]{AV19_QSEE_2}, Proposition \ref{prop:SMR_2} also yields stochastic maximal $L^2$-estimates where in \eqref{eq:SMR_v_T_strong_weak} the starting time $0$ is replaced by any stopping time $\tau: \O\to [0,T]$ and non-trivial initial data from the space $(v_0,\T_0)\in L^0_{\F_{\tau}}(\O;\Hs^1\times L^2)$ where $\F_{\tau}$ is the $\sigma$-algebra of the $\tau$-past.
\end{enumerate}
\end{remark}

\subsection{Proof of Proposition \ref{prop:SMR_2}}
\label{ss:proof_SMR_2}
Here we prove Proposition \ref{prop:SMR_2}. To focus on the main difficulties, we only discuss the case $\hp^{\ell,m}_n\equiv 0$ since the operator $\Lpp v$ can be shown to be of lower order type provided Assumption \ref{ass:well_posedness_primitive}\eqref{it:well_posedness_primitive_phi_smoothness} holds. For details, we refer to Remark \ref{r:Lpp_lower_order} below.

\begin{proof}[Proof of Proposition \ref{prop:SMR_2} -- Case $\hp_{n}^{\ell,m}\equiv 0$]
Let us set 
$$
X_0=L^2\times (H^1)^*
\qquad \text{ and }\qquad 
X_1=\Hs^{2}_{\n}\times H^1.
$$
To prove the claim, we employ the method of continuity as in \cite[Proposition 3.13]{AV19_QSEE_2}. 
Let us denote the strong Neumann Laplacian by
\begin{equation}
\label{eq:def_strong_Neumann}
\Delta_{\n}\colon\Hs^2_{\n}(\Dom)\subseteq \Ls^2\to \Ls^2,  \quad \hbox{where}\quad\Delta_{\n} v=\Delta v. 
\end{equation}
Here we use that $\p\Delta_{\n}=\Delta_{\n}\p=\Delta_{\n}$.
It is well-known that $\Delta_{\n}$ is self-adjoint. Similarly, by form methods, one can check that the weak Robin Laplacian $\Dr$ defined in \eqref{eq:weak_Laplace_operator_robin} is self-adjoint as well. Thus, it is well-established that $((-\Delta_{\n},-\Dr),0)\in \mathcal{SMR}^{\bullet}_2(T)$, see e.g.\ \cite[Theorem 6.14]{DPZ} which  applies up to a shift, and compare also with \cite[Section 3.2]{AV19_QSEE_1}.

For all $\lambda\in [0,1]$ and $U=(v,\T)\in X_1$, we set
\begin{equation}
\label{eq:A_B_temperature_maximal_L_2_regularity}
A_{\lambda}U:=
\begin{bmatrix}
-\Delta_{\n} v -\lambda\p \big[\op_{\kone} \T \big]\\
-\Dr \T
\end{bmatrix},
\ \ 
\text{ and } \ \ 
B_{n,\lambda} U:=\lambda
\begin{bmatrix}
\p\big[(\phi_n\cdot \nabla )v\big]\\
(\psi_n \cdot \nabla )\T
\end{bmatrix}.
\end{equation}

Let $\calL_2(\ell^2,X_{1/2})$ be the space of all Hilbert-Schmidt operators endowed with its natural norm.
By the previous considerations and the method of continuity in \cite[Proposition 3.13 and Remark 3.14]{AV19_QSEE_2}, it remains to prove the existence of $C>0$ such that, for each stopping time $\tau:\O\to [0,T]$, each 
\begin{align*}
f&=(f_v,f_{\T})\in L^2_{\Progress}((0,T)\times \O;X_0), \\ g&=(g_{v},g_{\T})=((g_{n,v})_{n\geq 1},(g_{n,\T})_{n\geq 1})\in L^2_{\Progress}((0,T)\times \O;\calL_2(\ell^2,X_{1/2})),
\end{align*}
 and each  $L^2$-strong-weak solution $$(v,\T)\in L^2_{\Progress}((0,\tau)\times \O;X_1)\cap L^2_{\Progress}(\O;C([0,\tau];X_{1/2}))$$ 
 on $[0,\tau]$ to 
\begin{equation}
\label{eq:a_priori_estimates_v_theta}
\begin{cases}
\displaystyle{dv -\Big[\Delta v-\lambda\p [ \op_{\kone} \T]\Big]dt =f_v dt 
+ \sum_{n\geq 1} \Big[\lambda\p [(\phi_n\cdot \nabla )v]+ g_{n,v}\Big] d\beta_t^n}, & \text{on }\Dom,\\
\displaystyle{d\Temp -\Delta \Temp dt =f_{\T} dt + \sum_{n\geq 1} 
\Big[\lambda(\psi_n\cdot \nabla )\Temp+ g_{n,\T}\Big]d\beta_t^n},
& \text{on }\Dom,\\
\partial_3 v(\cdot,0)=\partial_3 v(\cdot,-h)=0, &\text{on }\Tor^2,\\
v(0)=0,\qquad \Temp(0)=0,& \text{on }\Dom,
\end{cases}
\end{equation}
we have
\begin{equation}
\label{eq:claimed_a_priori_estimate}
\begin{aligned}
\|v\|_{L^2((0,\tau)\times\O;H^2)}
&+ \|\Temp\|_{L^2((0,\tau)\times\O;H^1)}\\		
&\leq C\|f\|_{L^2((0,\tau)\times \O;X_0)}+ C\|g\|_{L^2((0,\tau)\times \O;\calL_2(\ell^2,X_{1/2}))}.
\end{aligned}
\end{equation}

We split the proof of \eqref{eq:claimed_a_priori_estimate} into two steps. The key observation is that $v$ does not appear in the equation for $\Temp$. Thus, first we prove an estimate for $\Temp$ and then we use the latter to obtain the estimate    for $v$.

\emph{Step 1: Estimate on $\T$}. To show the maximal regularity estimate 
$$ 
\|\Temp\|_{L^2((0,\tau)\times\O;H^1)}\leq 
C\big(\|f_{\T}\|_{L^2((0,\tau)\times \O;(H^{1})^*)}+ \|g_{\T}\|_{L^2((0,\tau)\times \O;L^2(\ell^2))}\big),
$$ 
the idea is to apply It\^{o}'s formula to $\T\mapsto\|\T\|_{L^2}^2$. To this end we use an approximation argument. 
Recall that, by definition of $L^2$-strong-weak solution to \eqref{eq:a_priori_estimates_v_theta}, $\T$ satisfies, a.s.\ for all $t\in [0,\tau]$,
\begin{equation}
\label{eq:equation_theta_integrated}
\begin{aligned}
\T(t)
=\int_0^t \one_{[0,\tau]} \Dr \T (s)\,ds
&+\int_0^{t}\one_{[0,\tau]} f_{\T}(s)ds\\
&
+\sum_{n\geq 1}\int_0^t \one_{[0,\tau]}\Big[\lambda(\psi_n(s)\cdot \nabla) \theta(s)+ g_{n,\theta}(s)\Big]d\beta_s^n.
\end{aligned}
\end{equation}
For technical reasons, it is convenient to work with processes defined on $[0,T]$ rather than on the stochastic interval $[0,\tau]$. Thus we set $\Ttwo(t):=\T(t\wedge \tau)$ for $t\in [0,T]$. Note that $\Ttwo=\T$ a.e.\ on $[0,\tau]$, and for each $t\in[0,T]$, $\Ttwo(t)$ is equal to the right hand side in \eqref{eq:equation_theta_integrated}. 

Since $1+\Dr$ is a sectorial operator (see e.g.\ \cite[Definition 10.1.1]{Analysis2} for the definition of this notion), 
\begin{equation}
\label{eq:approximation_Dr}
\lim_{t\to \infty} t (t+1+\Dr)^{-1}= I \ \ \text{ strongly in } (H^1)^*,
\end{equation}
where $I$ is the identity operator. Since $\Do(\Dr)=H^1$ and $\Do((\Dr)^{1/2})=L^{2}$, \eqref{eq:approximation_Dr} also holds with  $(H^1)^*$ replaced by either $H^1$ or $L^{2}$.
For each $k\geq 1$, let
\begin{equation}
\label{eq:def_e_k}
\e_k :=k(k+1+\Dr)^{-1}, \quad \Ttwo_k:=\e_k\Ttwo, \quad \hbox{and}\quad \theta_k:=\e_k\theta.
\end{equation}
Note that $\e_k$ and $\Dr$ commute on $(H^{1})^*$. Thus, applying $\e_k$ to \eqref{eq:equation_theta_integrated} we have, a.s.\ for all $t\in [0,\tau]$, 
\begin{equation}
\label{eq:equation_theta_integrated_k}
\begin{aligned}
&\Ttwo_k(t)-\int_0^t \one_{[0,\tau]} \Dr \T_k (s)\,ds
=\int_0^{t}\one_{[0,\tau]} \e_k f_{\T}(s)ds\\
&\qquad \qquad
+\sum_{n\geq 1}\int_0^t \one_{[0,\tau]}\e_k\Big[(\psi_n(s)\cdot \nabla) \theta(s)+ g_{n,\theta}(s)\Big]d\beta_s^n.
\end{aligned}
\end{equation}
Since $\theta_k\in L^2((0,\tau)\times \O ;H^2)$ by the regularity of the strong Robin Laplacian, we have $ \Dr \T_k=\Delta \T_k$ in the strong sense and we may apply It\^{o}'s formula to compute $\|\theta_k\|_{L^2}^2$. Hence,  a.s.\ for all $t\in [0,T]$,
\begin{align*}
&\|\Ttwo_k(t)\|_{L^2}^2 +2 \int_{0}^t\one_{[0,\tau]} \|\nabla \Temp_k\|_{L^2}^2 ds 
+2\alpha\int_0^{t} \one_{[0,\tau]} \|\T_k(\cdot,0)\|_{L^2(\Tor^2)}^2\,ds 
\\
&\qquad
=\int_{0}^t \one_{[0,\tau]} 2(\e_k f_{\T}(s),\Temp_k(s))_{L^2}\,ds \\
&\qquad
+\int_{0}^t \one_{[0,\tau]}\sum_{n\geq 1} \Big(\big\|\big(\e_k[\lambda(\psi_n\cdot \nabla )\Temp+ g_{n,\T}]\big)_{n\geq 1}\big\|_{L^2}^2\Big) ds\\
&\qquad
+ 2\sum_{n\geq 1} \int_0^t\one_{[0,\tau]} \Big[\Big(\e_k[\lambda(\psi_n(s)\cdot \nabla) \Temp(s) +g_{\theta,n}(s)], \Temp_k(s)\Big)_{L^2}\Big]d\beta_s^n,
\end{align*}
where we also integrated by parts and used that $\partial_3 \T_k(\cdot,-h)=0$ and $\partial_3 \T_k(\cdot,0)=-\alpha \T_k(\cdot,0)$ on $\Tor^2$.

Recall that $\T\in L^2(\O;C([0,\tau];L^2))$ and hence $\Ttwo \in L^2(\O; C([0,T];L^2))$.
Moreover, the trace map $H^{1/2+r}\ni f \mapsto f(\cdot,x_3=0)\in L^2(\Tor^2)$ is bounded for all $r>0$ and therefore $\T(\cdot,x_3=0)\in L^2((0,\tau)\times \O;L^2(\Tor^2))$.
Thus, we may take $k\to\infty$ in the previous identity, and obtain a.s.\ for all $t\in [0,T]$,
\begin{equation}
\label{eq:Ito_formula_theta_L2}
\begin{aligned}
\|\Ttwo(t)\|_{L^2}^2 
&+2 \int_{0}^t\one_{[0,\tau]} \|\nabla \Temp\|_{L^2}^2 ds 
+2\alpha\int_0^{t} \one_{[0,\tau]} \|\T(\cdot,0)\|_{L^2(\Tor^2)}^2\,ds 
\\
&=2\int_{0}^t \one_{[0,\tau]} \l f_{\T},\Temp\r\, ds
+
\int_{0}^t \one_{[0,\tau]}\sum_{n\geq 1} \Big(\big\| \lambda(\psi_n\cdot \nabla )\Temp+ g_{n,\T}\big\|_{L^2}^2\Big) ds\\
&
+ 2\sum_{n\geq 1} \int_0^t\one_{[0,\tau]} \Big(\lambda(\psi_n(s)\cdot \nabla) \Temp(s) +g_{\theta,n}(s), \Temp(s)\Big)_{L^2}\,d\beta_s^n\\
&
=:I_t+II_t+III_t,
\end{aligned}
\end{equation}
where, as above, $\l \cdot,\cdot\r$ denotes the duality pairing for $H^1$ and $ (H^1)^*$.

Let $0\leq s\leq t\leq T$. Note that for all $r\in (0,\frac{1}{2})$ by the continuity of the trace map, we have
\begin{equation}
\begin{aligned}
\label{eq:estimate_I_temp_L_2_estimate}
\E\int_0^t \one_{[0,\tau]}\|\T(\cdot,0)\|_{L^2(\Tor^2)}^2\,ds 
&\lesssim \E\int_0^t \one_{[0,\tau]}\|\T\|_{H^{\frac{1}{2}+r}}^2\,ds,\\
\E\Big[\sup_{s\in [0,t]}|I_s|\Big]
&\leq 2 (\E\|f_{\T}\|_{L^2(0,\tau;(H^{1})^*)}^2)^{1/2}(\E\|\theta\|_{L^2(0,\tau;H^{1})}^2)^{1/2}.
\end{aligned}
\end{equation}
Let $\ellip$ be as in Assumption \ref{ass:well_posedness_primitive} and fix $\ellip' \in (\nu,2)$. Thus, for some $c_\nu>0$,
\begin{align*}
&\E\Big[\sup_{s\in [0,t]}|II_s|\Big]\\
&\leq \E \sum_{n\geq 1} \int_0^t \one_{[0,\tau]}\big\|\big(\lambda(\psi_n\cdot \nabla )\Temp+ g_{n,\T}\big)_{n\geq 1}\big\|_{L^2}^2 ds\\
&\leq \frac{\ellip'}{\ellip}\E \sum_{n\geq 1} \int_0^t \one_{[0,\tau]}\big\|\big((\psi_n\cdot \nabla )\Temp\big)_{n\geq 1}\big\|_{L^2}^2 ds
+c_{\ellip} \E\int_0^t \one_{[0,\tau]}\|(g_{n,\T})_{n\geq 1}\|_{L^2(\ell^2)}^2ds \\
&= \frac{\ellip'}{\ellip}\E \sum_{n\geq 1} \int_0^t \one_{[0,\tau]}\int_{\Dom}\Big(\sum_{j=1}^3\psi_n^j  \partial_j \T\Big)^2 dx ds
+c_{\ellip} \E\int_0^t \one_{[0,\tau]}\|(g_{n,\T})_{n\geq 1}\|_{L^2(\ell^2)}^2ds \\
&\leq \ellip'\E\int_0^t \one_{[0,\tau]}\|\nabla\T\|_{L^2}^2 ds
+c_{\ellip} \E\int_0^t \one_{[0,\tau]}\|(g_{n,\T})_{n\geq 1}\|_{L^2(\ell^2)}^2 ds.
\end{align*}
Since $\ellip'<2$ and $\E[III_T]=0$, taking the expectation in \eqref{eq:Ito_formula_theta_L2} and using standard interpolation inequalities, one has
\begin{equation}
\label{eq:first_step_estimate_temperature_L_2}
\begin{aligned}
 \E\int_{0}^{\tau} \| \nabla \T\|^2_{L^2}\,ds 
&\leq C_0 \E\|\theta\|_{L^2(0,\tau;L^2)}^2\\
&+ 
 C_0\E\|f_{\T}\|_{L^2(0,\tau;(H^{1})^*)}^2 + C_{0}\E \|(g_{n,\T})_{n\geq 1}\|_{L^2(0,\tau;L^2(\ell^2))}^2
\end{aligned}
\end{equation}
where $C_0>0$ is independent of $\tau,f_{\theta},g_{\theta}$ and $\lambda$.

The Burkholder-Davis-Gundy inequality implies
\begin{align*}
&\E\Big[\sup_{s\in [0,t]}| III_t|\Big] \\
&\lesssim \E \Big[
  \int_0^t\one_{[0,\tau]} \sum_{n\geq 1}\Big|  \Big(\lambda(\psi_n \cdot \nabla) \Temp +g_{\theta,n}, \Temp\Big)_{L^2}\Big|^{2} ds\Big]^{1/2}\\
 &\lesssim \E \Big[\Big(\sup_{s\in [0,t\wedge \tau]} \|\T(s)\|_{L^2}^2 \Big)^{1/2}
  \Big( \int_0^t\one_{[0,\tau]} \sum_{n\geq 1}\Big\|\lambda(\psi_n\cdot \nabla) \Temp+g_{\theta,n}\Big\|_{L^2}^2 ds\Big)^{1/2}\Big]\\
&\leq \frac{1}{2} 
 \E \Big[\sup_{s\in [0,t\wedge \tau]} \|\T(s)\|_{L^2}^2 \Big]
+ C\E \int_0^t\one_{[0,\tau]} \Big(\|\nabla \T\|_{L^2}^2 +\|(g_{\theta,n})_{n\geq 1}\|_{L^2(\ell^2)}^2\Big) ds\\
 &\stackrel{\eqref{eq:first_step_estimate_temperature_L_2}}{\leq} \frac{1}{2}
  \E \Big[\sup_{s\in [0,t\wedge \tau]} \|\T(s)\|_{L^2}^2 \Big]
+ C_1
\E \int_0^t\one_{[0,\tau]} \Big(\| \Temp\|_{L^2}^2 +\|(g_{\theta,n})_{n\geq 1}\|_{L^2(\ell^2)}^2\Big) ds.
\end{align*}

Note that $\E\big[\sup_{s\in [0,t\wedge \tau]} \|\T(s)\|_{L^2}^2 \big]\leq \E\big[\sup_{s\in [0,t]} \|\Ttwo(s)\|_{L^2}^2 \big]$ since $\Ttwo=\T$ a.e.\ on $[0,\tau]$. Thus, taking $\E\big[\sup_{s\in [0,t]} \cdot \big]$ in \eqref{eq:Ito_formula_theta_L2} and using the above estimates on $I_t,II_t,III_t$ and \eqref{eq:first_step_estimate_temperature_L_2} to estimate $\E\|\nabla \T\|_{L^2(0,\tau;L^2)}^2$, one gets 
\begin{equation}
\label{eq:estimate_step_before_Gronwall}
\E\Big[\sup_{s\in [0,t\wedge \tau]}\|\T(s)\|_{L^2}^2 \Big]\lesssim
 \E\int_0^{t}\one_{[0,\tau]} \|\theta(s)\|_{L^2}^2 \,ds+  
N_{f,g}^{\T}(t). 
\end{equation}
where 
$
N_{f,g}^{\T}(t):=\E \|f_{\theta}\|_{L^2(0,t\wedge \tau;(H^1)^*)}^2+
\E \|g_{\theta}\|_{L^2(0,t\wedge \tau;L^2(\ell^2))}^2
$
and the implicit constant is independent of $\lambda,\T,f_{\T}$ and $g_{\theta}$. Set $y(t):=\E\big[\sup_{s\in [0,t\wedge \tau]}\|\T(s)\|_{L^2}^2 \big]$ for $t\in [0,T]$.  
By \eqref{eq:estimate_step_before_Gronwall} we have $y(t)\lesssim \int_0^t y(s)ds + N_{f,g}(t)$. Thus by Gronwall's inequality, we have 
$$
\E\Big[\sup_{s\in [0, \tau]}\|\T(s)\|_{L^2(\Dom)}^2 \Big]\lesssim N^{\T}_{f,g}(T).
$$
Combining the previous and \eqref{eq:first_step_estimate_temperature_L_2} one obtains the claimed estimate in Step 1.

\emph{Step 2: Estimate on $v$}. To prove  of \eqref{eq:claimed_a_priori_estimate}, let us begin by collecting some useful facts.
Firstly, for all $f\in H^{1}(\Dom;\R^2)$,
\begin{equation}
\label{eq:contractivity_of_hydrostatic}
\|\partial_j \p  f\|_{L^2}\leq \|\partial_j f\|_{L^2}\ \ \ \text{ for all }j\in \{1,2,3\}.
\end{equation}
Since  $\partial_3 \p f=\partial_3 f $,  \eqref{eq:contractivity_of_hydrostatic} follows trivially in case $j=3$. For $i\in \{1,2\}$, note that $\partial_i(\p f) =\p (\partial_i f)$. Since $\p$ is an orthogonal projection on $L^2(\Dom;\R^2)$, \eqref{eq:contractivity_of_hydrostatic} for $j\in \{1,2\}$ follows from the previous identities.

Let $r\in (1,\infty)$ be such that $\frac{1}{2+\delta}+\frac{1}{r}=\frac{1}{2}$, where $\delta$ is as in Assumption~\ref{ass:well_posedness_primitive}. Note that by Assumption \ref{ass:well_posedness_primitive}\eqref{it:well_posedness_primitive_kone_smoothness}, a.s.\ for all $t\in\R_+$ and $\Tthree\in H^{1}$, by Cauchy-Schwartz inequality,
\begin{equation}
\label{eq:estimate_theta_term_in_v_variable}
\begin{aligned}
&\|\p[\op_{\kone} \Tthree ]\|_{L^2}
\lesssim  
\|\op_{\kone} \Tthree \|_{L^2}\\
&\lesssim_{M}\sup_{i\in \{1,2\}}
\|\partial_i \Tthree  \|_{L^2}+
\Big\|x_{\h}\mapsto 
\|\partial_i\kone(t,x_{\h},\cdot)\|_{L^2(-h,0)}
\|\Tthree (x_{\h},\cdot)\|_{L^2(-h,0)} \Big\|_{L^{2}(\Tor^2)}\\
&\lesssim_{M,\delta}
\| \Tthree\|_{H^1}+\| \Tthree\|_{L^{r}(\Tor^2;L^2(-h,0))}\stackrel{(*)}{\lesssim}_{\delta} \| \Tthree\|_{H^1}
\end{aligned}
\end{equation}
where in $(*)$ we used that $H^1=H^1(\Dom)\embed H^1(\Tor^2;L^2(-h,0))\embed L^r(\Tor^2;L^2(-h,0))$.

The previous estimate and Step 1 yield
\begin{equation}
\label{eq:estimate_temperature_step_2_SMR}
\E\|\p [\op_{\kone} \T]\|_{L^2(0,\tau;L^2)}^2\lesssim \E\|f_{\T}\|_{L^2(0,\tau;(H^{1})^*)}^2
+\E\|g_{\T}\|_{L^2(0,\tau;L^2(\ell^2))}^2
\end{equation}
with implicit constant independent of $\lambda,\T,f_{\T}$ and $g_{\T}$.

Applying It\^{o}'s formula to $v\mapsto \|v\|_{L^2}^2$ (since $v\in L^2(0,\tau;H^2)$ a.s., there is no need for an approximation argument), an integration by parts and using the argument performed in Step 1 and \eqref{eq:estimate_temperature_step_2_SMR}, one can check that,
\begin{equation}
\label{eq:L_2_estimate_v_step_2_smr}
\E \Big[\sup_{t\in [0,\tau)}\|v(t)\|_{L^2}^2\Big]+\E \int_0^{\tau} \|\nabla v(t)\|_{L^2}^2\,dt 
\lesssim N_{f,g},
\end{equation}
where 
\begin{align*}
N_{f,g}
&:=\E\|f_{\T}\|_{L^2(0, \tau;(H^{1})^*)}^2+\E\|g_{\T}\|_{L^2(0, \tau;L^2(\ell^2))}^2\\
&\ +\E\|f_{v}\|_{L^2(0, \tau;L^2)}^2+\E\|g_{v}\|_{L^2(0, \tau;H^{1}(\ell^2))}^2,
\end{align*}
and the implicit constant in \eqref{eq:L_2_estimate_v_step_2_smr} is independent of $\lambda,\T,f_{v},f_{\T}, g_{v}$ and $g_{\T}$.

To complete the proof of this step, it remains to show 
\begin{equation}
\label{eq:claim_Delta_v_smr}
\E \int_{0}^{\tau}\|\Delta v\|_{L^2}^2 \,ds\lesssim N_{f,g},
\end{equation}
where as above, the implicit constant in \eqref{eq:L_2_estimate_v_step_2_smr} is independent of $\lambda,\T,f_{v},f_{\T}, g_{v}$ and $g_{\T}$. To this end, we apply It\^{o}'s formula to $v\mapsto \|\nabla v\|_{L^2}^2$. Set $\vtau(t):=v(t\wedge \tau)$. Using an approximation argument similar to Step 1 and an integration by parts, one has a.s.\ for all $t\in [0,T]$,
\begin{equation*}
\begin{aligned}
\|\nabla \vtau(t)\|_{L^2}^2 
&+2 \int_{0}^t\one_{[0,\tau]} \|\Delta v\|_{L^2}^2\, ds 
=-2 \int_{0}^t \one_{[0,\tau]} ( f_{v}+ \lambda \p[\op_{\kone} \T],\Delta v)_{L^2}\,ds\\
&
+ \int_{0}^t \one_{[0,\tau]}\sum_{n\geq 1} \big\|\lambda\nabla \p[(\phi_n\cdot \nabla )v]+ \nabla g_{n,v}\big\|_{L^2}^2\, ds\\
&
+ 2\sum_{n\geq 1} \int_0^t\one_{[0,\tau]} \big(\lambda \nabla \p[(\phi_n\cdot \nabla) v] +\nabla g_{v,n}, \nabla v\big)_{L^2}\, d\beta_s^n\\
&
=:IV_t+V_t+VI_t.
\end{aligned}
\end{equation*}
Note that $\E[VI_T]=\E[VI_0]=0$. Thus, taking $t=T$ and the expected value in the previous formula we have 
\begin{equation}
\label{eq:step_2_final_claim}
2\E \int_{0}^{\tau} \|\Delta v\|_{L^2}^2 ds \leq \E[IV_T]+\E[V_T].
\end{equation}
By \eqref{eq:estimate_temperature_step_2_SMR} and the Cauchy-Schwartz inequality we have, for all $\varepsilon>0$,
\begin{equation}
\label{eq:estimate_IV_T_smr_2_proof_strong_weak_revision}
\begin{aligned}
\E[IV_T]
&\leq 
\varepsilon \E\int_0^{\tau} \|\Delta v\|_{L^2}^2\,ds 
+ C_{\varepsilon}\E \int_{0}^{\tau}\big(\|f_v\|_{L^2}^2+ \|\lambda \p[\op_{\kone} \T]\big\|_{L^2}^2\big)\,ds\\
&\leq 
\varepsilon \E\int_0^{\tau} \|\Delta v\|_{L^2}^2\,ds + C_{\varepsilon} N_{f,g}
\end{aligned}
\end{equation}
where $C_{\varepsilon}$ is independent of $\lambda,f_{v},f_{\T},g_{v},g_{\T}$ and $N_{f,g}$ is as below \eqref{eq:L_2_estimate_v_step_2_smr}.

We claim that, for some $c<2$ and $C>0$ (both independent of $\lambda,f_{v},f_{\T},g_{v},g_{\T}$),
\begin{equation}
\label{eq:claim_V_T_smr_proof}
\E[V_T]\leq c \E\int_0^{\tau}\|\Delta v\|_{L^2}^2 \,ds+C \E\int_0^{\tau} \big(\|v\|_{L^2}^2+\|g_{v}\|_{H^{1}(\ell^2)}^2)\,ds.
\end{equation}
It is easy to see that, if \eqref{eq:claim_V_T_smr_proof} holds, then \eqref{eq:claim_Delta_v_smr} follows by combining \eqref{eq:L_2_estimate_v_step_2_smr}, \eqref{eq:step_2_final_claim} and \eqref{eq:estimate_IV_T_smr_2_proof_strong_weak_revision} with $\varepsilon\in (0,2-c)$.

To prove \eqref{eq:claim_V_T_smr_proof}, we begin with a pointwise bound. Fix $\ellip<\ellip'<\ellip''<2$ and let $\varrho\in (2,6)$ be such that $\frac{1}{3+\delta}+\frac{1}{\varrho}=\frac{1}{2}$. Here $\delta$ and $\nu$ are as in Assumption \ref{ass:well_posedness_primitive}\eqref{it:well_posedness_primitive_phi_smoothness} and \eqref{it:well_posedness_primitive_parabolicity}. Since 
$
\partial_j [(\phi_n \cdot\nabla)v]=\sum_{k=1}^3\big( \phi_n^k\, \partial_{k,j}^2 v+  \partial_k v \, \partial_j \phi_n^k \big)$, one has 
\begin{align*}
&\sum_{n\geq 1}\| \nabla \p[(\phi_n\cdot \nabla )v]\|_{L^2}^2
= \sum_{j=1}^3\sum_{n\geq 1}\| \partial_j \p[(\phi_n\cdot \nabla )v]\|_{L^2}^2\\
&\qquad 
\stackrel{\eqref{eq:contractivity_of_hydrostatic}}{\leq }\sum_{j=1}^3 
\sum_{n\geq 1}\| \partial_j [(\phi_n\cdot \nabla )v]\|_{L^2}^2 \\
&\qquad  
\leq \sum_{j=1}^3 \sum_{h=1}^2  \sum_{n\geq 1} 
\Big( \frac{\ellip'}{\ellip} \int_{\Dom}
  \Big|\sum_{k=1}^3\phi_n^k\partial_{j,k}v^h  \Big|^2 \,dx+c_{\ellip} \int_{\Dom}
  \Big|\partial_k v^h \, \sum_{k=1}^3\partial_j \phi_n^k \Big|^2  \,dx\Big)
\\
&\qquad \leq 
  \ellip'\Big(\sum_{h=1}^2\sum_{k,j=1}^3 \int_{\Dom}
|\partial_{j,k}v^h|^2 \,dx\Big)
+c_{\ellip}' \max_{h,k}\Big(\big\|(\partial_h\phi_n^k)_{n\geq 1}\big\|_{L^{3+\alpha}(\ell^2)}^2 \|\nabla v\|_{L^{\varrho}}^2\Big)\\
&\qquad \stackrel{(i)}{=}
  \ellip'\|\Delta v\|_{L^2}^2
+c_{\ellip} M \|\nabla v\|_{L^{\varrho}}^2
\end{align*} 
where $M$ is as in Assumption \ref{ass:well_posedness_primitive} and in $(i)$ we used the Kadlec’s formula (see Lemma \ref{l:Kadlec_formula} with $\beta=0$) as well as the boundary conditions in \eqref{eq:a_priori_estimates_v_theta}.

Since $\varrho<6$, by the Sobolev embedding $H^{1}\embed L^6$ and standard interpolation theory, there exists $\theta\in (0,1)$ such that $\|\nabla v\|_{L^{\varrho}}\lesssim (\|\Delta v\|_{L^2}+\|v\|_{L^2})^{1-\theta}\|v\|_{L^2}^{\theta}$. In particular, $
\|\nabla v\|_{L^{\varrho}}^2
\leq (\ellip''-\ellip')\|\Delta v\|_{L^2}^2 + C_{\ellip',\ellip''}\|v\|_{L^2}^2$ and therefore
\begin{equation}
\label{eq:estimate_Ito_correction_v_smr}
\sum_{n\geq 1}\| \nabla \p[(\phi_n\cdot \nabla )v]\|_{L^2}^2	\leq 
 \ellip''\|\Delta v\|_{L^2}^2
+c_{\ellip'}C_{\ellip',\ellip''}M\| v\|_{L^{2}}^2.
\end{equation}
By \eqref{eq:estimate_Ito_correction_v_smr}, $\lambda\in [0,1]$ and the Young inequality, one can readily check that the pointwise estimate \eqref{eq:estimate_Ito_correction_v_smr} implies \eqref{eq:claim_V_T_smr_proof}. Thus as explained above, the latter yields \eqref{eq:claim_Delta_v_smr} as desired. 
\end{proof}

\begin{remark}[Proof of Proposition \ref{prop:SMR_2} -- Case $\hp_{n}^{\ell,m}\neq 0$]
\label{r:Lpp_lower_order}
If $\hp_{n}^{\ell,m}\neq 0$, then one can repeat the argument in Step 2  of the above proofs recalling that, for $\ell,m\in \{1,2\}$,
\begin{equation}
\label{eq:estimate_to_control_hp}
\|(\hp_{n}^{\ell,m} f)_{n\geq 1}\|_{L^2(\ell^2)}
\leq  \|(\hp_{n}^{\ell,m})_{n\geq 1}\|_{L^{3+\delta}(\ell^2)} \|f\|_{L^{\varrho}}
\leq \varepsilon \|\nabla f\|_{L^2} + C_{\varepsilon,M} \|f\|_{L^2},
\end{equation}
where $\frac{1}{\varrho}+\frac{1}{3+\delta}=\frac{1}{2}$ and we used the interpolation argument which also yield \eqref{eq:estimate_Ito_correction_v_smr}. 

Let us note that \eqref{eq:estimate_to_control_hp} has to be used twice. First to show that
$$
\Big| \int_{\Dom}\Lpp  v\, v\,  dx  \Big|\lesssim \|\nabla v\|_{L^2} \|(\hp_n^{\ell,m} v)_{n\geq 1}\|_{L^2(\ell^2)}
\leq \varepsilon \|\nabla v\|_{L^2}^2+ C_{\varepsilon,M} \|v\|_{L^2}^2,
$$
for $ \ell,m\in \{1,2\}$,
which provides \eqref{eq:L_2_estimate_v_step_2_smr} and the second one to show that 
\begin{align*}
\|\Lpp v\|_{L^2}
&\lesssim \varepsilon \max_{1\leq i,j\leq 3} \|\partial_{i,j}^2 v\|_{L^2}^2+ C_{\varepsilon,M} \|v\|_{L^2}^2\\
&\eqsim \varepsilon \max_{1\leq i,j\leq 3} \|\partial_{i,j}^2 v\|_{L^2}^2+ C_{\varepsilon,M} \|v\|_{L^2}^2,
\end{align*}
which, in combination with the Kadlec formula of Lemma \ref{l:Kadlec_formula}, yields \eqref{eq:claim_V_T_smr_proof}.
In both cases, one chooses $\varepsilon>0$ small enough to absorb the leading terms in the LHS of the corresponding estimate.
\end{remark}

\section{Proof of the main results in the strong-weak setting}
\label{s:proofs_strong_weak}
In this section we have collected the proofs of our main results in the strong-weak setting. Namely, the proofs of Theorems \ref{t:local_primitive} and \ref{t:global_primitive}
are given in Subsection \ref{ss:proof_local_strong_weak} and \ref{ss:proof_global_strong_weak}, respectively.

\subsection{Proof of Theorem \ref{t:local_primitive}}
\label{ss:proof_local_strong_weak}
To prove Theorem \ref{t:local_primitive} we employ the results in  \cite{AV19_QSEE_1, AV19_QSEE_2}, and therefore we reformulate \eqref{eq:primitive_weak_strong} as a semilinear stochastic evolution equation. 
To this end, let 
\begin{align*}
X_0=\Ls^2\times (H^1)^* \quad \hbox{and}  \quad X_1=\Hs^2_{\n}\times H^1,
\end{align*}
the linear operators $(A,B)=(A,(B_n)_{n\geq 1})$ be as in Remark \ref{r:SMR_2}\eqref{it:choice_AB_strong_weak} and for all $U:=(v,\T)\in X_1$ we set for the non-linearities
\begin{align*}
F(\cdot,U)
&:=\begin{bmatrix}
\p[ (v\cdot\nabla_{\h}) v + w(v)\cdot \partial_3 v +\fv(\cdot,v,\T,\nabla v)+\Lpg (\cdot,v)]\\
 \div_{\h}(v \T)+\partial_3 (w(v)\T)  \ft(\cdot,v,\T,\nabla v)
\end{bmatrix},\\
G(\cdot,U)
&:=\begin{bmatrix}
 (\p[\gvn(\cdot,v)])_{n\geq 1}\\
  (\gtn(\cdot,v,\T))_{n\geq 1}
\end{bmatrix},
\end{align*}
where $w(v)$ is as in \eqref{eq:def_w_v} and 
\begin{equation*}
\Lpg(\cdot,v):=\Big(\sum_{n\geq 1} \sum_{m=1}^2 \hp_n^{\ell,m} (t,x)
\big(\q[\gvn(\cdot,v)]\big)^m\Big)_{\ell=1}^2,
\end{equation*}
cf.\ 
\eqref{eq:def_P_gamma}.
With the above notation, \eqref{eq:primitive_weak_strong}-\eqref{eq:boundary_conditions_strong_weak} can be reformulated as a stochastic evolution equation on $X_0$
of the form
\begin{equation}
\label{eq:abstract_formulation_strong_weak}
\begin{cases}
d U+ A (\cdot)U\,dt=F(\cdot,U)dt + (B (\cdot)U +G(\cdot,U))d\Br_{\ell^2}(t),\\
U(0)=(v_0,\T_0),
\end{cases}
\end{equation}
where $\Br_{\ell^2}$ is the $\ell^2$-cylindrical Brownian motion induced by $(\beta^n)_{n\geq 1}$, see \eqref{eq:Bl2}.
It is straightforward  to see that $((v,\T),\tau)$ is an $L^2$-maximal (resp.\ -local) solution to \eqref{eq:primitive_weak_strong}-\eqref{eq:boundary_conditions_strong_weak} in the sense of Definition \ref{def:sol_strong_weak} if and only if  $U=(v,\T)$ is an $L^2$-maximal (resp.\ -local) solution to \eqref{eq:abstract_formulation_strong_weak} in the sense of \cite[Definition 4.4]{AV19_QSEE_1}.
With this preparation we can prove Theorem \ref{t:local_primitive}.

\begin{proof}[Proof of Theorem \ref{t:local_primitive}] 
Let us begin by proving the existence of maximal $L^2$-strong-weak solutions and Theorem \ref{t:local_primitive}\eqref{it:path_regularity_strong_weak}. To this end,
by the above discussion and Proposition \ref{prop:SMR_2}, the existence of an $L^2$-maximal solution to \eqref{eq:primitive_weak_strong}-\eqref{eq:boundary_conditions_strong_weak} satisfying \eqref{it:path_regularity_strong_weak} follows from \cite[Theorem 4.8]{AV19_QSEE_1} provided assumptions (HF) and (HG) in \cite[Section 4]{AV19_QSEE_1} hold. To check these assumptions, we need to estimate the nonlinearities $F$ and $G$. 

For notational convenience, for $U=(v,\T)\in X_1$, we let
\begin{align*}
&F_1(\cdot,U):=\begin{bmatrix}
\p\big[ (v\cdot\nabla_{\h}) v + w(v) \partial_3 v\big]\\
 \div_{\h}(v \T)+\partial_3 (w(v)\T) 
\end{bmatrix},
\qquad 
F_2(\cdot,U):=\begin{bmatrix}
\p[ \fv(\cdot,v,\T,\nabla v)]\\
  \ft(\cdot,v,\T,\nabla v)
\end{bmatrix}, 
\end{align*}
and
$F_{3}(\cdot,U):=
\big[
\Lpg(\cdot,v), \
0
\big]^{{\rm T}}$. Finally, set $X_{\beta}:=[X_0,X_1]_{\beta}$ where $\beta\in (0,1)$. Here $[\cdot,\cdot]_{\beta}$ denotes the complex interpolation functor.

\emph{Step 1: For all $U=(v,\T),\, U'=(v',\T') \in X_1$,}
\begin{equation}
\label{eq:estimate_F_1}
\|F_1(\cdot,U)-F_1(\cdot,U')\|_{X_0}\lesssim 
\big(\|(v,\T)\|_{X_{3/4}}+\|(v',\T')\|_{X_{3/4}}\big)\|(v,\T)-(v',\T')\|_{X_{3/4}}.
\end{equation} 
Let us begin by noticing that $F_1$ is a bilinear map, and therefore to prove \eqref{eq:estimate_F_1} it is enough to consider the case $(v',\T')=0$. To this end, we note that
\begin{equation}
\label{eq:X_3_4_proof_local}
X_{3/4}\embed [X_0,H^2\times H^1]_{3/4}\embed H^{3/2}\times H^{1/2}.
\end{equation}
Moreover,
\begin{align*}
\Big\|
\begin{bmatrix}
\p[(v\cdot \nabla_{\h}) v]\\
\div_{\h}(v \T)
\end{bmatrix}
 \Big\|_{X_0}
 &\lesssim 
 \|(v\cdot \nabla_{\h}) v\|_{L^2}
 +\|v \T\|_{L^2}\\
 &\leq \|v\|_{L^6} (\|v\|_{W^{1,3}}+\|\T\|_{L^3})\\
& \lesssim \|v\|_{H^{3/2}}(\|v\|_{H^{3/2}}+\|\T\|_{H^{1/2}}),
\end{align*}
where in the last step we used Sobolev embeddings.
The remaining terms in $F_1$ can be estimated as in  \cite[Lemma 5.1]{HK16} for $p=2$. For the reader's convenience we include some details. Note that, for all $v\in H^{2}$, 
\begin{equation}
\label{eq:estimate_w_strong_weak}
\begin{aligned}
\|\w(v)\|_{L^{\infty}(-h,0;L^{4}(\Tor^2))}
&\lesssim \|\w(v)\|_{W^{1,2}(-h,0; L^{4}(\Tor^2))}\\
&\lesssim \|\div_{\h} v\|_{L^2(-h,0; L^4(\Tor^2))}\\
&\lesssim \|v\|_{L^2(-h,0; W^{1,4}(\Tor^2))}\\
&\lesssim \|v\|_{L^2(-h,0; H^{3/2}(\Tor^2))}\lesssim \|v\|_{H^{3/2}}
\end{aligned}
\end{equation}
and
\begin{align*}
\|\partial_3 v\|_{L^2(-h,0;L^4(\Tor^2))}
&\lesssim 
\|\partial_3 v\|_{L^2(-h,0;H^{1/2}(\Tor^2))}\\
& \lesssim
\| v\|_{L^2(-h,0;H^{3/2}(\Tor^2))} \lesssim
\|v\|_{H^{3/2}}.
\end{align*}
Analogously, one can check that 
$
\|\T\|_{L^2(-h,0;L^4(\Tor^2))}\lesssim 
\|\T\|_{H^{1/2}}
$
for all $\T\in H^1$.
Thus, using the previous estimates, we get
\begin{align*}
\Big\|
\begin{bmatrix}
\p[w(v)\partial_3 v]\\
\partial_{3}(w(v)\T)
\end{bmatrix}
 \Big\|_{X_0}
 &\lesssim
 \|w(v)\partial_3 v\|_{L^2}
 +\|w(v) \T\|_{L^2}\\
 &\leq \|w(v)\|_{L^{\infty}(-h,0;L^4(\Tor^2))} 
 \big[\|\partial_3 v\|_{L^{2}(-h,0;L^4(\Tor^2))}+\|\T\|_{L^{2}(-h,0;L^4(\Tor^2))}\big]\\
 &\lesssim \|v\|_{H^{3/2}}(\|v\|_{H^{3/2}}+\|\T\|_{H^{1/2}}).
\end{align*}
By \eqref{eq:X_3_4_proof_local}, the previous estimates yield \eqref{eq:estimate_F_1}.

\emph{Step 2: For all $U=(v,\T),\, U'=(v',\T') \in X_1$,}
\begin{align*}
\|F_2(\cdot,U)-F_2(\cdot,U')\|_{X_0}
&\lesssim 
(1+\|(v,\T)\|_{X_{3/5}}^4+\|(v',\T')\|_{X_{3/5}}^4)\|(v,\T)-(v',\T')\|_{X_{3/5}}\\
&+(1+\|(v,\T)\|_{X_{4/5}}^{2/3}+\|(v',\T')\|_{X_{4/5}}^{2/3})\|(v,\T)-(v',\T')\|_{X_{4/5}}.
\end{align*} 
By Assumption \ref{ass:well_posedness_primitive}\eqref{it:nonlinearities_strong_weak} and the Cauchy-Schwartz inequality, for all $(v,\T),(v',\T')\in X_1$,
$\|F_2(\cdot,U)-F_2(\cdot,U')\|_{X_0}$ can be estimated by 
\begin{align*}
&(1+\|v\|_{L^{10}}^4+ \|v'\|^4_{L^{10}})\|v-v'\|_{L^{10}}+
(1+\|z\|_{L^{10/3}}^{2/3}+\|\T'\|_{L^{10/3}}^{2/3})\|\T-\T'\|_{L^{10/3}}\\
&
+(1+\|v\|_{W^{1,10/3}}^{2/3}+\|v'\|_{W^{1,10/3}}^{2/3})\|v-v'\|_{W^{1,10/3}}^{2/3}\\
&\lesssim 
(1+\|v\|_{H^{6/5}}^4+ \|v'\|^4_{H^{6/5}})\|v-v'\|_{H^{6/5}}+
(1+\|z\|_{H^{3/5}}^{2/3}+\|\T'\|_{H^{3/5}}^{2/3})\|\T-\T'\|_{H^{3/5}}\\
&
+(1+\|v\|_{H^{8/5}}^{2/3}+\|v'\|_{H^{8/5}}^{2/3})\|v-v'\|_{H^{8/5}}^{2/3}
\end{align*}
where in the last equality we have used Sobolev embeddings.
To obtain the claimed estimate it is enough to note that 
\begin{equation*}
X_{3/5}\embed H^{6/5}\times H^{1/5} \ \ \text{ and }\ \  
X_{4/5}\embed H^{8/5}\times H^{3/5}.
\end{equation*}

\emph{Step 3: For all $U=(v,\T),\, U'=(v',\T') \in X_1$,}
\begin{align*}
\|F_3(\cdot,U)-F_3(\cdot,U')\|_{X_{0}}&+
\|G(\cdot,U)-G(\cdot,U')\|_{X_{1/2}(\ell^2)}\\
&\lesssim 
(1+\|(v,\T)\|_{X_{3/5}}^4+\|(v',\T')\|_{X_{3/5}}^4)\|(v,\T)-(v',\T')\|_{X_{3/5}}\\
&+(1+\|(v,\T)\|_{X_{2/3}}^{2}+\|(v',\T')\|_{X_{2/3}}^{2})\|(v,\T)-(v',\T')\|_{X_{2/3}}
\end{align*} 
\emph{where $X_{1/2}(\ell^2):=H^{1}(\ell^2)\times L^2(\ell^2)$.}

Note that, by Assumption \ref{ass:well_posedness_primitive}\eqref{it:well_posedness_primitive_phi_smoothness} and the H\"{o}lder inequality,
\begin{align*}
\|F_3(\cdot,U)-F_3(\cdot,U')\|_{X_{0}}
&\lesssim \max_{\ell,m\in \{1,2\}} 
\Big\|\sum_{n\geq 1}\hp_{n}^{\ell,m}\big(\q[\gvn(\cdot,v)-\gvn(\cdot,v')]\big)^m\Big\|_{L^2}\\
&\lesssim_{M} \|(\gvn(\cdot,v)-\gvn(\cdot,v'))_{n\geq 1}\|_{L^{6}(\ell^2)}.
\end{align*}
By the Sobolev embedding $H^1(\ell^2)\embed L^6(\ell^2)$, it is enough to prove the estimate for $
\|G(\cdot,U)-G(\cdot,U')\|_{X_{1/2}(\ell^2)}$ which will be given below.

Since $\|(\gtn(\cdot,v,\T)-\gtn(\cdot,v',\T'))_{\geq 1}\|_{L^2(\ell^2)}$ can be estimated as in Step 2, we only consider the $\gvn$-term. By the chain rule 
\begin{align*}
\partial_{j} (\gvn(\cdot,v))=\partial_{x_j} \gvn(\cdot,v)+\partial_y \gvn(\cdot,v)\partial_{j} v \quad \hbox{for all }j\in \{1,2,3\}.
\end{align*}
Arguing as in Step 2, one can check that to estimate 
$\|(\gvn(\cdot,v)-\gvn(\cdot,v'))_{\geq 1}\|_{H^{1}(\ell^2)}$ it is enough to bound the term $$\|(\partial_y \gvn(\cdot,v)\partial_{j} v -\partial_y \gvn(\cdot,v')\partial_{j} v' )_{n\geq 1}\|_{L^2(\ell^2)}=:I(v,v').$$ By Assumption \ref{ass:well_posedness_primitive}\eqref{it:nonlinearities_strong_weak} and the H\"{o}lder inequality, for all $v,v'\in H^2$, $I(v,v')$ is less or equal than
\begin{align*}
&\|(\partial_y \gvn(\cdot,v)(\partial_{j} v -\partial_{j} v' )_{n\geq 1}\|_{L^2(\ell^2)}
+\|(\partial_y \gvn(\cdot,v)-\partial_y \gvn(\cdot,v'))\partial_{j} v' )_{n\geq 1}\|_{L^2(\ell^2)}\\
&\lesssim(1+\||v-v'|^2\|_{L^{9}} )\|\partial_j v\|_{L^{18/7}}+(1+\||v'|^2\|_{L^9}) \|\partial_j v'\|_{L^{18/7}}\\
&\lesssim(1+\|v-v'\|_{H^{4/3}}^2 )\|v\|_{H^{4/3}}+(1+\|v'\|_{H^{4/3}}^2) \| v'\|_{H^{4/3}};
\end{align*}    	
where in the last estimate we have used Sobolev embeddings. Since $X_{2/3}\embed H^{4/3}\times H^{1/3}$, the claim of this step follows.

\emph{Step 4: Conclusion}. Due to Steps 1-3 and Proposition \ref{prop:SMR_2}, the existence of an $L^2$-maximal strong-weak solution to \eqref{eq:primitive_weak_strong} and Theorem~\ref{t:local_primitive}\eqref{it:path_regularity_strong_weak} follow from \cite[Theorem 4.8]{AV19_QSEE_1}, where we set $m_F=3$ and $m_G=2$ which correspond the to number of different terms on the right hand side when estimating  $F$ and $G$, respectively, as done in Steps 1-3.
Moreover, each of these five terms involves numbers $\rho_j$ describing the power in the estimates of the non-linearities, and $\beta_j, \varphi_j$  indicating the order in the estimates of the non-linearities in terms of interpolation spaces $X_{\beta_j}$ and $X_{\varphi_j}$, respectively. Here, by Steps 1-3 we can chose in \cite[Theorem 4.8]{AV19_QSEE_1}
\begin{align*}
\rho_1&=1, \rho_2=\rho_4=4, \rho_3=2/3, \rho_5=2, \hbox{ and } \\ \beta_1&=\varphi_1=3/4,  \varphi_2=\varphi_4=\beta_2=\beta_4=3/5, \varphi_3=\beta_3=4/5, \varphi_5=\beta_5=2/3,
\end{align*}
where one also uses that
$$
\rho_j\Big(\varphi_j-1+\frac{1}{2}\Big)+\beta_j=1\quad \text{ for all }\quad j\in \{1,\dots,5\}.
$$

Theorem~\ref{t:local_primitive}\eqref{it:blow_up_criterium_strong_weak} follows from \cite[Theorem 4.11]{AV19_QSEE_2} and Proposition \ref{prop:SMR_2}.
\end{proof}

\subsection{Proof of Theorem \ref{t:global_primitive}}
\label{ss:proof_global_strong_weak}
The key ingredient in the proof is the following energy estimate. Recall that $\norm_k$ has been defined in \eqref{eq:def_norm}.

\begin{proposition}[Energy estimate]
\label{prop:energy_estimate_primitive}
Let Assumption 
\ref{ass:global_primitive} be satisfied, and $T\in (0,\infty)$. Assume that $v_0\in L^{\infty}_{\F_0}(\O;\Hs^1)$ and $ \T_0\in L^{\infty}_{\F_0}(\O;L^2)$, and
let $((v,\T),\tau)$ be the $L^2$-maximal strong-weak solution to \eqref{eq:primitive_weak_strong}-\eqref{eq:boundary_conditions_strong_weak} provided by Theorem \ref{t:local_primitive}.  Then there exists a sequence of stopping times $(\mu_k)_{k\geq 1}$ with values in $[0,T]$ such that $\mu_k\leq \tau$ a.s.\ for all $k\geq 1$ and the following hold:
\begin{enumerate}[{\rm(1)}]
\item $\P(\mu_k=\tau\wedge T)\to 1$ as $k\to \infty$;
\item for each $k\geq 1$ there exists $C_{k,T}>0$ (possibly depending on $v_0,v$) such that 
$$
\E[ \norm_{1}(\mu_k;v)] + 
\E[\norm_0(\mu_k;\T) ]\leq C_{k,T}\Big(1+\E\|v_0\|_{H^1}^2+\E\|\T_0\|_{L^2}^2\Big).
$$
\end{enumerate}
\end{proposition}

Let us first show how Proposition \ref{prop:energy_estimate_primitive} yields Theorem \ref{t:global_primitive}.

\begin{proof}[Proof of Theorem \ref{t:global_primitive}]
By \cite[Proposition 4.13]{AV19_QSEE_2} it is enough to consider $v_0\in L^{\infty}(\O;\Hs^{1})$ and $\T_0\in L^{\infty}(\O;L^2)$. Fix $T\in (0,\infty)$ and let $(\mu_k)_{k\geq 1}$ be as in Proposition \ref{prop:energy_estimate_primitive}.
In particular,
\begin{equation}
\label{eq:a_priori_estimate_global_existence}
\norm_1(\mu_k;v)+\norm_0(\mu_k;\T)<\infty \quad  \text{ a.s.\ for all }k\geq 1 .
\end{equation}
Since $\lim_{k\to \infty}\P(\mu_k=\tau\wedge T)=1$,
\begin{align*}
\P(\tau<T)
&\ \  =\lim_{k\to \infty} \P(\{\tau<T\}\cap \{\mu_k=\tau\})\\
&\stackrel{\eqref{eq:a_priori_estimate_global_existence}}{=}
\lim_{k\to \infty}\P\Big(\{\tau<T\}\cap \{\mu_k=\tau\}\cap \Big\{ \norm_1(\tau;v)+\norm_0(\tau;\T)<\infty\Big\}\Big)\\
& \ \ \leq   \P\Big(\tau<T,\, \norm_1(\tau;v)+\norm_0(\tau;\T)<\infty\Big)=0
\end{align*}
where the last equality follows from Theorem \ref{t:local_primitive}\eqref{it:blow_up_criterium_strong_weak}. The arbitrariness of $T$ yields $\P(\tau<\infty)=0$. Hence $\tau=\infty$ a.s.\ as desired.
\end{proof}

The proof of Proposition \ref{prop:energy_estimate_primitive} will be divided into two parts. Firstly, in Subsection \ref{sss:L_2_estimate_I} we prove a standard $L^2$-energy estimate for $L^2$-maximal strong-weak solutions to \eqref{eq:primitive_weak_strong}-\eqref{eq:boundary_conditions_strong_weak} and in Subsection \ref{sss:proof_proposition_energy_estimate_I}  we prove Proposition \ref{prop:energy_estimate_primitive}.

\subsubsection{An $L^2$-energy estimate}
\label{sss:L_2_estimate_I}
The aim of this subsection is to prove the following result. Recall that $\norm_k$ is as in \eqref{eq:def_norm}.

\begin{lemma}[$L^2$-energy estimate]
\label{l:L_2_estimate_I}
Let the assumptions of Proposition \ref{prop:energy_estimate_primitive} be satisfied. Let $((v,\T),\tau)$ be the $L^2$-maximal strong-weak solution to \eqref{eq:primitive_weak_strong}-\eqref{eq:boundary_conditions_strong_weak} provided by Theorem \ref{t:local_primitive}. Then for each $T\in (0,\infty)$ there exists $c_T>0$ independent of $v_0,v,\T_0,\T$ such that 
\begin{equation}
\label{eq:L_2_estimate_weak_strong}
\E[ \norm_{0}(\tau\wedge T;v)] + 
\E[\norm_0(\tau\wedge T;\T) ] 
 \leq c_T
\Big(1+\E\|v_0\|_{L^2}^2+\E\|\T_0\|_{L^2}^2\Big).
\end{equation}
\end{lemma}

\begin{proof}
For the reader's convenience we split the proof into several steps. Below $T\in (0,\infty)$ is fixed and $\l \cdot,\cdot\r$ denotes the duality pairing for $H^{1}$ and $(H^1)^*$. Recall that $L^2\embed (H^1)^*$ and the embedding is given by $\l \varphi,f\r:=\int_{\Dom}f\varphi\,dx$ for $\varphi\in H^1$.

For each $k\geq 1$, let us set
\begin{equation}
\label{eq:def_tau_n_energy_L_2_estimate}
\begin{aligned}
\tau_k:=\inf\big\{t\in [0,\tau)&\colon\|v(t)\|_{H^1}+\| v\|_{L^2(0,t;H^2)} \\
&+\|\T(t)\|_{L^2}+\|\nabla \T\|_{L^2(0,t;L^2)}\geq k\big\}\wedge T,
\end{aligned}
\end{equation}
where $\inf\emptyset:=\tau$.
By progressive measurability of $(v,\T)$ (see Definition \ref{def:sol_strong_weak}) and Theorem \ref{t:local_primitive}\eqref{it:path_regularity_strong_weak}, for each $k\geq 1$, $\tau_k$ is a stopping time and $\lim_{k\to \infty}\P(\tau_k=\tau)=1$. Therefore, by Fatou's lemma, it is enough to prove \eqref{eq:L_2_estimate_weak_strong} for $\tau$ replaced by $\tau_k$ provided $c_T$ is
independent of $k\geq 1$. Note that 
\begin{align*}
v|_{[ 0,\tau_k]\times \O}\in C([0,\tau_k];\Hs^1)\cap L^2(0,\tau_k;\Hs^2) \hbox{ and } \T|_{[ 0,\tau_k]\times \O}\in C([0,\tau_k];L^2)\cap L^2(0,\tau_k;H^1)
\end{align*}
uniformly in $\O$. In particular, all the integrals appearing below are finite.
%

By Gronwall's and Fatou's lemma, it is enough to prove the existence of $c_T$ independent of $k\geq 1$ such that, for all $t\in [0,T]$,
\begin{equation}
\begin{aligned}
\label{eq:L_2_estimate_gronwall_claim}
&\E\Big[\sup_{s\in [0,\tau_k\wedge t)} (\|v(s)\|_{L^2}^2+\|\T(s)\|_{L^2}^2)\Big]
+ \E \int_0^{\tau_k\wedge t} (\|\nabla  v(s)\|_{L^2}^2+\|\nabla \T(s)\|_{L^2}^2)\,ds \\
&\qquad
\leq c_T
\Big(1+\E\|v_0\|_{L^2}^2+\E\|\T_0\|_{L^2}^2\Big)+ c_T \E\int_0^{\tau_k\wedge t} (\|v(s)\|^2_{L^2}+\|\T(s)\|^2_{L^2})\,ds.
\end{aligned}
\end{equation}
To shorten the notation, in the following steps, we set $\sigma:=\tau_k$. 

\emph{Step 1: $L^{\infty}(L^2)$- and $L^2(H^1)$-estimate for $\T$, see \eqref{eq:energy_estimate_T_L2} below}. 
The idea is to apply It\^{o}'s formula to $\T\mapsto \|\T\|_{L^2}^2 $ and use an argument similar to the one used in the proof of Proposition \ref{prop:SMR_2}. Recall that, by integration by parts, one has for all $v\in \Hs_{\n}^2$ and $\T\in H^1$ (cf.\ \eqref{eq:T_map_definition} and the text below it)
\begin{equation}
\label{eq:cancellation_L_2_estimate}
\l \div_{\h}(v\, \T)+\partial_3 (w(v) \,\T), \T\r =\int_{\Dom} \T(v\cdot \nabla_{\h} \T+w(v)\partial_3 \T) dx = 0
\end{equation} 
since $\div_{\h}v+\partial_3 [w(v)]=0$ a.e.\ on $[ 0,\sigma]\times \O$. 

Reasoning as in Step 1 of Proposition \ref{prop:SMR_2}, we set $\T^{\sigma}:=\T(\cdot\wedge \sigma)$ and applying It\^{o}'s formula and \eqref{eq:cancellation_L_2_estimate} we get, a.s.\ for all $t\in [0,T]$,
\begin{align*}
\|\T^{\sigma}(t)\|_{L^2}^2
& +2\int_0^{t}\one_{[ 0,\sigma]}\|\nabla \T(s)\|_{L^2}^2\,ds-\|\T_0\|_{L^2}^2\\
&
=2\int_0^{t}\one_{[ 0,\sigma]} \int_{\Dom}  \ft(v,\T) \T\,dx ds +
\int_0^{t}\one_{[ 0,\sigma] } 
\sum_{n\geq 1}\|(\psi_n\cdot \nabla) \T+\gtn(v)\|_{L^2}^2\,ds\\
& 
+2\sum_{n\geq 1} \int_0^t \one_{[0,\sigma] } \Big((\psi_n\cdot \nabla) v+\gtn(v,\T),\T\Big)_{L^2}\,d\beta_s^n,
\end{align*}
where, for brevity, we set $\ft(v,\T)=\ft(\cdot,v,\T,\nabla v)$, $\gtn(v,\T)=\gtn(\cdot,v,\T,\nabla v)$ and $(\cdot,\cdot)_{L^2}$ denotes the scalar product in $L^2=L^2(\Dom)$.

By Assumption \ref{ass:global_primitive}\eqref{it:sublinearity_Gforce}, we have a.e.\ on  $[ 0,\sigma]\times \O$  
\begin{align*}
\Big|\int_{\Dom} \ft(v,\T) \T\,dx\Big| &\leq C
(\y^2+\|v\|_{L^2}^2+\|\nabla v\|_{L^2}^2+\|\T\|_{L^2}^2+ \|\T\|_{L^2}^2),\\
\|(\gtn(v,\T))_{n\geq 1}\|_{L^2(\ell^2)}^2 &\leq C(\y^2+\|v\|_{L^2}^2+\|\nabla v\|_{L^2}^2+\|\T\|_{L^2}^2).
\end{align*}
By using the argument in Step 1 of Proposition \ref{prop:SMR_2}, one can check that the above estimates ensure that, for all $t\in [0,T]$,
\begin{equation}
\label{eq:energy_estimate_T_L2}
\begin{aligned}
\E \Big[\sup_{s\in [0,\sigma\wedge t]}\|\T(s)\|_{L^2}^2 \Big]
&+\E \int_0^{\sigma\wedge t} \|\nabla \T(s)\|_{L^2}^2 \,ds
\leq C_{T,\T}\E\|\T_0\|_{L^2}^2\\
&+ C_{T,\T}\E\int_0^{\sigma\wedge t} \big(\|\T(s)\|_{L^2}^2 + \|v(s)\|_{L^2}^2+  \|\nabla v(s)\|_{L^2}^2\big) ds
\end{aligned}
\end{equation}
where $C_{T,\T}>0$ is independent of $\T,\T_0,v,v_0$ and $k\geq 1$.

\emph{Step 2: $L^{\infty}(L^2)$- and $L^2(H^1)$-estimate for $v$, see \eqref{eq:energy_estimate_v_L2}}. As in Step 1, the idea is to apply It\^{o}'s formula to $v\mapsto \|v\|_{L^2}^2 $ and use an argument similar to the one used in Step 1 of Proposition \ref{prop:SMR_2}. 
As in \eqref{eq:cancellation_L_2_estimate} we have the following cancellation
\begin{equation}
\label{eq:cancellation_L_2_estimate_v}
\int_{\Dom} (v\cdot \nabla_{\h} v+w(v)\partial_3 v)\cdot v\,dx = 0, \ \ \ \text{ for all }
v\in \Hs^2_{\n},
\end{equation} 
since $\div_{\h} v+\partial_3 w(v)=0$ a.e.\ on $[0,\sigma]\times \O$.
 Reasoning as in Step 1, we set $v^{\sigma}:=v(\cdot\wedge \sigma)$ and apply It\^{o}'s formula and \eqref{eq:cancellation_L_2_estimate_v} to get, a.s.\ for all $t\in [0,T]$,
\begin{align*}
\|v^{\sigma}(t)\|_{L^2}^2 
&+2\int_0^{t}\one_{[0,\sigma]} \|\nabla v(s)\|_{L^2}^2\,ds-\|v_0\|_{L^2}^2\\
& 
\leq 2\int_0^{t}\one_{[ 0,\sigma]} \int_{\Dom}  \Big(\fv(v,\T,\nabla v)+ \Lp (\cdot,v) -\op_{\kone} \T\Big)\cdot v\, dxds \\
&+
\int_0^{t}\one_{[ 0,\sigma] } 
\sum_{n\geq 1}\|(\phi_n\cdot \nabla) v+\gtn(v)\|_{L^2}^2\,ds\\
&
+2 \sum_{n\geq 1}\int_0^t \one_{[ 0,\sigma] } \Big((\phi_n\cdot \nabla) v+\gtn(v),v\Big)_{L^2}\,d\beta_s^n,
\end{align*}
where $\op_{\kone}$ is as in \eqref{eq:integral_operator_temperature}, and 
where we use that the hydrostatic Helmholtz projection $\p$ is orthogonal on $L^2(\Dom;\R^2)$, in particular  $\|\p\|_{\calL(L^2)}=1$.

Integrating by parts, we have a.e.\ on $[0,\sigma]\times \O$ for all $\varepsilon\in (0,1)$ and some $C_{\varepsilon}>0$
\begin{align*}
\Big| \int_{\Dom} \op_{\kone} \T\cdot v\, dx\Big| 
&=
\Big|\int_{\Dom} \Big[\int_{-h}^{x_3}(\kone(\cdot,x_{\h},\zeta) \T(\cdot,x_{\h},\zeta))d\zeta\Big] \div_{\h} v\, dx\Big|\\
&\leq \varepsilon\|\nabla v\|_{L^2}^2 + C_{\varepsilon} \|\T\|_{L^2}^2.
\end{align*}

Next we consider the $\Lp (\cdot,v)$-part. To this end, recall that 
$$
\nabla\wt{P}_n=\q[(\phi_n\cdot \nabla) v+ \gvn(\cdot,v)].
$$
Let $\delta>0$ be as in 
Assumption \ref{ass:well_posedness_primitive}\eqref{it:well_posedness_primitive_phi_smoothness} and let $\varrho\in (1,6)$ such that $\frac{1}{\varrho}+\frac{1}{3+\delta}=\frac{1}{2}$. 
Note that, by \eqref{eq:def_P_gamma} we have, a.e.\  on $[0,\sigma]\times \O$,
\begin{align*}
\Big| \int_{\Dom}  \Lp (\cdot,v)\cdot v\, dx\Big|
&\leq  \sum_{m,\ell= 1}^2 \sum_{n\geq 1} \int_{\Dom} |\partial_{m}\wt{P}_n| |\hp_n^{\ell,m}| | v |\,dx\\
&\leq \sum_{m,\ell= 1}^2 \|(\partial_{m}\wt{P}_n)_{n\geq 1}\|_{L^2(\ell^2)} 
\Big\| \|(\hp_n^{\ell,m})_{n\geq 1} \|_{\ell^2} |v|\Big\|_{L^2}\\
&\stackrel{(i)}{\lesssim}_{M} \max_{m\in \{1,2\}} 
 \|(\partial_{m}\wt{P}_n)_{n\geq 1}\|_{L^2}\|v\|_{L^{\varrho}}\\
&
\stackrel{(ii)}{\lesssim} 
(\y+\|v \|_{L^2}+\|\nabla v\|_{L^2})\|v\|_{L^{\varrho}}
\stackrel{(iii)}{\leq}
 \varepsilon \|\nabla v\|_{L^2}^2 + C_{\varepsilon} (\|v\|_{L^2}^2+ \y),
\end{align*}
where in $(i)$ we applied the H\"{o}lder inequality, in $(ii)$ the boundeedness of $\q$ and in $(iii)$ the Young's and standard interpolation inequalities.

The remaining terms can be estimated as in Step 1. Thus, choosing $\varepsilon$ small enough, one can check that Assumption \ref{ass:global_primitive}\eqref{it:sublinearity_Gforce} yields, for all $t\in [0,T]$,
\begin{equation}
\label{eq:energy_estimate_v_L2}
\begin{aligned}
\E \Big[\sup_{s\in [0,\sigma\wedge t]}\|v(s)\|_{L^2}^2 \Big]
&+\E \int_0^{\sigma\wedge t} \|\nabla v(s)\|_{L^2}^2 \,ds\\
&\leq C_{T,v}\Big(\E\|v_0\|_{L^2}^2
 + \E\int_0^{\sigma\wedge t} (\|\T(s)\|_{L^2}^2 + \|v(s)\|_{L^2}^2)\, ds\Big).
\end{aligned}
\end{equation}
where $C_{T,v}>0$ is independent of $\T,\T_0,v,v_0$ and $k\geq 1$.

\emph{Step 3: Proof of \eqref{eq:L_2_estimate_gronwall_claim}}. 
Let $C_{T,\T}$ and $C_{T,v}$ be as in \eqref{eq:energy_estimate_T_L2} and \eqref{eq:energy_estimate_v_L2}, respectively.
Without loss of generality we may assume that $C_{T,\T}, C_{T,v}\geq 1$.
The claimed inequality follows by noticing that 
$$
\frac{\eqref{eq:energy_estimate_T_L2}}{2 C_{T,\T}}+ 
\eqref{eq:energy_estimate_v_L2}\quad \Longrightarrow \quad
\eqref{eq:L_2_estimate_gronwall_claim}.
$$
More precisely, the above means that \eqref{eq:L_2_estimate_gronwall_claim} follows by multiplying \eqref{eq:energy_estimate_T_L2} by $(2C_{T,\T})^{-1}$ and then adding the estimate with \eqref{eq:energy_estimate_v_L2}.
On the RHS of the resulting estimate the term $\frac{1}{2} 
 \E\int_0^{\sigma\wedge t}  \|\nabla v(s)\|_{L^2}^2 ds$ appears and can be adsorbed into the LHS since $\sigma=\tau_k$ and therefore $\E\int_0^{\sigma\wedge t}  \|\nabla v(s)\|_{L^2}^2 ds\leq k$ a.s.\ by \eqref{eq:def_tau_n_energy_L_2_estimate}.
\end{proof}

\subsubsection{Higher order energy estimates and proof of Proposition \ref{prop:energy_estimate_primitive}}
\label{sss:proof_proposition_energy_estimate_I}
Through this subsection, we assume that the assumptions of Proposition \ref{prop:energy_estimate_primitive} holds, and in particular $T\in (0,\infty)$ is fixed. 
%
%

Let $((v,\T),\tau)$ be the $L^2$-maximal strong-weak solution to \eqref{eq:primitive_weak_strong}-\eqref{eq:boundary_conditions_strong_weak} provided by Theorem \ref{t:local_primitive}. For each $k\geq 1$ we set
\begin{equation}
\label{eq:def_ell_n}
\begin{aligned}
\ell_k:=\inf\big\{t\in [0,\tau)&\,:\,\|v(t)\|_{L^2}+\| v\|_{L^2(0,t;H^1)} \\
&+\|\T(t)\|_{L^2}+\| \T\|_{L^2(0,t;H^1)} +\|\y\|_{L^2(0,t:L^2)}\geq k\big\}\wedge T, 
\end{aligned}
\end{equation}
where $\inf\emptyset:=\tau$ and $\y$ is as in Assumption \ref{ass:global_primitive}\eqref{it:sublinearity_Gforce}.

By Lemma \ref{l:L_2_estimate_I} and $\y\in L^2_{\loc}([0,\infty);L^2)$ a.s.\ we have 
$
\lim_{k\to \infty}
\P(\ell_k=\tau)=1.
$
Note that 
\begin{equation}
\label{eq:boundedness_on_xi_n}
\norm_0(\ell_k;\T)+\norm_0(\ell_k;v)\leq k,\quad \text{ a.s.\ for all }k\geq 1,
\end{equation}
where $\norm_0$ is as in \eqref{eq:def_norm}.
To prove Proposition \ref{prop:energy_estimate_primitive}, it remains to find stopping times $(\mu_k)_{k\geq 1}$ such that, a.s.\ for all $k\geq 1$, one has  $\mu_k\leq \ell_k$, $\lim_{k\to \infty} \P(\mu_k=\ell_k)=1$ and 
\begin{equation}
\label{eq:claim_in_terms_v}
\E[ \norm_{1}(\mu_k;v)] \leq C_{k,T}\big(1+\E\|v_0\|_{H^1}^2\big),
\end{equation}
where $(C_{k,T})_{k\geq 1}$ are constants possibly depending on $v_0,v$ and $k\geq 1$.

Let us recall that, for all $k\geq 1$, $(v,\ell_k)$ is a $L^2$-local solution to 
\begin{equation}
\label{eq:primitive_v_proof_global}
\begin{cases}
\displaystyle{d v -\Delta v\, dt=\Big(\p\Big[-(v\cdot \nabla_{\h})v-\w(v)\partial_3 v+ \fvt+ \Lpp \Big] \Big)dt} \\
\qquad \qquad\qquad \qquad\qquad
\displaystyle{
+\sum_{n\geq 1}\p\Big[(\phi_{n}\cdot\nabla) v + \gvtn \Big] d{\beta}_t^n,} &\text{on }\Dom, \\
\partial_3 v(\cdot,-h)=\partial_3 v(\cdot,0)=0, &\text{on }\Tor^2,\\
v(\cdot,0)=v_0,& \text{on }\Dom,
\end{cases}
\end{equation}
where, for notational convenience, we set on $[0,\tau)\times \O$, 
\begin{equation}
\begin{aligned}
\label{eq:def_f_g_proof_global} 
\gvtn:=\gvn(\cdot,v), \text{ for }n\geq 1, \qquad
\Lpp :=\sum_{n\geq 1} \sum_{m=1}^2\hp^{\cdot,m}_n \big(\q[(\phi_n\cdot \nabla) v]\big)^m,&\\
\fvt
:=\nabla_{\h} \int_{-h}^{\cdot} (\kone(\cdot,\zeta) \T(\cdot,\zeta))\,d\zeta+ \fv(\cdot,v,\T,\nabla v)+ 
\sum_{n\geq 1} \sum_{m=1}^2\hp^{\cdot,m}_n \big(\q [\gvtn]\big)^m,&\\
\end{aligned}
\end{equation} 
where $\q$ is as in Subsection \ref{ss:set_up}. Finally, we set
\begin{equation}
\label{eq:def_N_v_T}
N_{v,\T}(t):=\|\fvt\|_{L^2(0,t \wedge \tau;L^2)}^2+
\|(\gvtn)_{n\geq 1}\|_{L^2(0,t \wedge \tau;H^1(\ell^2))}^2\text{ a.s.\ for all }t\in [0,\tau).
\end{equation}
 
Let us first show that $N_{v,\T}$ is bounded on the stochastic interval $[0,\ell_k]\times \O$ for all $k\geq 1$. To this end, note that, by Assumption \ref{ass:global_primitive}\eqref{it:sublinearity_Gforce} and \eqref{eq:boundedness_on_xi_n}, we have
$
\|(\gvtn)_{n\geq 1}\|_{L^2(0,\ell_k;H^1(\ell^2))}^2\leq C_{g,k}
$ a.s.\ for some $C_{g,k}>0$ independent of $v_0,v$. 
The previous estimate and Assumption \ref{ass:well_posedness_primitive}\eqref{it:well_posedness_primitive_phi_smoothness} yield a.s.
$$
\Big\|\sum_{n\geq 1} \sum_{m=1}^2 \hp^{\cdot,m}_n (\q [\gvtn])^m\Big\|_{L^2(0,\ell_k;L^2)}^2
\stackrel{(i)}{\leq}C_h M \|(\gvtn)_{n\geq 1 }\|_{L^2(0,\ell_k;L^6(\ell^2))}^2
\stackrel{(ii)}{\leq} C_h M k,
$$
where in $(i)$ we used the H\"{o}lder inequality, in $(ii)$ the embedding $H^1(\ell^2)\embed L^6(\ell^2)$
and \eqref{eq:boundedness_on_xi_n}.
Thus the previous estimates, Assumption \ref{ass:global_primitive}\eqref{it:sublinearity_Gforce} and \eqref{eq:def_ell_n} ensure that, 
for some $C_k$ independent of $v_0,v$, 
\begin{equation}
\label{eq:bound_f_g_expectations}
N_{v,\T}(\ell_k)=
\|\fvt\|_{L^2(0,\ell_k;L^2)}^2+
\|(\gvtn)_{n\geq 1}\|_{L^2(0,\ell_k;H^1(\ell^2))}^2\leq C_k \text{ a.s.\ }
\end{equation}

Following \cite{CT07}, we derive from \eqref{eq:primitive_weak_strong} a coupled system of SPDEs for the unknowns
\begin{equation}
\overline{v}(t,x_{\h}):=\frac{1}{h}\int_{-h}^0 v(t,x_{\h},\zeta)d\zeta,
 \ \  \text{ and }\ \  
\wt{v}(t,x):=v(t,x)-\overline{v}(t,x_{\h}),
\end{equation}
where $x=(x_{\h},x_3)\in \Tor^2 \times (-h,0)=\Dom$ and $t\in\R_+$. 

Recall that $\pr$ denotes the Helmholtz projection on $L^2(\Tor^2;\R^2)$ which acts on the horizontal variable $x_{\h}\in \Tor^2$ where $x=(x_{\h},x_3)\in \Dom$, see Subsection \ref{ss:set_up}. Since $\overline{\p v}= \pr \overline{v}$, applying the vertical avarage in \eqref{eq:primitive_weak_strong} and using Assumption \ref{ass:global_primitive}\eqref{it:independence_z_variable}, for all $k\geq 1$, $(\overline{v}, \ell_k)$ is an $L^2$-local strong solution to 
\begin{equation}
\label{eq:primitive_bar}
\begin{cases}
\displaystyle{d \overline{v} -\Delta_{\h} \overline{v}\, dt=\Big(\pr\Big[-(\overline{v}\cdot \nabla_{\h})\overline{v}- \force(\wt{v}) + \overline{\fvt}+ \Lpp\Big]\Big)dt} \\
\qquad \qquad  \qquad \qquad
\displaystyle{
+\sum_{n\geq 1}\pr \Big[(\phi_{n,\h}\cdot\nabla_{\h}) \overline{v}+\overline{\phi^3_n \partial_3 v }  +\overline{\gvtn}\Big] d{\beta}_t^n,} &\text{on }\Tor^2, \\
\displaystyle{\force(\wt{v})=\frac{1}{h}\int_{-h}^0 \Big[(\wt{v}\cdot \nabla_{\h}) \wt{v}+\wt{v} (\div_{\h} \wt{v})\Big]d\zeta,}& \text{on }\Tor^2, \\
\overline{v}(\cdot,0)=\overline{v_0},& \text{on }\Tor^2,
\end{cases}
\end{equation}
where $\phi_{n,\h}:=(\phi^1_n,\phi^2_n)$. Here we also used that
\begin{align*}
\overline{(v\cdot \nabla_{\h})v+ \w(v)\partial_3 v} = 
(\overline{v }\cdot\nabla_{\h} ) \overline{v}+
\overline{(\wt{v}\cdot\nabla_{\h}) \wt{v}+(\div_{\h} \wt{v})\, \wt{v} }
\end{align*}
which follows from  $\overline{\wt{v}}=0$, \eqref{eq:def_w_v} and an integration by parts, and by Assumption \ref{ass:global_primitive}\eqref{it:independence_z_variable},
\begin{align}
\label{eq:Lpp_independent_of_x_3}
\overline{\Lpp}
&=\sum_{n\geq 1}\sum_{m=1}^2 \overline{\hp_n^{\cdot,m} \Big(\q_{\h}\Big[\overline{(\phi_n\cdot \nabla) v}\Big]\Big)^m}\\
\nonumber
&=
\sum_{n\geq 1}\sum_{m=1}^2  \hp_n^{\cdot,m} \Big(\q_{\h}\Big[\overline{(\phi_n\cdot \nabla) v}\Big]\Big)^m=\Lpp.
\end{align}
Let us also note that the first equation in \eqref{eq:primitive_bar} and $v_0\in \Hs^1$ imply
$\div_{\h} \overline{v}=0$. 

%
Here, by $L^2$-local strong solution to \eqref{eq:primitive_bar} we understand that $(\overline{v},\ell_k)$ solves \eqref{eq:primitive_bar} in its natural integral form, cf.\ Definition \ref{def:sol_strong_weak}.

Analogously, noticing that $\p z-\p_{\h}\overline{z}=z-\overline{z} $ for all $z\in L^2$, one can readily check that $(\wt{v}, \ell_k)$ is an $L^2$-local strong solution to
\begin{equation}
\label{eq:primitive_tilde}
\begin{cases}
\displaystyle{d \wt{v} -\Delta \wt{v}\, dt=\Big[-(\wt{v}\cdot \nabla_{\h})\wt{v}+\forcetwo(\wt{v},\overline{v}) + \wt{\fvt} \,\Big]dt} \\
\qquad \qquad\qquad \qquad\qquad
\displaystyle{+\sum_{n\geq 1}\Big[(\phi_{n}\cdot\nabla) \wt{v}-\overline{\phi^3_n \partial_3 v } +\wt{\gvtn}\,\Big] d{\beta}_t^n, } &\text{ on }\Dom,\\
\displaystyle{\forcetwo(\wt{v},\overline{v})=- \w(v) \partial_3 \wt{v} -(\wt{v}\cdot \nabla_{\h}) \overline{v} -( \overline{v}\cdot \nabla_{\h} )\wt{v} +\force(\wt{v}),} &\text{on }\Dom,\\
\partial_3 \wt{v}(\cdot,-h)=\partial_3 \wt{v}(\cdot,0)=0,&\text{on }\ \Tor^2,\\
\wt{v}(\cdot,0)=\wt{v}_0:=v_0-\overline{v_0},&\text{on }\Dom.
\end{cases}
\end{equation}
Here we used that $\partial_3 v=\partial_3 \wt{v}$. Note also that $w(v)=w(\wt{v})$ since $\div_{\h} \overline{v}=0$. This fact will be used frequently in the following.

With this preparation we can prove an intermediate estimate which is the key ingredient in the proof of Proposition \ref{prop:energy_estimate_primitive}. 

\begin{lemma}[An intermediate estimate]
\label{l:intermediate_estimate}
Let the assumptions of Proposition \ref{prop:energy_estimate_primitive} be satisfied. Let $\ell_k$ be as in \eqref{eq:def_ell_n} and let $((v,\T),\tau)$ be the $L^2$-maximal strong-weak solution to \eqref{eq:primitive_weak_strong}-\eqref{eq:boundary_conditions_strong_weak}. 
Then there exists a sequence of constants $(C_k)_{k\geq 1}$ such that
\begin{equation}
\label{eq:X_Y_estimate}
\E\Big[\sup_{t\in [0,\ell_k)} X_t\Big]
+\E\int_0^{\ell_k} Y_t \,dt
\leq C_k \big(1+\E\|v_0\|_{H^1}^4\big)
\end{equation}
where, for each $t\in [0,\tau)$,
\begin{equation}
\label{eq:def_X_Y}
\begin{aligned}
X_{t}&:=\| \overline{v}(t)\|_{H^1(\Tor^2)}^2+ \|\partial_3 v(t)\|_{L^2}^2 + \|\wt{v}(t)\|_{L^4}^4,\\
Y_t&:=\|\overline{v}(t)\|_{H^2(\Tor^2)}^2+ \|\nabla \partial_3 v(t)\|_{L^2}^2 + \Big\||\wt{v}(t)| |\nabla \wt{v}(t)|\Big\|^2_{L^2}.
\end{aligned}
\end{equation} 
 \end{lemma}
 
In \eqref{eq:def_X_Y}, with a slight abuse of notation, we wrote $H^k(\Tor^2)$ instead of $H^k(\Tor^2;\R^2)$.
We will use the same notation also below if no confusion seems likely.

\begin{proof}[Proof of Lemma \ref{l:intermediate_estimate}]
We begin by collecting some useful facts.
By Definition \ref{def:sol_strong_weak}, for each $j\geq 1$, the following is a stopping time
\begin{equation}
\label{eq:tau_j_a_priori_estimates_strong_weak}
\tau_j:=\inf\big\{t\in [0,\tau) \,:\, \|v\|_{L^2(0,t;H^2)}+ \|v(t)\|_{H^1}\geq j\big\}\wedge T, \text{ where }
\inf\emptyset :=\tau.
\end{equation}
Note that, by Definition \ref{def:sol_strong_weak} and the definition of the $\tau_j$'s, 
\begin{equation}
\label{eq:properties_tau_j_localizing_sequence}
\|v\|_{L^2(0,\tau_j;H^2)}+\sup_{s\in [0,\tau_j]}\|v(s)\|_{H^1}\leq j \quad \text{ and }\lim_{j\to \infty}\tau_j=\tau \text{ a.s. }
\end{equation}
To prove \eqref{eq:X_Y_estimate}, it is enough to show that for each $k\geq 1$ there exists $C_{0,k}>0$ independent of $j$ and $v,v_0$ such that, for each $j\geq 1$ and any stopping times $0\leq \eta\leq \xi\leq  \tau_j\wedge \ell_k$,
\begin{equation}
\label{eq:energy_estimate_proof_claim}
\begin{aligned}
\E\Big[\sup_{s\in [\eta,\xi]} X_s \Big]
&+ \E\int_{0}^{T}\one_{[\eta,\xi]} Y_s \,ds 
\leq C_{0,k}+C_{0,k}\E [X_{\eta}]\\
&+C_{0,k} \E\Big[\int_{0}^T\one_{[\eta,\xi]} (1+\|v\|_{L^2}^2 )(1+ N_{v,\T}+\|v\|_{H^1}^2) X_s\, ds\Big],
\end{aligned}
\end{equation}
where $N_{v,\T}$ is as in \eqref{eq:def_N_v_T}.
Recall the Sobolev embedding $H^1\embed L^6$. Thus all the integrals in \eqref{eq:energy_estimate_proof_claim} are finite due to $\xi\leq \tau_j$ and \eqref{eq:properties_tau_j_localizing_sequence}. 

Let us first prove the sufficiency of \eqref{eq:energy_estimate_proof_claim}. Due to \eqref{eq:boundedness_on_xi_n}, for each $j,k\geq 1$,
\begin{align*}
&\int_0^{\ell_k\wedge \tau_j }(1+\|v\|_{L^2}^2 )(1+N_{v,\T}+\|v\|_{H^1}^2)\,ds\\
&\leq 
\Big(1+\sup_{s\in [0,\ell_k] }\|v(s)\|_{L^2}^2 \Big)
\int_0^{\ell_k}(1+N_{v,\T}+\|v\|_{H^1}^2)\,ds \leq  (k+T)(k+T+C_k).
\end{align*}
Therefore the stochastic Gronwall's lemma in \cite[Lemma 5.3]{GHZ_09} with $\tau=\ell_k\wedge \tau_j$ applies to \eqref{eq:energy_estimate_proof_claim}, and it ensures the existence of $C(k,T,C_{0,k})>0$ independent of $j,v,v_0$ such that  
\begin{equation}
\label{eq:conclusion_intermediate_estimate_proof}
\E\Big[\sup_{s\in [0,\ell_k\wedge \tau_j]} X_s \Big]
+ \E\int_{0}^{t}\one_{[0,\ell_k\wedge \tau_j]} Y_s \,ds \leq C\big(1+\E\|v_0\|_{H^1}^4\big),
\end{equation}
where we also used that $\E [X_0]\lesssim 1+\E\|v_0\|_{H^1}^4$ by Sobolev embeddings. 
Recall that $\ell_k\leq \tau$ a.s.\ for all $k\geq 1$. Thus the claimed estimate follows by taking $j\to \infty$ in \eqref{eq:conclusion_intermediate_estimate_proof} using that $C$ in \eqref{eq:conclusion_intermediate_estimate_proof} is independent of $j\geq 1$ and the second in \eqref{eq:properties_tau_j_localizing_sequence}.

The proof of \eqref{eq:energy_estimate_proof_claim} will be divided into several steps. The argument is an extension of the one in \cite[Subsection 1.4.3]{HH20_fluids_pressure} for the deterministic case. 
Recall that $\eta,\xi$ are stopping times such that $0\leq \eta\leq \xi\leq \tau_j\wedge \ell_k$ a.s.\ for some $j,k\geq 1$.

\emph{Step 1: $L^{\infty}_t(H^1_x)$- and $L^2_t(H^2_x)$-estimates for $\overline{v}$, see \eqref{eq:estimate_Step_1} below}. 
By repeating the argument of Step 2 of Proposition \ref{prop:SMR_2} where one uses the parabolicity condition in Remark \ref{r:global_primitive}\eqref{it:ellipticity_overline_v}, one can show that 
\begin{equation}
\label{eq:smr_bar_operator}
\Big(-\Delta_{\h}, \pr[(\phi_{n,\h}\cdot \nabla_{\h}) ])_{n\geq 1} \Big)\in \mathcal{SMR}^{\bullet}_2(T),
\end{equation}
with $X_0=\Ls_{\h}^2(\Tor^2)$ and $X_1=\Hs_{\h}^2(\Tor^2)$, where $\Ls_{\h}^2(\Tor^2)$ is the space of divergence free vector field on $\Tor^2$ and $\Hs_{\h}^2(\Tor^2):=\Ls_{\h}^2(\Tor^2)\cap H^2(\Tor^2;\R^2)$. Here $\mathcal{SMR}^{\bullet}_2(T)$ denotes for the set of couples having stochastic maximal $L^2$-regularity on $(X_0,X_1)$, cf.\ Remark \ref{r:SMR_2}\eqref{it:choice_AB_strong_weak} and \cite[Definitions 3.5-3.6]{AV19_QSEE_1}.
Since $(\overline{v},\xi)$ is a $L^2$-strong solution to \eqref{eq:primitive_bar}, by \eqref{eq:smr_bar_operator} there exists $\overline{C}>0$ independent of $v,v_0,\eta,\xi,j,k$ such that
\begin{equation}
\label{eq:estimate_overline_v_smr_2}
\E\Big[\sup_{s\in [\eta, \xi]} \|\overline{v}(s)\|_{H^1(\Tor^2)}^2\Big]
+\E \int_{\eta}^{\xi}\| \overline{v}\|_{H^2(\Tor^2)}^2\,ds
\leq \overline{C} \Big(\E\|\overline{v}(\eta)\|_{H^1}^2+ \sum_{j=1}^4 \Ib_j\Big),
\end{equation}
where 
\begin{align*}
\Ib_1&:= 
\E\|\overline{\fvt}\|_{L^2(\eta,\xi;L^2(\Tor^2))}^2 + 
\E \|(\overline{\gvtn})_{n\geq 1}\|_{L^2(\eta,\xi;H^1(\Tor^2;\ell^2))}^2 ,\\
\Ib_2&:=\E\|\force (\wt{v})\|_{L^2(\eta,\xi;L^2(\Tor^2))}^2 ,\\
\Ib_3&:=\E\|(\overline{\phi^3_n \partial_3 v})_{n\geq 1}\|_{L^2(\eta,\xi;H^1(\Tor^2;\ell^2))}^2,\\
\Ib_4&:=\E\|\Lpp\|_{L^2(\eta,\xi;L^2(\Tor^2))},\\
\Ib_5&:=\E\|\overline{v}\cdot\nabla_{\h} \overline{v}\|_{L^2(\eta,\xi;L^2(\Tor^2))}^2,
\end{align*}
and $\Lpp$ and $\force$ are as in \eqref{eq:def_P_gamma} and \eqref{eq:primitive_bar}, respectively.

Below we consider each term separately. 
By H\"{o}lder inequality and \eqref{eq:bound_f_g_expectations},
\begin{equation}
\label{eq:Ib_1_estimate}
\Ib_1\leq C_h C_{k,T},
\end{equation}
where $C_h$ depends only on $h$. Similarly, 
\begin{equation}
\label{eq:Ib_2_estimate}
\Ib_2\leq C_h \E\Big\||\wt{v}||\nabla \wt{v}|\Big\|^2_{L^2(\eta,\xi;L^2)}.
\end{equation}

Next we estimate $\Ib_3$. To this end, note that  $
\Ib_3\leq \Ib_{3,0}+\Ib_{3,1}+\Ib_{3,2}$ where 
\begin{align*}
\Ib_{3,0}&:= \E\big\|(\overline{\phi^3_n \partial_3 v})_{n\geq 1}\big\|_{L^2(\eta,\xi;L^2(\Tor^2;\ell^2))}^2,\\
\Ib_{3,1}&:= \max_{j\in \{1,2\}}
\E\big\|(\overline{\partial_j \phi^3_n \partial_3 v})_{n\geq 1}\big\|_{L^2(\eta,\xi;L^2(\Tor^2;\ell^2))}^2,\\
\Ib_{3,2}&:= \max_{j\in \{1,2\}}
\E\big\|(\overline{\phi^3_n \partial_j\partial_3 v})_{n\geq 1}\big\|_{L^2(\eta,\xi;L^2(\Tor^2;\ell^2))}^2.
\end{align*}
By Assumption \ref{ass:well_posedness_primitive}\eqref{it:well_posedness_primitive_phi_smoothness} we have $\|(\phi_n(t,x))_{n\geq 1}\|_{\ell^2}\leq C_{\delta} M$ a.s.\ for all $t\in \R_+$ and $x\in \Dom$ by Sobolev embeddings (cf.\ Remark \ref{r:assump_local_existence_strong_weak}\eqref{it:Holder_continuity_phi}) and therefore $\Ib_{3,0}\leq C_k$  by \eqref{eq:boundedness_on_xi_n} and $\xi\leq \ell_k$. Since $(\partial_j \phi_n)_{n\geq 1}\in L^{3+\delta}(\Dom;\ell^2)$ also by Assumption \ref{ass:well_posedness_primitive}\eqref{it:well_posedness_primitive_phi_smoothness}, by a standard interpolation inequality we get
$$
\Ib_{3,1}+\Ib_{3,2}\leq C\big(\E \|\nabla \partial_3 v\|_{L^2(\eta,\xi;L^2)}^2 
+ \E \| \partial_3 v\|_{L^2(\eta,\xi;L^2)}^2\big) \stackrel{\eqref{eq:boundedness_on_xi_n}}{\leq }
C\E \|\nabla \partial_3 v\|_{L^2(\eta,\xi;L^2)}^2+ C_k.
$$ 
In turn, we have proved
\begin{equation}
\label{eq:Ib_3_estimate}
\Ib_3\leq C\E \|\nabla \partial_3 v\|_{L^2(\eta,\xi;L^2)}^2+ C_k.
\end{equation}

To estimate $\Ib_4$, note that by Assumption \ref{ass:global_primitive}\eqref{it:independence_z_variable},
\begin{equation*}
\q[(\phi_n\cdot \nabla ) v]=
\q_{\h} \Big[\overline{(\phi_n\cdot \nabla ) v}\Big]
=
\q_{\h} \Big[(\phi_{n,\h}\cdot \nabla ) \overline{v} \Big]+ \q_{\h} \Big[\overline{\phi_n^3 \partial_3v} \Big].
\end{equation*} 
Since $(\hp_n^{i,j}(t,\om,\cdot))_{n\geq 1}\in L^{3}(\Dom;\ell^2)$ uniformly w.r.t.\ $(t,\om)$ by Assumption \ref{ass:well_posedness_primitive}\eqref{it:well_posedness_primitive_phi_smoothness}, the H\"{o}lder inequality yields
\begin{align*}
\Ib_4
&\lesssim
\E\|\big((\phi_{n,\h}\cdot \nabla ) \overline{v}\big)_{n\geq 1}\|_{L^2(\eta,\xi;L^6(\ell^2))}^2
+
\E\|\big((\overline{\phi^3_n \partial_3v}\big)_{n\geq 1}\|_{L^2(\eta,\xi;L^6(\ell^2))}^2\\
&\lesssim_M
\E\|\nabla \overline{v}\|_{L^2(\eta,\xi;L^6(\Tor^2))}^2
+
\E\|\partial_3 v\|_{L^2(\eta,\xi;L^6)}^2\\
&\leq \frac{1}{4\overline{C}}
\E\|\overline{v}\|_{L^2(\eta,\xi;H^2(\Tor^2))}^2
+
C_M \Big( \E\|\overline{v}\|_{L^2(\eta,\xi;H^1(\Tor^2))}^2
+
\E\|\partial_3 v\|_{L^2(\eta,\xi;L^6)}^2\Big)
\end{align*}
where $\overline{C}$ is as in \eqref{eq:estimate_overline_v_smr_2} and in the last inequality we used that $\overline{v}$ is two-dimensional and $\|f\|_{L^6(\Tor^2)}\leq \varepsilon \|f\|_{H^1(\Tor^2)}+C_{\varepsilon}\|f\|_{L^2(\Tor^2)}$ by Young's and standard interpolation inequalities. 
By \eqref{eq:boundedness_on_xi_n}, the above displayed estimate and the Sobolev embedding $H^1\embed L^6$, we get
\begin{equation}
\begin{aligned}
\label{eq:Ib_4_estimate}
\Ib_4
&\leq \frac{1}{4\overline{C}}
\E\|\overline{v}\|_{L^2(\eta,\xi;H^2(\Tor^2))}^2
+
C_k
+ C_M
\E\|\partial_3 v\|_{L^2(\eta,\xi;L^6(\ell^2))}^2\\
&\leq \frac{1}{4\overline{C}}
\E\|\overline{v}\|_{L^2(\eta,\xi;H^2(\Tor^2))}^2
+
C_k
+ C_M
\E\|\nabla \partial_3 v\|_{L^2(\eta,\xi;L^2(\ell^2))}^2.
\end{aligned}
\end{equation}
It remains to estimate $\Ib_5$. Recall that, by standard interpolation inequalities and Sobolev embeddings,
\begin{equation}
\label{eq:interp_inequality_2D}
\|f\|_{L^4(\Tor^2)}\lesssim \|f\|_{H^{1/2}(\Tor^2)}
\lesssim \|f\|_{H^1(\Tor^2)}^{1/2}\|f\|_{H^1(\Tor^2)}^{1/2}
\end{equation}
where $t>0$. Applying \eqref{eq:interp_inequality_2D} to $\overline{v}$ and $\nabla_{\h}\overline{v}$, we get
\begin{align*}
\E\Big\||\overline{v}||\nabla_{\h} \overline{v}|\Big\|_{L^2(\eta,\xi;L^2(\Tor^2))}^2
&\leq  
\E\int_{\eta}^{\xi} 
\Big[\|\overline{v}\|_{L^4(\Tor^2)}^2\|\nabla_{\h} \overline{v}\|_{L^4(\Tor^2)}^2\Big]\,ds 
\\
&\stackrel{\eqref{eq:boundedness_on_xi_n}}{\lesssim} \E\int_{\eta}^{\xi}\Big[
k\|\overline{v}\|_{H^1(\Tor^2)}^2\|\overline{v}\|_{H^2(\Tor^2)}\Big]\,ds.
\end{align*}
Since $\|\overline{v}\|_{H^1(\Tor^2)}\lesssim_h\|v\|_{H^1}$, the Young's inequality yields
\begin{equation}
\begin{aligned}
\label{eq:product_overline_v_estimate}
\E\Big\||\overline{v}||\nabla \overline{v}|\Big\|_{L^2(\eta,\xi;L^2(\Tor^2))}^2
\leq C_{k} \E\int_{\eta}^{\xi}\|v\|_{H^1}^2\| \overline{v}\|_{H^1(\Tor^2))}^2\,ds 
+\frac{1}{4\overline{C}}\, \E\|\overline{v}\|_{L^2(\eta,\xi;H^2(\Tor^2))}^2.
\end{aligned}
\end{equation}
Using \eqref{eq:Ib_1_estimate}, \eqref{eq:Ib_2_estimate}, \eqref{eq:Ib_3_estimate}, \eqref{eq:Ib_4_estimate} and \eqref{eq:product_overline_v_estimate} in \eqref{eq:estimate_overline_v_smr_2} we get
\begin{equation}
\label{eq:estimate_Step_1}
\begin{aligned}
&\E\Big[\sup_{s\in [\eta,\xi]} \|\overline{v}(s)\|_{H^1(\Tor^2)}^2\Big]
+\E \int_{\eta}^{\xi}\| \overline{v}\|_{H^2(\Tor^2)}^2\,ds\\
& \qquad
\leq   C_{k,T}^{(1)}\Big[1+\|\overline{v}(\eta)\|_{H^1(\Tor^2)}^2 +
\E\int_{\eta}^{\xi} (1+\|v\|_{H^1}^2)\|\overline{v}\|_{H^1(\Tor^2)}^2\,ds\\
&\qquad \qquad
 + \E\int_{\eta}^{\xi}  \big\| |\wt{v}| |\nabla \wt{v}|\big\|_{L^2}^2\,ds +\E\int_{\eta}^{\xi} \|\nabla \partial_3 v\|_{L^2}^2\,ds\Big]
\end{aligned}
\end{equation}
where $C_{k,T}^{(1)}$ is a constant independent of $v,v_0,\eta,\xi$ and $j\geq 1$. Note that the term $\frac{1}{2}\E \| \bar{v}\|_{L^2(\eta,\xi;H^2)}^2$ on the right hand side of \eqref{eq:product_overline_v_estimate} has been absorbed in the left hand side of \eqref{eq:estimate_Step_3}. This is possible since $\xi\leq \tau_j$ and therefore $\|\overline{v}\|_{L^2(\eta,\xi;H^2)}\lesssim_h\|v\|_{L^2(\eta,\xi;H^2)}\leq j$ a.s.\

\emph{Step 2: $L^{\infty}_t(L^2_x)$- and $L^2_t(H^1_x)$-estimates for $\vz:=\partial_3 v$, see \eqref{eq:estimate_Step_2} below}. Let us set $\vz^{\eta,\xi}:=\vz((\cdot\vee \eta)\wedge \xi)$. Note that $\vz^{\eta,\xi}\in C([0,T];L^2)$ a.s.\ and $\|\vz^{\eta,\xi}\|_{C([0,T];L^2)}\leq j $ a.s.\ since $\xi\leq \tau_j$ by assumption.
Using an approximation argument similar to the one used in Step 2 of Proposition \ref{prop:SMR_2}, we may apply the It\^{o}'s formula applied to $v\mapsto  \| \vz\|_{L^2}^2$ and an integration by parts argument yield, a.s.\ for all $t\in [0,T]$,
\begin{equation}
\label{eq:Ito_Step_2_primitive}
\|\vz^{\eta,\xi}(t)\|_{L^2}^2 -\|\vz(\eta)\|_{L^2}^2
+ 2\int_{0}^t\one_{[\eta,\xi]} \int_{\Dom}|\nabla \vz|^2 \,dx ds =  \sum_{j=1}^4 I_j(t) ,
\end{equation}
where 
\begin{align*}
I_1(t)&:= 2\int_{0}^t\one_{[\eta,\xi]} 
\Big( -(v\cdot \nabla_{\h}) v- \w(v)\partial_3 v, \partial_3 \vz\Big)_{L^2} ds,\\
I_2(t)&:=2\int_{0}^t\one_{[\eta,\xi]}(\fvt, \partial_3 \vz)_{L^2}ds,
\\
I_3(t)&:=\sum_{n\geq 1}\int_{0}^t\one_{[\eta,\xi]}
	\Big\|\partial_3 [(\phi_n\cdot \nabla ) v] + \partial_3 \gvtn]\Big\|_{L^2}^2\,ds,\\
I_4(t)&:=2 \sum_{n\geq 1}\int_{0}^t \one_{[\eta,\xi]} 
 \big(\partial_3 [(\phi_n\cdot \nabla ) v] +\partial_3 \gvtn,\partial_3 \vz\big)_{L^2} d\beta_s^n,
\end{align*}
and where we used that $\partial_3 \p f= f$ and $(\p f,\partial_3 \vz)_{L^2}=(f,\p (\partial_3 \vz))_{L^2}=(f,\partial_3 \vz)_{L^2}$ and $(\Lpp , \partial_3 v_3)=0$ for all $f\in L^2$ since $\vz(\cdot,-h)=\vz(\cdot,0)=0$ on $\Tor^2$ due to \eqref{eq:boundary_conditions_strong_weak}  and to the fact that $\Lpp$ is $x_3$-independent (cf.\ \eqref{eq:Lpp_independent_of_x_3}).

Fix $\ellip_1\in (0,2-\ellip)$ where $\ellip<2$ is as in Assumption \ref{ass:well_posedness_primitive}\eqref{it:well_posedness_primitive_parabolicity}. By repeating the argument in \cite[Step 2, p. 24]{HH20_fluids_pressure} one can check that, a.s.\ for all $t\in [0,T]$,
\begin{equation}
\label{eq:bound_I_det_s_step_2}
\begin{aligned}
\E\Big[\sup_{t\in [0,T]}|I_{1}(t)|\Big]
&\leq C \E\int_{\eta}^{\xi}\Big[(1+\|v\|_{H^1}^2)\|\vz\|_{L^2}^2
+ \big\||\wt{v}(s)| |\nabla \wt{v}(s)|\big\|_{L^2}^2\Big]\, ds\\
&+\ellip_1 \E\int_{\eta}^{\xi}\|\nabla \vz\|_{L^2}^2\,ds.
\end{aligned}
\end{equation}
Moreover, by \eqref{eq:bound_f_g_expectations}, for all $\ellip_2\in (0,2-\ellip-\ellip_1)$ we have a.s.
\begin{equation}
\label{eq:I_2_estimate_step_2_vz}
\E \Big[\sup_{t\in [0,T]}|I_2(t)|\Big] \leq C_k+ \ellip_2 \E\|\nabla\vz\|_{L^2(\eta,\xi;L^2)}^2.
\end{equation}

Next we estimate $I_3(t)$. To this end, let us fix $\ellip_2'\in (\ellip,2-\ellip_1-\ellip_2)$. Note that such choice is always possible since $\ellip<2-\ellip_1-\ellip_2$. As in the proof of \eqref{eq:estimate_Ito_correction_v_smr}, using Assumption \ref{ass:well_posedness_primitive}\eqref{it:well_posedness_primitive_phi_smoothness} and \eqref{it:well_posedness_primitive_parabolicity} one gets
\begin{equation*}
\sup_{t\in [0,T]}\sum_{n\geq 1}\int_{\eta}^{\xi}
	\Big\|\partial_3 [(\phi_n\cdot \nabla ) v]\Big\|_{L^2}^2\,ds
	\leq \ellip_2' \|\nabla \vz\|_{L^2}^2 + c_{\ellip'}\|v\|_{L^2}^2.
\end{equation*}
By \eqref{eq:bound_f_g_expectations} and the Young's inequality we have, for all $\ellip_3\in (\ellip_2',2-\ellip_1-\ellip_2)$, 
\begin{equation}
\label{eq:I_3_estimate_step_2_vz}
\E\Big[\sup_{t\in [0,T]}|I_3(t)|\Big]
\leq \ellip_3 \E \|\nabla\vz\|_{L^2(\eta,\xi;L^2)}^2+  C_{k,\ellip_2,\ellip_3}.
\end{equation}

Taking $\E$ in \eqref{eq:Ito_Step_2_primitive}, using that $\E[I_4(t)]=0$ and $2-\sum_{j=1}^3\ellip_j>0$, one infers
\begin{equation}
\label{eq:estimate_step_2_primitive_intermediate}
\begin{aligned}
\E \int_{\eta}^{\xi} \|\nabla \vz(s)\|^2_{L^2} \,ds
 &\leq \E\|\vz(\eta)\|_{L^2}^2+ C_k'\\
&+C \E \int_{\eta}^{\xi}\Big[\big\||\wt{v}||\nabla\wt{v}|\big\|_{L^2}^2+ \big(1+\|v\|_{H^1}^2\big) \|\vz\|_{L^2}^2\Big]\,ds 
\end{aligned}
\end{equation}
where $C_k'$ and $C$ are independent of $\eta,\xi$ and $j\geq 1$.

Next we apply $\E\big[\sup_{s\in [0,T]} |\cdot| \big]$ to \eqref{eq:Ito_Step_2_primitive}. To this end, it remains to estimate $I_4$. For notational convenience, we write $I_4=I_{4,1}+I_{4,2}$ where
\begin{align*}
I_{4,1}(t)&:=
\sum_{n\geq 1}\int_{0}^t \one_{[\eta,\xi]} 
 \big(\partial_3 [(\phi_n\cdot \nabla ) v], \partial_3\vz\big)_{L^2} d\beta_s^n,\\
 I_{4,2}(t)&:=
\sum_{n\geq 1}\int_{0}^t \one_{[\eta,\xi]} 
 \big(\partial_3 [\gvn(v)],\partial_3 \vz\big)_{L^2} d\beta_s^n.
\end{align*}
As above, $\E\int_{\eta}^{\xi}\|(\gvtn )_{n\geq 1}\|_{L^2}^2 \lesssim k$ by \eqref{eq:bound_f_g_expectations}. Thus, the estimate \eqref{eq:estimate_step_2_primitive_intermediate} and Burkholder-Gundy-Davis inequality yield
$$
\E\Big[\sup_{t\in [0,T]}|I_{4,2}(t)|\Big]\leq C_k'
+ C\E \int_{\eta}^{\xi}\Big[\big\||\wt{v}||\nabla\wt{v}|\big\|_{L^2}^2+ \big(1+\|v\|_{H^1}^2\big) \|\vz\|_{L^2}^2\Big]\,ds.
$$

Again, by the Burkholder-Gundy-Davis inequality,
\begin{align*}
&\E\Big[\sup_{s\in [0,T]}\big|I_{4,1}(t)\big|  \Big]\\
&\leq C 
\E\Big[
\int_{\eta}^{\xi} \Big|
\int_{\Dom} \sum_{n\geq 1} \big[\partial_3 ((\phi_n\cdot \nabla ) v) \big]\cdot \vz\,dx\Big|^2 \,ds \Big]^{1/2}
\\
&\leq C \E
\Big[
\int_{\eta}^{\xi}(1+\|\nabla \vz(s)\|_{L^2}^2) \| \vz(s)\|_{L^2}^2 \,ds 
\Big]^{1/2}\\
&\leq  C \E
\Big[\Big(\sup_{s\in [\eta,\xi]}\|\vz(s)\|_{L^2}^2\Big)^{1/2}
\Big(
\int_{\eta}^{\xi}
(1+\|\nabla \vz(s)\|_{L^2}^2)  \,ds 
\Big)^{1/2}\Big]\\
&\stackrel{(i)}{\leq} \frac{1}{2}\E\Big[\sup_{s\in [0,T]}\|\vz^{\eta,\xi}(s)\|_{L^2}^2\Big]
+C\E\int_{\eta}^{\xi}(1+
\|\nabla \vz(s)\|_{L^2}^2)  \,ds\\
&\stackrel{(ii)}{\leq} \frac{1}{2}\E\Big[\sup_{s\in [0,T]}\|\vz^{\eta,\xi}(s)\|_{L^2}^2\Big]
+ C_k'+ \E \int_{\eta}^{\xi}\Big[\big\||\wt{v}||\nabla\wt{v}|\big\|_{L^2}^2+C \big(1+\|v\|_{H^1}^2\big) \|\vz\|_{L^2}^2\Big]\,ds 
\end{align*}
where in $(i)$ we used $\vz^{\eta,\xi}=\vz((\cdot\vee \eta)\wedge \xi)$ and in $(ii)$ \eqref{eq:estimate_step_2_primitive_intermediate}.

Applying
$\E\big[\sup_{s\in [0,T]} |\cdot| \big]$ to \eqref{eq:Ito_Step_2_primitive} and using that \eqref{eq:bound_I_det_s_step_2}, \eqref{eq:I_2_estimate_step_2_vz}, \eqref{eq:I_3_estimate_step_2_vz}, \eqref{eq:estimate_step_2_primitive_intermediate} as well as the previous on $I_{4,1}$ and $I_{4,2}$, we have 
\begin{equation}
\label{eq:estimate_Step_2}
\begin{aligned}
\E \Big[\sup_{s\in [\eta, \xi]} \|\vz(t)\|^2_{L^2}\Big]
&+\E\int_{\eta}^{\xi}\|\nabla \vz\|_{L^2}^2 \,ds
\leq  C_{k,T}^{(2)}\Big[ 1+ \E\| \vz(\eta)\|_{L^2}^2 \\
&
+\E\int_{\eta}^{\xi} \Big((1+\|v\|_{H^1}^2) \|\vz\|_{L^2}^2
+ \big\||\wt{v}||\nabla\wt{v}|\big\|_{L^2}^2 \Big)\,ds\Big]
\end{aligned}
\end{equation}
where $C^{(2)}_{k,T}$ is a constant independent of $\eta,\xi$ and $j\geq 1$. Let us remark that the term $\frac{1}{2}\E[\sup_{s\in [0,T]}\|\vz^{\eta,\xi}(s)\|_{L^2}^2]$ appearing in the estimate of $I_{4,1}$ has been absorbed in the left hand side of the previous estimate.

\emph{Step 3: An $L^{\infty}_t(L^4_x)$--estimate for $\wt{ v}$, see \eqref{eq:estimate_Step_3} below}. As in the previous step we set $\wt{v}^{\eta,\xi}:=\wt{v}((\cdot\vee \eta)\wedge \xi)$. By the Sobolev embedding $H^1\embed L^4$ and $\xi\leq \tau_j$, we have $\wt{v}^{\eta,\xi}\in C([0,T];L^4)$ a.s.\ and $\|\wt{v}^{\eta,\xi}\|_{C([0,T];L^4)}\leq j$ a.s.

The It\^{o}'s formula applied to $\wt{v}\mapsto  \| \wt{v}\|_{L^4}^4$, gives a.s.\ for all $t\in [0,T]$,
\begin{equation}
\label{eq:Ito_identity_L_4_norm}
\begin{aligned}
\|\wt{v}^{\eta,\xi}(t)\|_{L^4}^4 -\|\wt{v}(\eta)\|_{L^4}^4 &
+2\int_0^t \one_{[\eta,\xi]}\Big\|\nabla[ |\wt{v}|^2 ]\Big\|_{L^2}^2\,ds 
\\
&+ 4\int_0^t \one_{[\eta,\xi]}
\Big\| |\wt{v}||\nabla \wt{v}| \Big\|_{L^2}^2 \,ds = \sum_{j=1}^5 J_j(t),
\end{aligned}
\end{equation}
where
\begin{align*}
J_1 (t)&:=- 4 \int_0^t \one_{[\eta,\xi]}\int_{\Dom} [(\wt{v}\cdot\nabla_{\h} \overline{v})] \cdot\wt{v}  |\wt{v}|^2 \,dx ds\\
&-\frac{4}{h}
 \int_0^t \one_{[\eta,\xi]}\int_{\Dom} \Big(\int_{-h}^0 \big[(\wt{v}\cdot\nabla_{\h} \wt{v}) + \wt{v}(\div_{\h} \wt{v})\big]\,d\zeta\Big) \cdot \wt{v} |\wt{v}|^2 \,dx ds\\
 &+4
 \int_0^t \one_{[\eta,\xi]}\int_{\Dom} \wt{\fvt} \cdot \wt{v} |\wt{v}|^2 \,dx ds,\\
 J_2(t)&:=2 \int_0^t \one_{[\eta,\xi]} \int_{\Dom} |\wt{v}|^2 
\sum_{n\geq 1} \Big|(\phi_n\cdot \nabla )\wt{v}+\overline{\phi^3_n \partial_3 v }+\wt{\gvtn} \Big|^2\,dxds\\
&\qquad \qquad + 4\int_0^t \one_{[\eta,\xi]}  \int_{\Dom}\sum_{n\geq 1} \Big|\wt{v}\cdot \Big[
(\phi_n\cdot \nabla )\wt{v}+\overline{\phi^3_n \partial_3 v }+\wt{\gvtn}\Big]\Big|^2\,dxds,
\\ 
J_3(t)&:=
4\int_{0}^t \one_{[\eta,\xi]} \int_{\Dom}\sum_{n\geq 1} |\wt{v}|^2 
\wt{v}\cdot \big[(\phi_n\cdot \nabla ) \wt{v}+\overline{\phi^3_n \partial_3 v }+ \wt{\gvtn}\big]\,dx d\beta_s^n,
\end{align*}
and we used that integrating by parts (cf.\ \cite[Lemma 2a), p.\ 21]{HH20_fluids_pressure})
\begin{align*}
\int_{\Dom} \big[(\overline{v}\cdot \nabla_{\h}) \wt{v}\big]\cdot \wt{v} |\wt{v}|^2\,dx 
=\int_{\Dom} \big[(\wt{v}\cdot \nabla_{\h}) \wt{v} +w(v)\partial_3 \wt{v}\big]\cdot \wt{v} |\wt{v}|^2\,dx=0,
\end{align*}
since $\div_{\h} \wt{v}+\partial_3 (w(v))=\div_{\h}\overline{v}=0$ a.e.\ on $[\eta,\xi]\times \O$.

Let us remark that, to justify the above identity, one needs a standard approximation argument. More precisely to prove \eqref{eq:Ito_identity_L_4_norm}, one applies the It\^{o}'s formula to $\wt{v}\mapsto \int_{\Dom}\M(\mathcal{R}_m(\wt{v}))\,dx$ where
$$
\M(y):=
\begin{cases}
|y|^4,&\text{ on }|y|\leq m,\\
m^{2}\big( 6 |y|^2- 8 m |y|+3 m^2), &\text{ on }|y|>m,
\end{cases}  \quad \text{ for all } y\in \R, m\geq 1,
$$
and $\mathcal{R}_m:=m(m+1+\Delta_{\n})^{-1}$ (here $\Delta_\n$ denotes the Neumann Laplacian on $L^2$) and
 then taking the limit as $m\to \infty$ in the obtained equality. Since $\M\in C^2_{\rm b}(\R^2)$, $\M$ has quadratic growth at infinity, $\mathcal{R}_m\to I$ strongly in $H^k$ for $k\in\{0,1\}$ and $\Do(\Delta_{\n})\embed H^2\embed L^{\infty}$, the It\^{o}'s formula can be applied. By \eqref{eq:tau_j_a_priori_estimates_strong_weak} and the fact that $\xi\leq \tau_j$ a.s., the limit as $m\to \infty$ can be justified by recalling that $H^1\embed L^6$ and noticing that, for all $y=(y_1,y_2)\in \R^2$, $i,j\in \{1,2\}$, 
\begin{align*}
\M(y)\to |y|^4, \ \partial_{y_j}\M(y)\to 4 y_j |y|^3, \ \partial_{y_i,y_j}\M(y)\to 4|y|^2 \delta_{i,j} + 8 y_j y_i,  \text{ as }m\to \infty,&\\
|\M(y)|\leq C |y|^4, \  |\partial_{y_j} \M(y)|\leq C|y|^3, \ |\partial_{y_i,y_j}\M(y)| \leq C |y|^2, \text{ for all }m\geq 1,&
\end{align*}
where $C>0$ is independent of $m\geq 1$ and $\delta_{i,j}$ is the Kronecker's delta.

Let us turn to the proof of the main estimate. Fix $\ellip'\in (\ellip,2)$. Reasoning as in \cite[Step 3, p.\ 26]{HH20_fluids_pressure}, we have 
\begin{equation}
\label{eq:estimate_J_1}
\sup_{s\in [0,t]} |J_{1}(t)|
\leq C_k
\Big[1+
\int_{\eta}^{\xi}\big(1+N_{v,\T}+
\|v\|_{H^1}^2 \big) \|\wt{v}\|_{L^4}^4 \,ds\Big]
+
\frac{2-\ellip'}{2}\int_{\eta}^{\xi} \Big\|\nabla[ | \wt{v}|^2]\Big\|_{L^2}^2\,ds.
\end{equation}
Next we analyse $J_2$. Note that, a.s.\ for all $t\in [0,T]$,
\begin{equation}
\label{eq:J_2_decomposition_estimate}
J_2(t)
\leq \frac{\ellip'}{\ellip} J_{2,1}(t)+ C_{\ellip,\ellip'} (J_{2,2}(t)+J_{2,3}(t)),
\end{equation}
where
\begin{align*}
J_{2,1}(t)
&:=\sum_{n\geq 1} \int_0^t \one_{[\eta,\xi]} \int_{\Dom}\Big[ 2  |\wt{v}|^2 
\big|(\phi_n\cdot \nabla )\wt{v}\big|^2 +
4 \big|\wt{v}\cdot \big[ (\phi_n\cdot \nabla )\wt{v}\big]\big|^2\Big]\,dxds,
\\ 
J_{2,2}(t)
&:=
 \int_0^t \one_{[\eta,\xi]} \int_{\Dom} |\wt{v}|^2 
\sum_{n\geq 1} \big|\overline{\phi^3_n \partial_3 v} \big|^2\,dxds,\\
J_{2,3}(t)
&:=
 \int_0^t \one_{[\eta,\xi]} \int_{\Dom} |\wt{v}|^2 
\sum_{n\geq 1} |\wt{\gvtn} |^2\,dxds.
\end{align*}
Next we estimate $J_{2,1}$, $J_{2,2}$ and $J_{2,3}$, separately. Let us begin by looking at $J_{2,1}$. Note that 
$$
\wt{v}\cdot [(\phi_n\cdot \nabla) \wt{v}]=\frac{1}{2}  (\phi_n\cdot \nabla )|\wt{v}|^2.
$$
Thus, by Assumption \ref{ass:well_posedness_primitive}\eqref{it:well_posedness_primitive_parabolicity} applied twice, a.s.\ for all $t\in [0,T]$,
\begin{align*}
J_{2,1}(t)
&= \sum_{n\geq 1}   \int_0^t \one_{[\eta,\xi]} \int_{\Dom} \Big[ 2 |\wt{v}|^2 
|(\phi_n\cdot \nabla )\wt{v}|^2+ \big|
(\phi_n\cdot \nabla )|\wt{v}|^2\big|^2\Big]\,dxds\\
&\leq 
\ellip  \int_0^t \one_{[\eta,\xi]} \int_{\Dom} \Big[ 2 |\wt{v}|^2 
|\nabla \wt{v}|^2+ \big| \nabla |\wt{v}|^2\big|^2\Big]\,dxds.
\end{align*}
Next we estimate $J_{2,2}$. Set $\overline{F}:=(\sum_{n\geq 1} 
|\overline{\phi^3_n \partial_3 v}|^2)^{1/2}$. By Cauchy-Schwartz inequality, a.s.\ for all $t\in [\eta,\xi]$, 
\begin{equation}
\label{eq:new_term_estimate_1}
\int_{\Dom} |\wt{v}|^2 
\sum_{n\geq 1} \big|\overline{\phi^3_n \partial_3 v} \big|^2\,dx
\leq \big\||\overline{F}|^2\big\|_{L^2}\big\||\wt{v}|^2\big\|_{L^2}
= \|\overline{F}\|_{L^4}^2\|\wt{v}\|_{L^4}^2.
\end{equation}
Since $\overline{F}$ is $x_3$-independent, $\|\overline{F}\|_{L^4}^2\lesssim_h\|\overline{F}\|_{L^4(\Tor^2)}^2$. Thus, reasoning as in the proof of \eqref{eq:Ib_3_estimate}, by Assumption \ref{ass:well_posedness_primitive}\eqref{it:well_posedness_primitive_phi_smoothness} and \eqref{eq:interp_inequality_2D}, we get
\begin{equation}
\label{eq:new_term_estimate_2}
\|\overline{F}\|_{L^4(\Tor^2)}^2\lesssim \|\overline{F}\|_{L^2(\Tor^2)} \|\overline{F}\|_{H^1(\Tor^2)}
\lesssim_{M,\delta} \|v\|_{H^1}^2+\|v\|_{H^{1}}\|\nabla \partial_3 v\|_{L^2}.
\end{equation}
Let $\varepsilon>0$ be chosen later. By \eqref{eq:new_term_estimate_1}-\eqref{eq:new_term_estimate_2} we get, a.s.\
\begin{equation}
\label{eq:new_term_estimate}
\begin{aligned}
\sup_{t\in [0,T]} |J_{2,2}(t)| 
&\lesssim 
\int_{\eta}^{\xi}
(\|v\|_{H^1}^2+\|v\|_{H^{1}}\|\nabla \partial_3 v\|_{L^2})\|v\|_{L^4}^2\,ds\\
&\leq \varepsilon \int_{\eta}^{\xi}\|\nabla \partial_3 v\|_{L^2}^2\,ds + C_{\varepsilon}
\int_{\eta}^{\xi}
(1+\|v\|_{H^1}^2)\|v\|_{L^4}^4\,ds.
\end{aligned}
\end{equation}
Similarly we estimate $J_{2,3}$. Set $\wt{G}:=(\sum_{n\geq 1} |\wt{g_{n}}|^2)^{1/2}$. By Sobolev embeddings and \eqref{eq:bound_f_g_expectations}, we have $\|\wt{G}\|_{L^2(\eta,\xi;L^4)}\lesssim_h \|g\|_{L^2(0,\ell_k;H^1(\ell^2))}$ a.s. Thus, a.s.\
\begin{equation}
\label{eq:Gforce_L_4_estimate}
\begin{aligned}
\sup_{t\in [0,T]}|J_{2,3}(t) |
&\lesssim 
\int_{\eta}^{\xi}\big\|\wt{G}^2\big\|_{L^2}\big\||\wt{v}|^2\big\|_{L^{2}}\,ds \\
&
\lesssim_h 
\int_{\eta}^{\xi}\|g\|_{L^2(0,\ell_k;H^1(\ell^2))}^2 \|\wt{v}\|_{L^{4}}^2 \,ds
\stackrel{}{\leq} C_k+
\int_{\eta}^{\xi}  N_{v,\T}\|\wt{v}\|_{L^4}^4\,ds 
\end{aligned}
\end{equation}
where in the last inequality we used \eqref{eq:def_N_v_T}, \eqref{eq:bound_f_g_expectations} and $\xi\leq \ell_k$ a.s.

Using the previous estimates in \eqref{eq:J_2_decomposition_estimate}, we have a.s.\ 
\begin{multline}
\label{eq:estimate_J_2}
\sup_{t\in[0,T]}|J_2(t)|
\leq 
C_{\varepsilon}\Big[1+
 \int_{\eta}^{\xi}(1+N_{v,\T}+\|v\|_{H^1}^2)
   \|\wt{v}\|^4_{L^4}\,ds\Big] \\
  + \varepsilon \int_{\eta}^{\xi}\| \nabla \partial_3 v\|_{L^2}^2 \,ds 
   +
2 \ellip' \int_{\eta}^{\xi} \Big\||\wt{v}||\nabla \wt{v}|\Big\|_{L^2}^2\,ds +
\ellip' \int_{\eta}^{\xi}\|\nabla \partial_3 v\|_{L^2}^2\,ds .
\end{multline}
Recall that $\ellip'<2$. Take expectations in \eqref{eq:Ito_identity_L_4_norm}, using that \eqref{eq:estimate_J_1}, \eqref{eq:estimate_J_2} and using that  $\E[J_3(T)]=0$,
\begin{equation}
\label{eq:step_3_estimate_after_expectations}
\begin{aligned}
&\int_{\eta}^{\xi}\big\| |\wt{v}||\nabla \wt{v}|\big\|_{L^2}^2 \,ds
\leq C\big[1+ \E\|\wt{v}(\eta)\|_{L^4}^4\big] \\
& \qquad \qquad 
+ C_{\varepsilon}\E \int_{\eta}^{\xi}(1+N_{v,\T}+\|v\|_{H^1}^2)\|\wt{v}\|_{L^4}^4\,ds
+ \varepsilon \E\int_{\eta}^{\xi}\|\nabla |\wt{v}|^2\|_{L^2}\,ds.
\end{aligned}
\end{equation}

Let $\overline{F}=(\sum_{n\geq 1}|\overline{\phi^3_n \partial_3 v}_n|^2)^{1/2}$ and $\wt{G}=(\sum_{n\geq 1}|\wt{g}_n|^2)^{1/2}$ be as above.
Reasoning as in the previous step, we use the Burkholder-Davis-Gundy inequality to handle the martingale part $J_3$:
\begin{align*}
&\E\Big[\sup_{t\in [0,T]}|J_3(t)|\Big]
\lesssim 
\E 
\Big[\int_{\eta}^{ \xi}\Big|\int_{\Dom} |\wt{v}|^2 \sum_{n\geq 1} 
\big[(\phi_n\cdot \nabla ) \wt{v}+\overline{F}+\wt{\gvtn}\big]\cdot \wt{v}\,dx \Big|^2 ds\Big]^{1/2}\\
&\lesssim 
\E 
\Big[\int_{\eta}^{ \xi} \Big|\int_{\Dom} |\wt{v}|^3( |\nabla \wt{v}|+|\overline{F}|+|\wt{G}|)\,dx \Big|^{2}\,ds\Big]^{1/2}\\
&\lesssim\E
\Big[\Big(\sup_{s\in [\eta, \xi]}\|\wt{v}(s)\|_{L^4}^4\Big)^{1/2}
\Big( \int_{\eta}^{ \xi}	\int_{\Dom} |\wt{v}|^2 \big(|\nabla \wt{v}|^2+|\overline{F}|^2+|\wt{G}|^2\big) \,dxds\Big)^{1/2}\Big]\\
&\leq \frac{1}{2}
\E\Big[\sup_{s\in [\eta,\xi]}\|\wt{v}(s)\|_{L^4}^4\Big]+C 
\E\int_{\eta}^{ \xi} \Big(\Big\| |\wt{v}| |\nabla \wt{v}|\Big\|_{L^2}^2+  \Big\||\wt{v}||\overline{F} |\Big\|_{L^2}^{2} + \Big\||\wt{v}||\wt{G} |\Big\|_{L^2}^{2}\Big)\,ds\\
&\leq 
\frac{1}{2}\E\Big[\sup_{s\in [\eta,\xi]}\|\wt{v}^{\eta,\xi}(s)\|_{L^4}^4\Big]
+  C_{J_3}(1+\E\|\wt{v}_0\|_{L^4}^4) 
\\
&+ C_{J_3}
\E\int_{\eta}^{\xi}\Big[\big(1+N_{v,\T}+
\|v\|_{H^1}^2  \big) \|\wt{v}\|_{L^4}^4 \Big]\,ds
+C_{J_3}\varepsilon\int_{\eta}^{\xi} \|\nabla \partial_3v \|_{L^2}^2 \,ds
\end{align*}
where in the last inequality we used \eqref{eq:new_term_estimate}, \eqref{eq:Gforce_L_4_estimate} and \eqref{eq:step_3_estimate_after_expectations}. Note that $C_{J_3}\geq 1$ is a constant independent of $\eta,\xi$ and $j\geq 1$.

For future convenience, we choose $\varepsilon=1/(16 (1+C_{J_3}) (C_{k,T}^{(2)}\vee 1))$ where $C_{k,T}^{(2)}$ is as in \eqref{eq:estimate_Step_2}.
Then by taking $\E\big[\sup_{s\in [0,T]}|\cdot|\big]$ in \eqref{eq:Ito_identity_L_4_norm} and using the above inequalities we have
\begin{equation}
\label{eq:estimate_Step_3}
\begin{aligned}
&\E\Big[\sup_{t\in [\eta,\xi]}\|\wt{v}\|_{L^4}^4\Big] 
+ \E\int_{\eta}^{ \xi} \Big\||\wt{v}||\nabla \wt{v}|\Big\|_{L^2}^2\,ds 
\leq C_{k,T}^{(3)}\Big[1+ \E\|\wt{v}_0\|_{L^4}^4\big]\\
&+C_{k,T}^{(3)} \E \int_{\eta}^{\xi}(1+N_{v,\T}+\|v\|_{H^1}^2)\|\wt{v}\|_{L^4}^4\,ds
+\frac{1}{16 (C_{k,T}^{(2)}\vee 1)} \E\int_{\eta}^{\xi} \|\nabla \partial_3 v\|_{L^2}^2\,ds
\end{aligned}
\end{equation}
where $C_{k,T}^{(3)}>0$ is independent of $\eta,\xi$ and $j\geq 1$.

\emph{Step 4: Proof of \eqref{eq:energy_estimate_proof_claim}}. 
Let $C^{(1)}_{k,T}$, $C_{k,T}^{(2)}$ and $C_{k,T}^{(3)}$ be as in \eqref{eq:estimate_Step_1}, \eqref{eq:estimate_Step_2} and \eqref{eq:estimate_Step_3}, respectively. 
Without loss of generality we may assume that $C_{k,T}^{(i)}\geq 1$ for $i\in \{1,2,3\}$.
This claim of this step follows by noticing that 
\begin{equation*}
\frac{\eqref{eq:estimate_Step_1}}{8 C^{(1)}_{k,T}C^{(2)}_{k,T}}
+\frac{\eqref{eq:estimate_Step_2}}{4 C^{(2)}_{k,T} } + 
 \eqref{eq:estimate_Step_3} \ \ \ \Longrightarrow  \ \  \
 \eqref{eq:energy_estimate_proof_claim}.
\end{equation*}
More precisely, the above means that \eqref{eq:energy_estimate_proof_claim} follows by  multiplying \eqref{eq:estimate_Step_1} with $(8 C^{(1)}_{k,T}C^{(2)}_{k,T})^{-1}$,  \eqref{eq:estimate_Step_2} with $(4 C^{(2)}_{k,T} )^{-1}$ and then adding the resulting estimates with \eqref{eq:estimate_Step_3}. 

Note that, on the RHS of the resulting estimate, the following terms appear: 
\begin{equation}
\label{eq:resulting_estimate_additional_terms}
\Big(\frac{1}{8 C_{k,T}^{(2)} }+\frac{1}{16 	C_{k,T}^{(2)}}\Big)\E\|\nabla \partial_3 v\|_{L^2(\eta,\xi;L^2)}^2, 
  \ \ \text{ and }\ \ 
\Big(\frac{1}{8 C^{(2)}_{k,T}}+ \frac{1}{4}\Big)
\E\Big\||\wt{v}||\nabla \wt{v}|\Big\|_{L^2(\eta,\xi;L^2)}^2 .
\end{equation}
The first term follows from the RHS of \eqref{eq:estimate_Step_1} and the RHS of \eqref{eq:estimate_Step_3} and the second one from the RHS of \eqref{eq:estimate_Step_1} and the  RHS of \eqref{eq:estimate_Step_2}. However, in the LHS of the resulting estimate we get $\frac{1}{4 C_{k,T}^{(2)} }\E\|\nabla \partial_3 v\|_{L^2(\eta,\xi;L^2)}^2$ by $\frac{\eqref{eq:estimate_Step_2}}{4 C^{(2)}_{k,T}}$ and 
$\E\big\||\wt{v}||\nabla \wt{v}|\big\|_{L^2(\eta,\xi;L^2)}^2$ by  \eqref{eq:estimate_Step_3}. Since $C^{(i)}_{k,T}\geq 1$,  the terms in \eqref{eq:resulting_estimate_additional_terms} can be adsorbed in the LHS of the resulting estimate since $\xi\leq \tau_j$ and therefore a.s.
\begin{align*}
\E\|\nabla v_3\|_{L^2(\eta,\xi;L^2)}^2
&\leq\E\int_{\eta}^{ \xi}\|v\|_{H^2}^2\,ds\leq j,\\
\E\Big\||\wt{v}| |\nabla \wt{v}| \Big\|_{L^2(\eta,\xi;L^2)}^2
&\leq \E\Big[\|\wt{v}\|_{L^{\infty}(\eta,\xi;L^6)}^2 \| \wt{v}\|_{L^2(\eta,\xi;H^{1,3})}^2\Big]\lesssim j^2,
\end{align*} 
by \eqref{eq:tau_j_a_priori_estimates_strong_weak} and Sobolev embeddings.
\end{proof}

We are now in the  position to prove Proposition \ref{prop:energy_estimate_primitive}.

\begin{proof}[Proof of Proposition \ref{prop:energy_estimate_primitive}]
Let $\ell_k$ and $X_t,Y_t$ be as in \eqref{eq:def_ell_n} and \eqref{eq:def_X_Y}, respectively. Recall that $\lim_{k\to \infty} \P(\ell_k=\tau)=1$. 
By Lemma \ref{l:intermediate_estimate}, for each $k\geq 1$ there exists $R_k>0$ for which the stopping times 
\begin{equation}
\label{eq:definition_mu_k}
\mu_k :=\inf\Big\{t\in [0,\ell_k)\,:\, X_t+\int_0^t Y_s\,ds \geq R_k\Big\}, \ \ \text{ where }\ \ \inf\emptyset:=\ell_k,
\end{equation}
satisfy
$$
\lim_{k\to \infty}\P(\mu_k=\tau)=1.
$$
As explained at the beginning of this subsection it is sufficient to prove \eqref{eq:claim_in_terms_v}. 
The idea is to use the stochastic Gronwall's lemma as in the proof of Lemma \ref{l:intermediate_estimate}. To this end, we need a localization argument. For any $j\geq 1$, let $\tau_j $ be the stopping time defined as
$$
\tau_j :=\inf\big\{t\in [0,\tau)\,:\, \|v(t)\|_{H^1}+ \|v\|_{L^2(0,t;H^2)}\geq j\big\}\wedge T, \ \text{ where }\ \inf\emptyset:=\tau.
$$
Reasoning as in the proof of Lemma \ref{l:intermediate_estimate}, by the stochastic Gronwall's lemma \cite[Lemma 5.3]{GHZ_09} and \eqref{eq:definition_mu_k}, it is enough to show the existence of $C_{0,k}>0$ such that for each $j\geq 1$ and each stopping times $0\leq \eta\leq \xi\leq \tau_j\wedge \mu_k$,
\begin{equation}
\label{eq:claim_global_strong_weak_final_part}
\begin{aligned}
&\E\Big[\sup_{t\in [\eta,\xi]} \| v(t)\|_{L^2}^2\Big] 
+\E\int_{\eta}^{\xi} \| v(s)\|_{H^2}^2\,ds
\leq
C_{0,k}[1+\E\|v(\eta)\|_{H^1}^2]\\
& \qquad 
+C_{0,k}\Big[\E\int_{\eta}^{\xi } (1+\|v(s)\|_{L^2}^2+X_s) (1+\|v(s)\|_{H^1}^2+Y_s) \| v(s)\|_{L^2}^2\,ds\Big].
\end{aligned}
\end{equation}

Let $\eta,\xi$ be stopping time such that $0\leq \eta\leq \xi\leq \tau_j\wedge \mu_k$. Recall that $(v,\xi)$ is an $L^2$-local solution to \eqref{eq:primitive_v_proof_global} (see also the text below it). 
Thus, by Proposition \ref{prop:SMR_2} (cf.\ Remark \ref{r:SMR_2}\eqref{it:SMR_random_initial_time}), there is a constant $C_0>0$ independent of $\eta,\xi,j,k$ such that 
\begin{equation}
\label{eq:Gronwall_type_estimate_step_2}
\begin{aligned}
&\E\Big[\sup_{t\in [\eta,\xi]} \| v(t)\|_{L^2}^2\Big] 
+\E\int_{\eta}^{\xi} \| v\|_{H^2}^2\,ds\\
&\leq C_0\Big( \E\|v(\eta)\|_{H^1}^2+ \E\|(v\cdot \nabla_{\h})v\|_{L^2(\eta,\xi;L^2)}^2+
\E\|\w(v)\partial_3 v\|_{L^2(\eta,\xi;L^2)}^2\Big)\\
&+C_0\Big( \E\|\fvt\|_{L^2(\eta,\xi;L^2)}^2+\E\|(\gvtn)_{n\geq1 }\|_{L^2(\eta,\xi;H^1(\ell^2))}^2\Big)
\end{aligned}
\end{equation}
where $\fvt$ and $\gvtn$ are as in \eqref{eq:def_f_g_proof_global}. Note that the last two terms on the right hand side of \eqref{eq:Gronwall_type_estimate_step_2}  are finite due to \eqref{eq:bound_f_g_expectations}.
It remains to estimate the second and the third term on the RHS of \eqref{eq:Gronwall_type_estimate_step_2}.
 
Writing $v=\wt{v}+\overline{v}$, we have
\begin{equation}
\label{eq:v_nabla_h_v_decomposition}
(v\cdot \nabla_{\h})v=
(\overline{v}\cdot \nabla_{\h})\overline{v}
+(\wt{v}\cdot \nabla_{\h})\overline{v}
+(\overline{v}\cdot \nabla_{\h})\wt{v}+(\wt{v}\cdot \nabla_{\h})\wt{v}.
\end{equation}
Reasoning as in Step 1 of Lemma \ref{l:intermediate_estimate} (see \eqref{eq:product_overline_v_estimate}), one has
$
\E\|(\overline{v}\cdot \nabla_{\h})\overline{v}\|_{L^2(\eta,\xi;L^2(\Tor^2))}^2\leq R_k^2.
$
Moreover, by Sobolev embeddings,
\begin{align*}
\E\|(\wt{v}\cdot \nabla_{\h})\overline{v}\|_{L^2(\eta,\xi;L^2)}^2 
&\leq \E\Big[\Big(\sup_{t\in [\eta,\xi]} \|\wt{v}(t)\|_{L^4}^2\Big) \Big(\|\nabla_{\h}\overline{v}\|_{L^2(\eta,\xi;L^4(\Tor^2))}^2\Big)\Big]\\
&\leq R_k \E\Big[\|\nabla_{\h}\overline{v}\|_{L^2(\eta,\xi;L^4(\Tor^2))}^2\Big]
\lesssim R_k^2,
\end{align*}
and 
\begin{align*}
\E\|(\overline{v}\cdot \nabla_{\h})\wt{v}\|_{L^2(\eta,\xi;L^2)}^2 
&\leq \E\Big[\Big(\sup_{t\in [\eta,\xi]} \|\overline{v}(t)\|_{L^6}^2\Big) 
\Big(\|\nabla_{\h}\wt{v}\|_{L^2(\eta,\xi;L^3)}^2\Big)\Big]\\
&\stackrel{\eqref{eq:definition_mu_k}}{\leq} R_k \E
\|\nabla_{\h}v\|_{L^2(\eta,\xi;L^3)}^2\\
&\stackrel{(i)}{\leq} \frac{1}{4C_0}\E\|v\|_{L^2(\eta,\xi;H^2)}^2+ C_{k}\E\|v\|_{L^{2}(\eta,\xi;H^1)}^2,
\end{align*}
where in $(i)$ we used the interpolation inequality $\|f\|_{L^3}\lesssim \|f\|_{L^2}^{1/2}\| f\|_{H^1}^{1/2}$ and Young's inequality.
 
Since $\E\|(\wt{v}\cdot \nabla_{\h})\wt{v}\|_{L^2(\eta,\xi;L^2)}^2\leq R_k$ by definition of $Y_t$ in \eqref{eq:def_X_Y}, \eqref{eq:v_nabla_h_v_decomposition} and the previous estimates yield
\begin{equation}
\label{eq:estimate_v_nabla_v_sum}
\E\|(v\cdot \nabla_{\h})v\|_{L^2(\eta,\xi;L^2)}^2\leq C_{k} +\frac{1}{4C_0}\E\|v\|_{L^2(\eta,\xi;H^2)}^2+ C_{k}\E\|v\|_{L^{2}(\eta,\xi;H^1)}^2.
\end{equation} 
It remains to estimate $w(v)\partial_3 v$. Applying the interpolation inequality $\|f\|_{L^4(\Tor^2)}\lesssim \|f\|_{L^2(\Tor^2)}^{1/2}\| f\|_{H^1(\Tor^2)}^{1/2}$ twice, we have, a.e. on $[\eta,\xi]\times \O$, 
\begin{align*}
\|w(v)\partial_3 v\|_{L^2}^2
&\lesssim \|w(v)\|_{L^{\infty}(-h,0;L^4(\Tor^2))}^2
 \|\partial_3 v\|_{L^{2}(-h,0;L^4(\Tor^2))}^2\\
&\lesssim \|\div_{\h} v\|_{L^{2}(-h,0;L^4(\Tor^2))}^2
 \|\partial_3 v\|_{L^{2}(-h,0;L^4(\Tor^2))}^2\\
&\lesssim \| v\|_{H^{1}}\| v\|_{H^2}\|\partial_3 v\|_{L^2}\| \partial_3 v\|_{H^1}\\
&\leq C  \| v\|_{H^1}^2\|\partial_3 v\|_{L^2}^2\| \partial_3 v\|_{H^1}^2 + \frac{1}{4C_0}\| v\|_{H^2}^2.
\end{align*}
In particular, 
\begin{equation}
\label{eq:w_v_estimate}
\E
\|w(v)\partial_3 v\|_{L^2(\eta,\xi;L^2)}^2
\leq 
C\E\int_{\eta}^\xi  \|\partial_3 v\|_{L^2}^2\| \partial_3 v\|_{H^1}^2\| v\|_{H^1}^2 \,ds+ 
\frac{1}{4C_0}\E\| v\|_{L^2(\eta,\xi;H^2)}^2.
\end{equation}

Using the estimates \eqref{eq:estimate_v_nabla_v_sum} and \eqref{eq:w_v_estimate} in \eqref{eq:Gronwall_type_estimate_step_2}, one obtains \eqref{eq:claim_global_strong_weak_final_part} as desired. Let us remark that the term $\frac{1}{2C_0}\| v\|_{L^2(\eta,\xi;H^2)}^2$ can be absorbed in the left hand side of the resulting estimate since $\xi\leq \tau_j$ and therefore $\| v\|_{L^2(\eta,\xi;H^2)}^2\leq j$ a.s. 
\end{proof}

\section{Local and global existence in the strong-strong setting}
\label{s:strong_strong}
In this section we study \eqref{eq:primitive_intro_v} in the \emph{strong-strong}, i.e.\ in case the equations for $v$ and $\T$ both are understood in the \emph{strong} setting. Compared to the strong-weak setting analyzed in Section \ref{s:strong_weak}, we need additional assumptions on $\psi$. However, we can allow $\fv$ and $\gvn$ to depend on $\nabla \T$ and $\T$, respectively.

As in Section \ref{s:strong_weak}, we begin by reformulating the problem \eqref{eq:primitive_intro_v}. To this end, we apply the hydrostatic Helmholtz projection $\p$ to \eqref{eq:primitive_intro_v}, and at least formally, \eqref{eq:primitive_intro_v} is equivalent to
\begin{equation}
\label{eq:primitive_strong_strong}
\begin{cases}
\displaystyle{d v -\Delta v\, dt=\p\Big[ -(v\cdot \nabla_{\h}) v- \w(v)\partial_3 v + \Lp(\cdot,v,\T)}\\
\qquad \qquad\quad
\displaystyle{
+\nabla_{\h}\int_{-h}^{\cdot} (\kone(\cdot,\zeta)\T(\cdot,\zeta))\,d\zeta + \fv (v,\T,\nabla v,\nabla \T) \Big]dt }\\ 
\ \ \qquad \qquad\qquad \qquad \qquad\
\displaystyle{+\sum_{n\geq 1}\p \Big[(\phi_{n}\cdot\nabla) v  +\gvn(v,\T)\Big] d\beta_t^n}, &\text{on }\Dom,\\
\displaystyle{d \Temp -\Delta \Temp\, dt=\Big[ -(v\cdot\nabla_{\h}) \T -\w(v) \partial_3 \T+ \ft(v,\T,\nabla v,\nabla \T ) \Big]dt}
\\
\ \ \qquad \qquad\qquad \qquad \quad \ \ \ \ 
\displaystyle{+\sum_{n\geq 1}\Big[(\psi_n\cdot \nabla) \Temp+\gtn(v,\T)\Big] d\beta_t^n}, &\text{on }\Dom,\\
\\
v(\cdot,0)=v_0,\ \ \T(\cdot,0)=\T_0, &\text{on }\Dom,
\end{cases}
\end{equation}
complemented with the following boundary conditions 
\begin{equation}
\label{eq:boundary_conditions_strong_strong}
\begin{aligned}
\partial_3 v (\cdot,-h)=\partial_3 v(\cdot,0)=0 \ \  \text{ on }\Tor^2,&\\
\partial_3 \T(\cdot,-h)= \partial_3 \T(\cdot,0)+\alpha \T(\cdot,0)=0\ \  \text{ on }\Tor^2.&
\end{aligned}
\end{equation}
Here $\alpha\in \R$ is given, the subscript $\h$ stands for the horizontal part (see Subsection \ref{ss:set_up}) and a.s.\ for all $t\in \R_+$ and $x=(x_{\h},x_3)\in\Tor^2\times (-h,0)= \Dom$
\begin{align}
\label{eq:def_P_gamma_strong_strong}
\Lp (\cdot,v,\T)
&:= \Big(\sum_{n\geq 1} \sum_{m=1}^2 \hp_n^{\ell,m} (t,x)
\big(\q[(\phi_n\cdot\nabla) v + \gvn(\cdot,v,\T)]\big)^m\Big)_{\ell=1}^2,\\
\label{eq:def_w_v_strong_strong}
w(v)
&:=-\int_{-h}^{x_3}\div_{\h} v(t,x_{\h},\zeta)\,d\zeta.
\end{align}
As in Section \ref{s:strong_weak}, in the stochastic part of the equation for the velocity field $v$, the operator $\p$ cannot be (in general) removed and the term $\Lp$ defined in \eqref{eq:def_P_gamma_strong_strong} coincides with $\partial_{\hp}\wt{P}_n$ in \eqref{eq:primitive_intro_v} since $\q[(\phi_n\cdot\nabla) v + \gvn(\cdot,v,\T)]=\nabla\wt{P}_n$.

\subsection{Main assumptions and definitions}
\label{ss:strong_strong_statement}
We begin by listing the main assumption of this section. Below we employ the notation introduced in Subsection \ref{ss:set_up}.
 
\begin{assumption} There exist $M,\delta>0$ for which the following hold.
\label{ass:well_posedness_primitive_double_strong}
\begin{enumerate}[{\rm(1)}]
\item\label{it:well_posedness_measurability_strong_strong}
For all $j\in \{1,2,3\}$ and $n \geq 1$, the maps
$$\phi_n^j,\psi_n^j,\kone: \R_+\times \O\times \Dom\to \R$$ are $\Progress\otimes \mathscr{B}$-measurable;
\item\label{it:well_posedness_primitive_phi_psi_smoothness} 
a.s.\ for all $t\in \R_+$, $j,k\in \{1,2,3\}$, $\ell,m\in \{1,2\}$,
\begin{align*}
\Big\|\Big(\sum_{n\geq 1}| \phi^j_n(t,\cdot)|^2\Big)^{1/2} \Big\|_{L^{3+\delta}(\Dom)}+
\Big\|\Big(\sum_{n\geq 1}|\partial_k \phi^j_n(t,\cdot)|^2\Big)^{1/2} \Big\|_{L^{3+\delta}(\Dom)} \leq M,&\\
\Big\|\Big(\sum_{n\geq 1}| \psi^j_n(t,\cdot)|^2\Big)^{1/2} \Big\|_{L^{3+ \delta}(\Dom)}+
\Big\|\Big(\sum_{n\geq 1}|\partial_k \psi^j_n(t,\cdot)|^2\Big)^{1/2} \Big\|_{L^{3+ \delta}(\Dom)} \leq M,&\\
\Big\|\Big(\sum_{n\geq 1}|\hp^{\ell,m}_n(t,\cdot)|^2\Big)^{1/2}\Big\|_{L^{3+\delta}(\Dom)}\leq M;&
\end{align*}
\item\label{it:well_posedness_primitive_kone_smoothness_strong_strong}
a.s.\ for all $t\in \R_+$, $x_{\h}\in \Tor^2$, $j\in \{1,2,3\}$ and $i\in \{1,2\}$,
$$
\| \kone(t,x_{\h},\cdot) \|_{L^2(-h,0)} +\|\partial_i \kone(t,\cdot) \|_{L^{2+\delta}(\Tor^2;L^2(-h,0))} \leq M;
$$
\item\label{it:well_posedness_primitive_parabolicity_strong_strong} there exists $\ellip\in (0,2)$ such that, a.s.\ for all $t\in \R_+$, $x\in \Dom$ and $\xi\in \R^d$,
\begin{align*}
\sum_{n\geq 1} \Big(\sum_{j=1}^3 \phi^j_n(t,x) \xi_j\Big)^2\leq \ellip|\xi|^2,
\ \ \text{ and }\ \ 
\sum_{n\geq 1} \Big(\sum_{j=1}^3 \psi^j_n(t,x) \xi_j\Big)^2 \leq \ellip |\xi|^2;
\end{align*}
\item\label{it:nonlinearities_measurability_strong_strong}
for all $n\geq 1$, the maps 
\begin{align*}
&\fv\colon \R_+\times \O\times \R^2\times \R^6\times \R \to \R^2, \quad
\ft\colon \R_+\times \O\times \R^2\times \R^6\times \R \to \R, \\
&\gvn\colon\R_+\times \O\times \R\to \R^2, \quad \hbox{and}\quad
\gtn\colon\R_+\times \O\times \R^2\times \R^6\times \R\to \R
\end{align*}
are $\Progress\otimes \Borel$-measurable; 
\item\label{it:nonlinearities_strong_strong} for all $T\in (0,\infty)$ and $i\in \{1,2\}$,
\begin{align*}
\fv^i(\cdot,0), 
\ft (\cdot,0)&\in L^2((0,T)\times\O\times \Dom), \\
(\gvn^i(\cdot,0))_{n\geq 1}, (\gtn(\cdot,0))_{n\geq 1} &\in  L^2((0,T)\times \O;H^1(\Dom; \ell^2)).
\end{align*}
Moreover, for all $n\geq 1$, $t\in \R_+$,  $x\in \Dom$, $y,y'\in \R^2$, $Y,Y'\in \R^6$ and $z,z'\in \R$,
\begin{align*}
&|F_{v}(t,x,y,z,Y,Z)-F_{v}(t,x,y',z',Y',Z')|\\
&+|F_{\T}(t,x,y,z,Y,Z)-F_{\T}(t,x,y',z',Y',Z')|\\
&\quad\qquad
\lesssim (1+|y|^4+ |y'|^4)|y-y'|+
(1+|z|^{4}+|z'|^{4})|z-z'|\\
&\quad\qquad
+(1+|Y|^{2/3}+|Y'|^{2/3})|Y-Y'|+(1+|Y|^{2/3}+|Y'|^{2/3})|Y-Y'|.
\end{align*}
Finally, a.s.\ for all $t\in\R_+$, $n\geq 1$, the mapping $$\Dom\times \R^2\times \R\ni (x,y,z)\mapsto (\gvn(t,x,y,z),\gtn(t,x,y,z))$$ is continuously differentiable, and a.s.\ for all $k\in\{0,1\}$, $j\in \{1,2,3\}$, $x\in \Dom$, $y,y'\in \R^2$ and $z,z'\in\R$,
\begin{align*}
&\|(\partial_{x_j}^{k}\gvn(t,x,y,z)-\partial_{x_j}^{k}\gvn(t,x,y',z'))_{n\geq 1}\|_{\ell^2}\\
&+\|(\partial_{x_j}^{k}\gtn(t,x,y,z)-\partial_{x_j}^{k}\gtn(t,x,y',z'))_{n\geq 1}\|_{\ell^2}\\
&\qquad \qquad \qquad\qquad
\lesssim (1+|y|^4+ |y'|^4)|y-y'| + (1+|z|^4+ |z'|^4)|z-z'|,\\
&\|(\partial_{y}\gvn(t,x,y)-\partial_{y}\gvn(t,x,y'))_{n\geq 1}\|_{\ell^2}\\
&+\|(\partial_{y}\gtn(t,x,y)-\partial_{y}\gtn(t,x,y'))_{n\geq 1}\|_{\ell^2}\\
&\qquad \qquad \qquad\qquad
\lesssim (1+|y|^2+|y'|^2)|y-y'|
+ (1+|z|^2+|z'|^2)|z-z'|.
\end{align*}
\end{enumerate}
\end{assumption}

Some remarks on Assumption \ref{ass:well_posedness_primitive_double_strong} may be in order. 

\begin{remark}\
\begin{enumerate}
\item As in Remark \ref{r:assump_local_existence_strong_weak}\eqref{it:Holder_continuity_phi}, \eqref{it:well_posedness_primitive_phi_psi_smoothness} and Sobolev embeddings it follows that $(\phi^j_n(t,\cdot))_{n\geq 1}$, $(\psi^j_n(t,\cdot))_{n\geq 1} \in C^{\delta/(3+\delta)}(\Dom;\ell^2)$ uniformly w.r.t.\ $(t,\om)$.
\item 
As in Remark \ref{r:assump_local_existence_strong_weak}\eqref{it:remark_parabolicity}, \eqref{it:well_posedness_primitive_parabolicity_strong_strong} 
is equivalent to the usual \emph{stochastic parabolicity} and therefore \eqref{it:well_posedness_primitive_parabolicity_strong_strong} is optimal in the parabolic setting. 
\item  \eqref{it:nonlinearities_strong_strong} contains the optimal growth assumptions on the nonlinearities which allows to prove local existence for data in $\Hs^1(\Dom)\times H^1(\Dom)$, cf.\ the proof of Theorem \ref{t:local_primitive_strong_strong} below.
\end{enumerate}
\end{remark}

We are in position to define $L^2$-strong-strong solutions to \eqref{eq:primitive_strong_strong}-\eqref{eq:boundary_conditions_strong_strong}. Recall that $\Br_{\ell^2}$
is as in Subsection \ref{ss:set_up}. For notational convenience, we set
\begin{equation}
\label{eq:def_Hn_Hr}
\begin{aligned}
\Hs^{2}_{\n}(\Dom)
&:=\Big\{v \in H^{2}(\Dom;\R^2)\cap \Ls^2(\Dom)\,:\, \partial_3 v(\cdot,-h)=\partial_3 v(\cdot, 0)=0\text{ on }\Tor^2\Big\},\\
\Hr^{2}(\Dom)
&:=\Big\{ \T \in H^{2}( \Dom) \,:\, \partial_3 \T(\cdot,0)+\alpha \T(\cdot,0)= \partial_3 \T (\cdot,-h)=0\text{ on }\Tor^2
\Big\}.
\end{aligned}
\end{equation}

\begin{definition}[$L^2$-strong-strong solutions]
\label{def:sol_strong_strong}
Let Assumption \ref{ass:well_posedness_primitive_double_strong} be satisfied.
\begin{enumerate}[{\rm(1)}]
\item Let $\tau$ be a stopping time, $v:[0,\tau)\times \O \to \Hs_{\n}^2(\Dom)$ and $\T:[0,\tau)\times \O \to \Hr^2(\Dom)$
be stochastic processes. We say that $((v,\T),\tau)$ is called an \emph{$L^2$-local strong-strong solution} to \eqref{eq:primitive_strong_strong}-\eqref{eq:boundary_conditions_strong_strong} if there exists a sequence of stopping times $(\tau_k)_{k\geq 1}$ for which the following hold.
\begin{itemize}
\item $\tau_k\leq \tau$ a.s.\ for all $k\geq 1$ and $\lim_{k\to \infty}\tau_k=\tau$ a.s.;
\item a.s.\ we have $(v,\T)\in L^2(0,\tau_k;\Hs_{\n}^2(\Dom)\times \Hr^2(\Dom))$ and 
\begin{equation}
\begin{aligned}
\label{eq:integrability_strong_strong}
(v\cdot \nabla_{\h}) v+ \w(v)\partial_3 v +\fv (v,\T,\nabla v,\nabla \T)+\Lp (\cdot,v,\T)&\in L^2(0,\tau_k;L^2(\Dom;\R^2)),\\
(v\cdot \nabla_{\h}) \T+ \w(v)\partial_3 \T +\ft (v,\T,\nabla v,\nabla \T)&\in L^2(0,\tau_k;L^2(\Dom)),\\
(\gvn(v,\T))_{n\geq 1 }&\in L^2(0,\tau_k;H^1(\Dom;\ell^2(\N;\R^2))),\\
(\gtn(v,\T))_{n\geq 1 }&\in L^2(0,\tau_k;H^1(\Dom;\ell^2)).
\end{aligned}
\end{equation}
\item a.s.\ for all $k\geq 1$ the following equality holds for all $t\in [0,\tau_k]$:
\begin{align*} 
v(t)-v_0
&=\int_0^t \Delta v(s)+ \p\Big[\nabla_{\h}\int_{-h}^{x_3}(\kone(\cdot,\zeta) \T(\cdot,\zeta))\,d\zeta\\
&\qquad
- (v\cdot \nabla_{\h}) v- \w(v)\partial_3 v  + \fv (v,\T,\nabla v,\nabla \T) +\Lp(\cdot,v,\T)\Big]\,ds\\
&\qquad
+\int_0^t \Big(\one_{[0,\tau_k]}\p[ (\phi_{n}\cdot\nabla) v   +\gvn(v,\T)] \Big)_{n\geq 1}\, d\Br_{\ell^2}(s),\\
\T(t)-\T_0
&=
\int_0^t  \Big[\Delta \T-(v\cdot \nabla_{\h})\T -w(v)\partial_3 \T+ \ft(v,\T,\nabla v,\nabla \T )\Big]\,ds\\
&+
\int_0^t \Big(\one_{[0,\tau_k]}[ (\psi_{n}\cdot\nabla) \T   +\gtn(v,\T )] \Big)_{n\geq 1}\, d\Br_{\ell^2}(s).
\end{align*}
\end{itemize}
\item An $L^2$-local strong-strong solution $((v,\T),\tau)$ to \eqref{eq:primitive_strong_strong}-\eqref{eq:boundary_conditions_strong_strong} is said to be an \emph{$L^2$-maximal strong-strong solution to} \eqref{eq:primitive_strong_strong}-\eqref{eq:boundary_conditions_strong_strong}
if for any other local solution $((v',\T'),\tau')$ we have 
\begin{equation*}
\tau'\leq \tau \  \text{ a.s.\ \  and } \ \ (v,\T)=(v',\T') \ \text{ a.e.\ on }[0,\tau')\times \O.
\end{equation*}
\end{enumerate}
\end{definition}

Note that $L^2$-maximal strong-strong solution are \emph{unique} in the class of $L^2$-local strong-strong solutions by definition. By \eqref{eq:integrability_strong_strong}, the deterministic integrals and the 
stochastic integrals in the above definition are well-defined as Bochner and $H^1$-valued It\^{o}'s integrals, respectively.

\subsection{Statement of the main results}
We begin this subsection by stating our local existence result for \eqref{eq:primitive_strong_strong}-\eqref{eq:boundary_conditions_strong_strong}. To economize the notation, for $k\geq 0$, $m\geq 1$ and a map $f:[0,t)\to H^{k+1}(\Dom;\R^m)$, we set
\begin{equation}
\label{eq:norm_k_strong_strong}
\norm_k(t;f):=\sup_{s\in [0,t)}\|f(s)\|_{H^{k}(\Dom;\R^m)}^2+\int_0^t \|f(s)\|_{H^{k+1}(\Dom;\R^m)}^2\,ds .
\end{equation}

\begin{theorem}[Local existence]
\label{t:local_primitive_strong_strong}
Let Assumption \ref{ass:well_posedness_primitive_double_strong} be satisfied. Then for each 
$$
v_0\in L^0_{\F_0}(\O;\Hs^1(\Dom)), \ \ \text{ and }\ \ \T_0\in L^0_{\F_0}(\O;H^1(\Dom)),
$$
there exist an \emph{$L^2$-maximal strong-strong solution} $((v,\theta) ,\tau)$ to \eqref{eq:primitive_strong_strong}-\eqref{eq:boundary_conditions_strong_strong} such that $\tau>0$ a.s. Moreover the following hold.
\begin{enumerate}[{\rm(1)}]
\item\label{it:path_reg_strong_strong}{\em (Pathwise regularity)} 
There exists a sequence of  stopping times $(\tau_k)_{k\geq 1}$ such that, a.s.\ for all $k\geq 1$, one has $0\leq \tau_k\leq \tau$, $\lim_{k\to \infty}\tau_k=\tau$ and
\begin{align*}
(v,\T)\in L^2(0,\tau_k;\Hs_{\n}^2(\Dom)\times \Hr^2(\Dom))\cap C([0,\tau_k];\Hs^1(\Dom)\times H^1(\Dom)).
\end{align*}  
\item\label{it:blow_up_criterium_strong_strong}{\em (Blow-up criterion)} For all $T\in (0,\infty)$, 
$$
\P\Big(\tau<T,\,  \norm_{1}(\tau;v)+\norm_{1}(\tau;\T) <\infty\Big)=0.
$$
\end{enumerate}
\end{theorem}

Finally, we turn our attention to the existence of \emph{global} strong-strong solutions to \eqref{eq:primitive_strong_strong}-\eqref{eq:boundary_conditions_strong_strong}. To formulate our global existence result, we need the following

\begin{assumption}\
\label{ass:global_primitive_strong_strong}
Let Assumption \ref{ass:well_posedness_primitive_double_strong} be satisfied.
\begin{enumerate}[{\rm(1)}] 
\item\label{it:phi_conditions_strong_strong} For all $n\geq 1$, $x=(x_{\h},x_3)\in \Tor^2\times (-h,0)= \Dom$, 
$t\in \R_+$, $j,k\in \{1,2\}$ and a.s. 
\begin{align*}
\phi_n^j(t,x)\text{ and } \hp^{j,k}_n(t,x)\text{ are independent of $x_3$}.
\end{align*}
\item\label{it:sublinearity_Gforce_strong_strong}
There exist $C>0$ and $\y\in L^0_{\Progress}(\O;L^2_{\loc}([0,\infty);L^2(\Dom)))$ such that, a.s.\ for all $t\in \R_+$, $j\in \{1,2,3\}$, $x\in \Dom$, $y\in \R^2$, $z\in \R$, $Y\in \R^{6}$ and $Z\in \R^3$, 
\begin{align*}
|\fv(t,x,y,z,Y,Z)|+|\ft(t,x,y,z,Y,Z)| &\leq C(\y(t,x)+|y|+|z|\\
&\quad \qquad \quad \ \ +|Y|+|Z|),\\
\|(\gvn (t,x,y,z))_{n\geq 1}\|_{\ell^2}
+\|(\partial_{x_j}\gvn (t,x,y,z))_{n\geq 1}\|_{\ell^2}&\leq C(\y(t,x) + |y|+|z|), \\
\|(\gtn (t,x,y,z))_{n\geq 1}\|_{\ell^2}
+\|(\partial_{x_j}\gtn (t,x,y,z))_{n\geq 1}\|_{\ell^2}&\leq C(\y(t,x) + |y|+|z|), \\
\|(\partial_{y}\gvn(t,x,y,z))_{n\geq 1}\|_{\ell^2}+
\|(\partial_y\gtn(t,x,y,z))_{n\geq 1}\|_{\ell^2}
&\leq C.
\end{align*}
\end{enumerate}
\end{assumption}
\begin{remark}
Assumption \ref{ass:global_primitive_strong_strong} should be compared with Assumption \ref{ass:global_primitive}. 
Note that \eqref{it:phi_conditions_strong_strong}
has already been discussed in Remark \ref{r:global_primitive} below Assumption \ref{ass:global_primitive} and that \eqref{it:sublinearity_Gforce_strong_strong} is symmetric w.r.t.\ $v$ and $\T$. As before Assumption \ref{ass:well_posedness_primitive_double_strong}\eqref{it:well_posedness_primitive_parabolicity_strong_strong} implies the parabolicity condition from Remark \ref{r:global_primitive}\eqref{it:ellipticity_overline_v}. 
\end{remark}
Next we state our main result on global existence to \eqref{eq:primitive_strong_strong}-\eqref{eq:boundary_conditions_strong_strong} in the strong-strong setting.
Recall that $\Hs^2_{\n}$ and $\Hr^2$ have been defined in \eqref{eq:def_Hn_Hr}.

\begin{theorem}[Global existence]
\label{t:global_primitive_strong_strong}
Let Assumption  
\ref{ass:global_primitive_strong_strong} be satisfied, and let
$$
v_0\in L^0_{\F_0}(\O;\Hs^1(\Dom)), \ \ \text{ and }\ \ \T_0\in L^0_{\F_0}(\O;H^1(\Dom)).
$$
Then the $L^2$-maximal strong-strong solution $((v,\T),\tau)$ to \eqref{eq:primitive_strong_strong}-\eqref{eq:boundary_conditions_strong_strong} provided by Theorem \ref{t:local_primitive_strong_strong} is \emph{global in time}, i.e.\ $\tau=\infty$ a.s. In particular
$$
(v,\T)\in L^2_{\loc}([0,\infty);\Hs_{\n}^2(\Dom)\times \Hr^2(\Dom))\cap C([0,\infty);\Hs^1(\Dom)\times H^1(\Dom))\text{ a.s. }
$$
\end{theorem}

The proofs of Theorems \ref{t:local_primitive_strong_strong} and \ref{t:global_primitive_strong_strong} follow
the strategy used in the proof of Theorems \ref{t:local_primitive} and \ref{t:global_primitive}. As in the proof of Theorem \ref{t:local_primitive}, to show Theorem \ref{t:local_primitive_strong_strong} we employ the results in \cite{AV19_QSEE_1,AV19_QSEE_2}. In particular, we need to prove \emph{stochastic maximal $L^2$-regularity} estimates for the linearization of \eqref{eq:primitive_strong_strong}-\eqref{eq:boundary_conditions_strong_strong}.  
Such estimates will be proven in Subsection \ref{ss:smr_2_strong_strong}. The proof of Theorem \ref{t:global_primitive_strong_strong} is similar to the one of Theorem \ref{t:global_primitive} where we have followed the arguments in \cite{CT07,HK16}. In the present case, we need to prove also $L^2_t(H^2_x)$- and $L^{\infty}_t(H^1_x)$-estimates for the temperature $\T$ to apply the blow-up criteria of Theorem \ref{t:local_primitive_strong_strong}\eqref{it:blow_up_criterium_strong_strong}.

\subsection{Stochastic maximal $L^2$-regularity}
\label{ss:smr_2_strong_strong}

In this subsection we study maximal $L^2$-regularity estimates for the linearized problem of \eqref{eq:primitive_strong_strong}-\eqref{eq:boundary_conditions_strong_strong}, see \cite[Section 3]{AV19_QSEE_1} and the references therein for the general theory of stochastic maximal $L^p$-regularity. 

Here, we study maximal $L^2$-regularity estimates for
\begin{equation}
\label{eq:SMR_v_T_strong_strong}
\begin{cases}
\vspace{0.1cm}
\displaystyle{dv -\Big[\Delta v+ \p [\Lpp v+\op_{\kone} \T]\Big]dt =f_v dt 
+ \sum_{n\geq 1} \Big[\p [(\phi_n\cdot \nabla )v]+ g_{n,v}\Big] d\beta_t^n}, & \text{on }\Dom,\\
\displaystyle{d\Temp -\Delta\Temp dt =f_{\T} dt + \sum_{n\geq 1} 
\Big[(\psi_n\cdot \nabla )\Temp+ g_{n,\T}\Big]d\beta_t^n},
& \text{on }\Dom,\\
\partial_3 v(\cdot,0)=\partial_3 v(\cdot,-h)=0, &\text{on }\Tor^2,\\
\partial_3 \T(\cdot,0)+\alpha \T(\cdot,0)=\partial_3 \T(\cdot,-h)=0, &\text{on }\Tor^2,\\
v(0)=0,\qquad \Temp(0)=0,& \text{on }\Dom,
\end{cases}
\end{equation}
where $\Lpp$ and $\op_{\kone}$ are as in \eqref{eq:Lpp_operator} and \eqref{eq:integral_operator_temperature}, respectively. Moreover 
\begin{equation}
\label{eq:f_v_etc_integrability_smr_strong_strong}
\begin{aligned}
(f_v,f_{\T})\in L^2_{\Progress}((0,T)\times \O;\Ls^2\times L^2),\\
((g_{n,v})_{n\geq 1}, (g_{n,\T})_{n\geq 1})\in 
L^2_{\Progress}((0,T)\times \O;\Hs^1(\ell^2)\times H^1(\ell^2)).
\end{aligned}
\end{equation}
Let $\tau$ be a stopping time. Recall that $\Hs_{\n}^2$ and $ \Hr^2$ are as in \eqref{eq:def_Hn_Hr}.
We say that $(v,\T)\in L^2_{\Progress}((0,\tau)\times \O;\Hs_{\n}^2\times \Hr^2)$ is an \emph{$L^2$-strong solution} to \eqref{eq:SMR_v_T_strong_strong} on $[0,\tau]\times \O$ if 
\begin{align*} 
v(t) 
&=\int_0^t \Big[\Delta v(s)+\p\big[\Lpp v -\op_{\kone} \T] +f_v  \Big]\,ds\\
&\qquad
+
\int_0^t \Big(\one_{[0,\tau]} \Big[\p[ (\phi_{n}\cdot\nabla) v]   +g_{v,n}\Big]\Big)_{n\geq 1}\, d\Br_{\ell^2}(s),\\
\T(t)
&=
\int_0^t  \Big[\Delta \T+ f_{\T}\Big]\,ds+
\int_0^t \Big(\one_{[0,\tau]}[(\psi_{n}\cdot\nabla) \T +g_{n,\T}] \Big)_{n\geq 1}\, d\Br_{\ell^2}(s),
\end{align*}
a.s.\ for all $t\in [0,\tau]$. Here $\Br_{\ell^2}$ is as in equation \eqref{eq:Bl2}.

\begin{proposition}[Stochastic maximal $L^2$-regularity]
\label{prop:SMR_2_strong}
Let Assumption \ref{ass:well_posedness_primitive_double_strong}\eqref{it:well_posedness_measurability}--\eqref{it:well_posedness_primitive_parabolicity_strong_strong} be satisfied and let $T\in (0,\infty)$.
Let $f_v,f_{\T},g_{v,n},g_{\T,n}$ be as in \eqref{eq:f_v_etc_integrability_smr_strong_strong}. Then for any stopping time $\tau:\O\to [0,T]$ there exists a unique $L^2$-strong solution to \eqref{eq:SMR_v_T_strong_strong} on $[0,\tau]\times \O$ such that 
$$
(v,\T)
\in L^2_{\Progress}((0,\tau)\times \O;\Hs_{\n}^2\times \Hr^2)\cap  L^2_{\Progress}(\O;C([0,\tau];\Hs^1\times H^1))
$$ 
and moreover for any $L^2$-strong solution $(v,\T)$ to \eqref{eq:SMR_v_T_strong_weak} on $[0,\tau]\times \O$ we have
\begin{equation}
\label{eq:SMR_2_estimate_strong_strong}
\begin{aligned}
&\|(v,\T)\|_{L^2((0,\tau)\times \O;\Hs^2\times H^2)}+
\|(v,\T)\|_{L^2(\O;C([0,\tau];\Hs^1\times H^1))}\\
&\qquad \qquad \qquad
\lesssim \|(f_{v},f_{\T})\|_{L^2((0,\tau)\times \O;\Ls^2\times L^2)}\\
&\qquad \qquad \qquad
+\|((g_{n,v})_{n\geq 1},(g_{n,\T})_{n\geq 1})\|_{L^2((0,\tau)\times \O;\Hs^1(\ell^2)\times H^1(\ell^2))}
\end{aligned}
\end{equation}
where the implicit constant is independent of $f_{v},f_{\T}, (g_{n,v})_{n\geq 1},(g_{n,\T})_{n\geq 1}$ and $\tau$. 
\end{proposition}

By \cite[Proposition 3.9 and 3.12]{AV19_QSEE_2}, Proposition \ref{prop:SMR_2_strong} also proves maximal $L^2$-estimates with non-trivial initial data and also the starting time $0$ can be replaced by any stopping time $\tau$ with values in $[0,T]$ where $T\in (0,\infty)$.

For all $U=(v,\T)\in \Hs^2_{\n}\times \Hr^2$, we set a.s.\ for all $t\in \R_+$,
\begin{equation}
\label{eq:choice_AB_strong_strong}
A (t)U:=
\begin{bmatrix}
-\Delta v -\p \big[\Lpp v +\op_{\kone} \T \big]\\
-\Delta \T
\end{bmatrix},
\ \ 
\text{ and } \ \ 
B_{n}(t) U:=
\begin{bmatrix}
\p\big[(\phi_n(t)\cdot \nabla )v\big]\\
(\psi_n(t) \cdot \nabla )\T
\end{bmatrix}.
\end{equation}
Then using the notation introduced in \cite[Section 3]{AV19_QSEE_1}, Proposition \ref{prop:SMR_2} shows $(A,(B_n)_{n\geq 1})\in \mathcal{SMR}^{\bullet}_{2}(T)$.

\begin{proof}[Proof of Proposition \ref{prop:SMR_2_strong}]
The proof is similar to the one of Proposition \ref{prop:SMR_2}. As in the latter case we only consider the case $\hp_{n}^{\ell,m}\equiv 0$. Thus, it is enough to prove a priori estimates (uniform in $\lambda\in [0,1]$) for strong solutions of the following problem
\begin{equation*}
\begin{cases}
\displaystyle{dv -\Big[\Delta v+\lambda \p[ \op_{\kone} \T]\Big]dt =f_v dt 
+ \sum_{n\geq 1} \p\Big[\lambda (\phi_n\cdot \nabla )v+ g_{n,v}\Big] d\beta_n}, & \text{on }\Dom,\\
\displaystyle{d\Temp -\Delta \Temp dt =f_{\T} dt + \sum_{n\geq 1} 
\Big[\lambda(\psi_n\cdot \nabla )\Temp+ g_{n,\T}\Big]d\beta_t^n},
& \text{on }\Dom,\\
v(0)=0,\qquad \Temp(0)=0,& \text{on }\Dom.
\end{cases}
\end{equation*}
As in the proof of Proposition \ref{prop:SMR_2}, we first estimate the temperature $\T$ and then we use this estimate for estimating $v$.
Using also \eqref{eq:estimate_theta_term_in_v_variable} one can see that the argument in Step 2 of Proposition \ref{prop:SMR_2} can be reproduced almost verbatim to get an estimate for $v$. 
Thus it remains to prove the estimate for the temperature. Let us remark that, if $\alpha=0$ in \eqref{eq:boundary_conditions_strong_strong} (i.e.\ $\T$ also satisfied Neumann boundary conditions), then the estimate for the temperature can be performed again as in Step 2 Proposition \ref{prop:SMR_2}. In the case $\alpha\neq0$, we slightly modify that argument.

Let $\tau:\O\to [0,T]$ be a stopping time where $T\in (0,\infty)$. It remains to show the existence of $C_0>0$ independent of $(\lambda,f_{\T},g_{\T},\tau)$ such that any strong solution $\T\in L^2_{\Progress}((0,\tau)\times \O;\Hr^2)\cap L^2_{\Progress}(\O; C([0,\tau];H^1))$ to 
\begin{equation}
\label{eq:temperature_equation_SMR_strong_strong}
\begin{cases}
\displaystyle{
d\Temp -\Delta \Temp dt =f_{\T} dt + \sum_{n\geq 1} 
\Big[\lambda(\psi_n\cdot \nabla )\Temp+ g_{n,\T}\Big]d\beta_t^n,}&  \text{ on }\Dom,\\
\T(0)=0, &\text{ on }\Dom,
\end{cases}
\end{equation}
satisfies the estimate
\begin{equation}
\label{eq:estimate_temperature_step_2_SMR_strong}
\E\| \T\|_{L^2(0,\tau;H^2)}^2
\leq C_0 \E\|f_{\T}\|_{L^2(0,\tau;L^2)}^2+C_0 \E\|g_{\T}\|_{L^2(0,\tau;H^1(\ell^2))}^2.
\end{equation}
Let us begin by noticing that, due to Step 1 of Proposition \ref{prop:SMR_2}, there exists $C_0'>0$ independent of $(\lambda,f,g,\tau)$ such that 
\begin{equation}
\begin{aligned}
\label{eq:estimate_temperature_step_2_SMR_weak}
\E\Big[\sup_{t\in [0,\tau]}\|\T(t)\|_{L^2}^2\Big]
&+\E\| \T\|_{L^2(0,\tau;H^1)}^2\\
&\leq C_0' \E\|f_{\T}\|_{L^2(0,\tau;L^2)}^2+C_0 \E\|g_{\T}\|_{L^2(0,\tau;H^1(\ell^2))}^2.
\end{aligned}
\end{equation}
 Following Step 1 of Proposition \ref{prop:SMR_2} we set $\T^{\tau}:=\T(\tau\wedge \cdot)$ and we need an approximation argument to estimate the gradient of $\T$. For $k\geq 1$, we set $\e_k :=k(1+k+\Delta_{{\rm R}})^{-1}$ as in  \eqref{eq:def_e_k} where $\Delta_{{\rm R}}: \Hr^{2}\subseteq L^2\to L^2$ is the Laplace operator with Robin boundary conditions. Set $\T_k:=\e_k\T$ and $\T^{\tau}_k:=\e_k \T^{\tau}$. By \eqref{eq:temperature_equation_SMR_strong_strong}, $\T^{\tau}_k$ is given by
\begin{align*}
\Temp_k^{\tau}(t)=\int_0^{ t}\one_{[0,\tau]}(\Delta \Temp_k +\e_k f_{\T}) ds + \sum_{n\geq 1} \int_0^{t}\one_{[0,\tau]}
\e_k\Big[\lambda(\psi_n\cdot \nabla )\Temp+ g_{n,\T}\Big]d\beta_s^n,
\end{align*}
a.s.\ for all $t\in [0,T]$. Here we also used that $\e_k$ and $\Delta$ commute on $\Hr^2(\Dom)$.

Set $f_k:=\e_k f_{\T}$ and $g_{n,k}:=\e_k g_{\T,n}$. The idea is to apply  It\^{o}'s formula to the functional $\T_k\mapsto \Ff_{\alpha}(\varphi):= \|\nabla \varphi\|^2_{L^2(\Dom)}+\alpha\|\varphi(\cdot,0)\|^2_{L^2(\Tor^2)}$. Note that 
$$
\Ff_{\alpha}'(\varphi)\psi=2\int_{\Dom} \nabla \varphi \cdot \nabla \psi \, dx+2\alpha \int_{\Tor^2} \varphi(\cdot,0) \psi(\cdot,0)\,dx, \ \ \text{ for all }u,v\in H^{1}.
$$
By It\^{o}'s formula applied to $\Ff_{\alpha}$, we have a.s.\ for all $t\in [0,T]$ and $k\geq 1$,
\begin{equation}
\label{eq:identity_F_alpha}
\begin{aligned}
&\Ff_{\alpha}(\T^{\tau}_k(t))
=2\int_0^{t}\one_{[0,\tau]} \Ff_{\alpha}'(\Delta \T_k+ f_k)\T_k\,ds \\
&+\sum_{n\geq 1} \int_{0}^t \one_{[0,\tau]}  \Ff_{\alpha}(\e_k [\lambda(\psi\cdot \nabla) \T]+g_{n,k}) \,ds \\
&+2\sum_{n\geq 1}\int_0^{t}\one_{[0,\tau]} \Ff_{\alpha}'(\e_k [\lambda(\psi\cdot \nabla) \T]+g_{n,k})\T_k\, dw_s^n
=:I_t^k+II_t^k+III_t^k. 
\end{aligned}
\end{equation}
Next, we want to take $k\to \infty$ in the previous identity. To this end, let us recall that the trace operator $H^{1/2+s}\ni f\mapsto f(\cdot,0)\in L^2(\Tor^2)$ is bounded for all $s>0$. Since $\T\in L^2((0,\tau)\times \O;H^2)$ and $\e_k\to I $ strongly in $H^{1}=\Do(\Delta_{{\rm R}}^{1/2})$, it follows that, as $k\to \infty$,
\begin{align*}
II_t^k \to \sum_{n\geq 1} \int_{0}^t \one_{[0,\tau]}  \Ff_{\alpha}(\lambda(\psi\cdot \nabla) \T+g_{n}) \,ds=:II_t,&\\
III_t^k\to 2\sum_{n\geq 1}\int_0^{t}\one_{[0,\tau]} \Ff_{\alpha}'(\lambda(\psi\cdot \nabla) \T+g_{n})\T\, dw_s^n=:III_t,&
\end{align*}
where both limits take place in $L^1(\O;C([0,T]))$.
To analyze $I_t^k$, note that $\Delta \T_k+f_k$ and $\T_k\in \Hr^2$. In particular $\partial_3\T_k(\cdot,0)=-\alpha \T_k(\cdot,0)$ and $\partial_3 \T_k (\cdot,-h)=0$ both on $\Tor^2$. Integrating by parts, one gets a.s.\ for all $t\in [0,T]$,
\begin{align*}
I_t^k= -2\int_0^{t} \one_{[0,\tau]}\int_{\Dom} (\Delta \T_k+ f_k)\Delta \T_k \,dx ds.
\end{align*}
Since $\T\in L^2((0,\tau)\times \O;H^2)$ and $f\in L^2((0,T)\times \O;L^2)$, as $k\to \infty$
$$
I_t^k \to -2\int_0^{t} \one_{[0,\tau]}\int_{\Dom}\Big( |\Delta \T|^2+ f\Delta \T \Big)\,dx ds=:I_t \quad \text{ in }L^1(\O;C([0,T])).
$$
Taking $k\to \infty$ and afterwards expectations in \eqref{eq:identity_F_alpha}, we have
\begin{equation}
\begin{aligned}
\label{eq:L_2_inequality_Robin_conditions}
&\E \|\nabla \T^{\tau}(T,\cdot)\|_{L^2}^2+ 2\E\int_0^{ \tau}  \|\Delta \T\|_{L^2}^2\,ds \\
&\quad \leq 2 \E \int_0^{ \tau}  \int_{\Dom}|\Delta \T f|\, dx ds+|\alpha| \E\|\T^{\tau}(T,\cdot,0)\|_{L^2(\Tor^2)}^2\\
&\quad + \sum_{n\geq 1}
\E\int_{0}^{\tau}  \|\lambda(\psi_n\cdot \nabla) \T+g_{n}\|_{L^2}^2 \,ds\\
&\quad +|\alpha|\sum_{n\geq 1}\E
\int_{0}^{\tau}  \Big\|\lambda[(\psi_n\cdot \nabla) \T](\cdot,0)+g_{n}(\cdot,0)\Big\|_{L^2(\Tor^2)}^2 \,ds
\end{aligned}
\end{equation}
where we used that $\tau\leq T$ a.s.\ and $\E[III_T]=\E[III_0]=0$.

Let us estimate each term on the right hand side of \eqref{eq:L_2_inequality_Robin_conditions} separately. Fix $\ellip<\ellip'<2$ and let $\varepsilon>0$ to be chosen later. Here $\ellip $ is as in Assumption \ref{ass:well_posedness_primitive_double_strong}\eqref{it:well_posedness_primitive_parabolicity}. By Cauchy-Schwartz and Young's inequality
$$
 \E \int_0^{ \tau}  \int_{\Dom}|\Delta \T f|\, dx ds\leq \varepsilon  \E\int_0^{ \tau}  \|\Delta \T\|_{L^2}^2\,ds
+C_{\varepsilon} \E\int_0^{ \tau}  \|f\|_{L^2}^2\,ds.
$$
Repeating the argument in \eqref{eq:estimate_Ito_correction_v_smr}, one can check that, for all $\lambda\in [0,1]$, 
\begin{equation}
\label{eq:estimate_robin_similar_neumann}
\E\int_0^{\tau}\sum_{n\geq 1}\|\nabla [\lambda(\psi_n\cdot \nabla) \T+g_{n}]\|_{L^2}^2 \,ds
\leq \E\int_0^{\tau} \Big( \ellip' \|\Delta \T\|_{L^2}^2+C_{\ellip} \|g_n\|^2_{H^1(\ell^2)}\Big)\,ds,
\end{equation}
where one also uses Lemma \ref{l:Kadlec_formula} with $\beta=\alpha$.  Thus, by the above mentioned boundedness of the trace operator $H^{1/2+s}\ni f\mapsto f(\cdot,0)\in L^2(\Tor^2)$ for $s>0$ and \eqref{eq:estimate_robin_similar_neumann}, 
\begin{align*}
&\sum_{n\geq 1}\E
\int_{0}^{\tau}  \Big\|\lambda[(\psi_n\cdot \nabla) \T](\cdot,0)+g_{n}(\cdot,0)\Big\|_{L^2(\Tor^2)}^2 \,ds\\
&\leq 
\varepsilon \E\int_0^{ \tau}  \|\Delta \T\|_{L^2}^2\,ds
+C_{\varepsilon,\ellip'} \E\int_0^{ \tau}  (\|g\|_{H^1(\ell^2)}^2+ \|\T\|_{L^2}^2)\,ds.
\end{align*}

Analogously, $\E\|\T^{\tau}(T,\cdot,0)\|_{L^2(\Tor^2)}^2\leq \varepsilon \E\|\nabla \T^{\tau}(T,\cdot)\|_{L^2}^2+ C_{\varepsilon}\E\|\T^{\tau}(T,\cdot)\|_{L^2}^2$.

Let $\varepsilon>0$ be such that $\ellip+2\varepsilon(1+|\alpha|)<2$.
Using the previous estimates  in \eqref{eq:L_2_inequality_Robin_conditions},
\begin{equation*}
\E\int_0^{ \tau}  \|\Delta \T\|_{L^2}^2\,ds\lesssim \E\int_0^{ \tau} 
 (\|f\|_{L^2}^2+\|g\|_{H^1(\ell^2)}^2)\,ds+ \E\Big[\sup_{t\in [0,\tau]}\|\T(t)\|_{L^2}^2\Big]
\end{equation*}
where we used that $\T^{\tau}=\T(\cdot\wedge\tau)$ and the implicit constant is independent of $(\lambda,f,g,\tau)$. By \eqref{eq:estimate_temperature_step_2_SMR_weak}, the previous inequality yields \eqref{eq:estimate_temperature_step_2_SMR_strong} as desired.
\end{proof}

\subsection{Proof of Theorem \ref{t:local_primitive_strong_strong}}
As in Subsection \ref{ss:proof_local_strong_weak}, the proof is based on the results in \cite{AV19_QSEE_1,AV19_QSEE_2}, and we reformulate \eqref{eq:primitive_strong_strong}-\eqref{eq:boundary_conditions_strong_strong} as a stochastic evolution equation on $X_0$
of the form
\begin{equation}
\label{eq:abstract_formulation_strong_strong}
\begin{cases}
d U+ A (\cdot)U\,dt=F(\cdot,U)dt + (B (\cdot)U +G(\cdot,U))d\Br_{\ell^2}(t),\\
U(0)=(v_0,\T_0).
\end{cases}
\end{equation}
Here $X_0=\Ls^2\times L^2$, $X_1=\Hs^{2}_{\n}\times \Hr^2$ (see \eqref{eq:def_Hn_Hr}), $\Br_{\ell^2}$ is the $\ell^2$-cylindrical Brownian motion associated to $(\beta^n)_{n\geq 1}$ (see equation \eqref{eq:Bl2}), $(A,B)$ are as in \eqref{eq:choice_AB_strong_strong} and for $U=(v,\T)\in X_1$
\begin{align*}
F(\cdot,U)
&:=\begin{bmatrix}
\p[ (v\cdot\nabla_{\h}) v + w(v)\cdot \partial_3 v +\fv(\cdot,v,\T,\nabla v,\nabla\T)+\Lpg(\cdot,v,\T)]\\
 (v\cdot\nabla_{\h})\T+w(v)\partial_3 \T + \ft(\cdot,v,\T)
\end{bmatrix},\\
G(\cdot,U)
&:=\begin{bmatrix}
 (\p[\gvn(\cdot,v)])_{n\geq 1}\\
  (\gtn(\cdot,v,\T))_{n\geq 1}
\end{bmatrix},
\end{align*}
where $w(v)$ is as in \eqref{eq:def_w_v_strong_strong} and
\begin{equation*}
\Lpg(\cdot,v):=\Big(\sum_{n\geq 1} \sum_{m=1}^2 \hp_n^{\ell,m} (t,x)
\big(\q[\gvn(\cdot,v,\T)]\big)^m\Big)_{\ell=1}^2.
\end{equation*}

It is easy to see that $((v,\T),\tau)$ is an $L^2$-maximal (resp.\ -local) solution to \eqref{eq:primitive_strong_strong}-\eqref{eq:boundary_conditions_strong_strong} (see Definition \ref{def:sol_strong_strong}) if and only if $(U,\tau)$ where $U:=(v,\T)$ is an $L^2$-maximal (resp.\ -local) solution to \eqref{eq:abstract_formulation_strong_strong} in the sense of \cite[Definition 4.4]{AV19_QSEE_1}.

Theorem \ref{t:local_primitive_strong_strong} can be proven as Theorem \ref{t:local_primitive} and therefore we only give a sketch of its proof.

\begin{proof}[Proof of Theorem \ref{t:local_primitive_strong_strong} -- Sketch]
By the above discussion and Proposition \ref{prop:SMR_2_strong}, the existence of an $L^2$-maximal solution to \eqref{eq:primitive_strong_strong}-\eqref{eq:boundary_conditions_strong_strong} satisfying \eqref{it:path_reg_strong_strong} follows from \cite[Theorem 4.8]{AV19_QSEE_1} provided assumptions (HF) and (HG) in \cite[Section 4]{AV19_QSEE_1} hold. To check those assumptions, one can argue as in Steps 1--3 of Theorem \ref{t:local_primitive}. 

Finally, \eqref{it:blow_up_criterium_strong_strong} 
 follows from \cite[Theorem 4.11]{AV19_QSEE_2} and Proposition \ref{prop:SMR_2_strong}.
\end{proof}

\subsection{Proof of Theorem \ref{t:global_primitive_strong_strong}}
The strategy of the proof is similar to the one used for Theorem \ref{t:global_primitive}. As in the proof of Theorem \ref{t:global_primitive}, the global existence result of Theorem \ref{t:global_primitive_strong_strong} is a consequence of the blow-up criterion in Theorem \ref{t:local_primitive_strong_strong}\eqref{it:blow_up_criterium_strong_strong} and the following energy estimate.
Recall that $\norm_k$ has been defined in \eqref{eq:norm_k_strong_strong}.

\begin{proposition}[Energy estimate]
\label{prop:energy_estimate_primitive_strong_strong}
Let Assumptions \ref{ass:well_posedness_primitive_double_strong} and \ref{ass:global_primitive_strong_strong} be satisfied. Let $T\in (0,\infty)$.
Assume that $v_0\in L^{\infty}_{\F_0}(\O;\Hs^1)$ and $\T_0\in L^{\infty}_{\F_0}(\O;H^1)$. Let $((v,\T),\tau)$ be the $L^2$-maximal strong-strong solution to \eqref{eq:primitive_strong_strong} provided by Theorem \ref{t:local_primitive_strong_strong}.
Then there exists a sequence of stopping times $(\mu_k)_{k\geq 1}$ with values in $[0,T]$ such that $\mu_n\leq \tau$ a.s.\ for all $k\geq 1$ and for which the following hold.
\begin{enumerate}[{\rm(1)}]
\item $\P(\mu_k=\tau\wedge T)\to 1$ as $k\to \infty$;
\item For each $k\geq 1$ there exists $C_{k,T}>0$ (possibly depending on $v_0,v,\T_0,\T$) such that 
$$
\E[ \norm_{1}(\mu_k;v)] + 
\E[\norm_1(\mu_k;\T) ]\leq C_{k,T}\Big(1+\E\|v_0\|_{H^1}^2+\E\|\T_0\|_{H^1}^2\Big).
$$
\end{enumerate}
\end{proposition}

Theorem \ref{t:global_primitive_strong_strong} follows from Proposition \ref{prop:energy_estimate_primitive_strong_strong} in the same way as Theorem \ref{t:global_primitive} follows from Proposition \ref{prop:energy_estimate_primitive}. To avoid repetitions, we only give a sketch of the proof of Proposition \ref{prop:energy_estimate_primitive_strong_strong}, since it is an easy extension of the one given for Proposition \ref{prop:energy_estimate_primitive}. 

\begin{proof}[Proof of Proposition \ref{prop:energy_estimate_primitive_strong_strong} -- Sketch]
Let $T\in (0,\infty)$ be fixed.
Reasoning as in the proof of Lemma \ref{l:L_2_estimate_I}, one can check that there exists $C_T>0$ independent of $v_0,v,\T_0,\T$ such that 
\begin{equation}
\label{eq:norm_0_v_T_strong_strong}
\E[\norm_0(\tau;v)]
+\E[\norm_0(\tau;\T)]
\leq C_T\Big(1+\E\|v_0\|_{L^2}^2+\E\|\T_0\|_{L^2}^2\Big).
\end{equation}
The estimate
\eqref{eq:norm_0_v_T_strong_strong} and Assumption \ref{ass:global_primitive_strong_strong}\eqref{it:sublinearity_Gforce_strong_strong} yield 
\begin{equation}
\begin{aligned}
\label{eq:integrability_inhomogeneity_strong_strong}
\nabla_{\h} \int_{-h}^{\cdot}  (\kone(\cdot,\zeta)\T(\cdot,\zeta))\,d\zeta+ \fv(\cdot,v,\T,\nabla v,\nabla \T)
&\in L^2((0,\tau)\times \O;L^2),\\
(\gvn(\cdot,v,\T))_{n\geq 1}&\in L^2((0,\tau)\times \O;H^1(\ell^2)).
\end{aligned}
\end{equation}
Due to
\eqref{eq:integrability_inhomogeneity_strong_strong}, one can check that the proof of Proposition \ref{prop:energy_estimate_primitive} can be repeated also in the strong-strong setting. In particular, there exists a sequence of stopping times $(\mu_k')_{k\geq 1}$  with values in $[0,T]$ such that, a.s.\ for all $k\geq 1$, one has $\mu_k'\leq \tau$, $\lim_{k\to \infty}\P(\mu_k'=\tau)=1$ and 
\begin{equation}
\label{eq:norm_1_norm_0_estimate_strong_strong}
\E[\norm_1(\mu_k';v)]
+\E[\norm_0(\mu_k';\T)]
\leq c_{k,T}\Big(1+\E\|v_0\|_{H^1}^2+\E\|\T_0\|_{L^2}^2\Big).
\end{equation}
where $c_{k,T}$ is a constant (possibly) depending on $v_0,v$ and  $k\geq 1$.

It remains to prove the estimate for $\norm_1(\tau;\T)$. By \eqref{eq:norm_1_norm_0_estimate_strong_strong} and the above choice of $(\mu_k')_{k\geq 1}$, there exists a sequence of constants $(R_k)_{k\geq 1}$ for which the sequence of stopping times $(\mu_k)_{k\geq 1}$ defined as
\begin{align*}
\mu_k:=
\inf\big\{t\in [0,\mu_k')&\,:\,\|v(t)\|_{H^1}+\| v\|_{L^2(0,t;H^2)} \\
&+\|\T(t)\|_{L^2}+\| \T\|_{L^2(0,t;H^1)}+\|\y\|_{L^2(0,t;L^2)} \geq R_k\big\}\wedge T, 
\end{align*}
 where $\inf\emptyset:=\mu_k'$ and $\y$ is as in Assumption \ref{ass:global_primitive_strong_strong}\eqref{it:sublinearity_Gforce_strong_strong}, satisfies
$$
\lim_{k\to \infty}\P(\mu_k=\tau )=1.
$$
By \eqref{eq:norm_1_norm_0_estimate_strong_strong} and $\mu_k'\leq \mu_k$, to conclude the proof, it remains to prove that for each $k\geq 1$ there exists $C_{k,T}>0$ such that 
\begin{equation}
\label{eq:norm_1_T_estimate_strong_strong}
\E[\norm_1(\mu_k;\T)]\leq C_{k,T}\Big(1+\E\|\T_0\|_{L^2}^2\Big).
\end{equation}

As usual, to prove \eqref{eq:norm_1_T_estimate_strong_strong} we employ a localization argument. For each $j \geq 1$, let
$$
\tau_j :=\inf\big\{t\in [0,\tau)\,:\, \|\T(t)\|_{H^1}+ \|\T\|_{L^2(0,t;H^2)}\geq j\big\}\wedge T, \ \text{ where }\ \inf\emptyset:=\tau.
$$
By the stochastic Gronwall's lemma \cite[Lemma 5.3]{GHZ_09}, to prove \eqref{eq:norm_1_T_estimate_strong_strong} it remains to show that there exists $C_k>0$ such that for any $j\geq 1$ and any stopping times $\eta,\xi$ satisfying $0\leq\eta\leq \xi\leq \tau_j \wedge \mu_k$ a.s.\ one has
\begin{equation}
\label{eq:claim_T_estimate_strong}
\begin{aligned}
\E\Big[\sup_{t\in [\eta,\xi]}\|\T(t)\|_{H^1}^2 \Big] 
&+\E \|\T\|_{L^2(\eta,\xi;H^2)}^2\\
&\leq C_{k}\Big( 1+\E\|\T(\eta)\|_{H^1}^2+ \E\int_{\eta}^{\xi} (1+\|v\|_{H^2}^2)\|\T(s)\|_{H^1}^2\,ds \Big).
\end{aligned}
\end{equation}
Note that \cite[Lemma 5.3]{GHZ_09} is applicable since $\int_0^{\mu_k}\|v\|_{H^2}^2\,ds \leq R_k$ a.s.

To prove \eqref{eq:claim_T_estimate_strong}, we collect some useful facts.
By Proposition \ref{prop:SMR_2_strong} and \eqref{eq:primitive_strong_strong}, there exists a constant $C_0$ independent of $v_0,v,\T,\T_0,\eta,\xi,j,k$ such that 
\begin{equation}
\label{eq:T_estimate_SMR_global_C_0}
\begin{aligned}
\E\Big[\sup_{t\in [\eta,\xi]}\|\T(t)\|_{H^1}^2 \Big] 
&+\E \|\T\|_{L^2(\eta,\xi;H^2)}\leq C_0  \E\|\T(\eta)\|_{L^2}^2\\
& 
+C_0 \E\|\ft(\cdot,v,\T,\nabla v,\nabla\T)\|_{L^2(\eta,\xi;L^2)}^2\\
&
+C_0\E\|(\gtn(\cdot,v,\T))_{n\geq 1}\|_{L^2(\eta,\xi;H^1(\ell^2))}^2\\
& 
+ C_0(\E \|w(v)\partial_3 \T\|_{L^2(\eta,\xi;L^2)}^2+ \E \|(v\cdot \nabla_{\h}) \T\|_{L^2(\eta,\xi;L^2)}^2).
\end{aligned}
\end{equation}
It remains to estimate each term on the right hand side of the previous inequality.
Let us begin by noticing that \eqref{eq:norm_0_v_T_strong_strong} and Assumption \ref{ass:global_primitive_strong_strong}\eqref{it:sublinearity_Gforce_strong_strong} imply
\begin{equation}
\begin{aligned}
\label{eq:T_estimate_forcing_1}
\|\ft(\cdot,v,\T,\nabla v,\nabla \T)\|_{L^2((0,\tau)\times \O;L^2)}&\leq C_k,\\
\|(\gtn(\cdot,v,\T))_{n\geq 1}\|_{ L^2((0,\tau)\times \O;H^1(\ell^2))}&\leq C_k,
\end{aligned}
\end{equation}
where $C_k$ is independent of $\eta,\xi$ and $j\geq 1$.

Thus it remains to estimate the last two terms on the RHS of \eqref{eq:T_estimate_SMR_global_C_0}. To this end, note that by \eqref{eq:estimate_w_strong_weak} and the fact that $\xi\leq \mu_k$ a.s., 
\begin{equation}
\begin{aligned}
\label{eq:w_v_estimate_T_proof}
\|w(v)\|_{L^{\infty}(-h,0;L^4(\Tor^2))}
&\lesssim \|v\|_{H^{3/2}}\\
&\lesssim \|v\|_{H^{1}}^{1/2}\|v\|_{H^2}^{1/2}
\leq R_k^{1/2} \|v\|_{H^2}^{1/2},
\text{ a.e.\ on }[\eta,\xi]\times \O.
\end{aligned}
\end{equation}
Moreover, by the Sobolev embedding $H^1\embed L^6$, we get
\begin{equation}
\sup_{t\in [0,\mu_k]}\|v(t)\|_{L^6}\lesssim R_k\quad \text{a.s.\ }
\label{eq:v_L_6_estimate_T_proof}
\end{equation}
Using the interpolation inequality $\|\nabla \T\|_{L^{2}(-h,0;L^4(\Tor^2))}\lesssim \|\T\|_{H^{3/2}}\lesssim \|\T\|_{H^1}^{1/2}\|\T\|_{H^2}^{1/2}$, we have
\begin{equation}
\begin{aligned}
\label{eq:T_estimate_forcing_2}
\E \|w(v)\partial_3 \T\|_{L^2(\eta,\xi;L^2)}^2
&\stackrel{\eqref{eq:w_v_estimate_T_proof}}{\lesssim_{R_k}}
 \E\int_{\eta}^{\xi} \|v\|_{H^2}\|\nabla \T\|_{L^{2}(-h,0;L^4(\Tor^2))}^2\,ds\\
&\leq  C_{k}\E\int_{\eta}^{\xi} \|v\|_{H^2}^2 \| \T\|_{H^1}^2\,ds + \frac{1}{4C_0} \E\|\T\|_{L^2(\eta,\xi;H^2)}^2.
\end{aligned}
\end{equation}
Finally, by \eqref{eq:v_L_6_estimate_T_proof} and the interpolation inequality $\|\nabla \T\|_{L^{3}}\lesssim \|\T\|_{H^1}^{1/2}\|\T\|_{H^2}^{1/2}$, we have
\begin{equation}
\begin{aligned}
\label{eq:T_estimate_forcing_3}
\E\|(v\cdot \nabla_{\h})\T\|_{L^2(\eta,\xi;L^2)}^2
&\leq \E\int_{\eta}^{\xi}\|v\|_{L^6}^2\|\T\|_{H^{1,3}}^2\,ds\\
&\leq C_{k} \E\int_{\eta}^{\xi} \|\T\|_{H^1}^2\,ds + \frac{1}{4C_0} \E\|\T\|_{L^2(\eta,\xi;H^2)}^2. 
\end{aligned}
\end{equation}
Using the estimates \eqref{eq:T_estimate_forcing_1}, \eqref{eq:T_estimate_forcing_2} and \eqref{eq:T_estimate_forcing_3} in \eqref{eq:T_estimate_SMR_global_C_0}, one sees that \eqref{eq:claim_T_estimate_strong} follows.
As mentioned above, one can check that the stochastic Gronwall's lemma in \cite[Lemma 5.3]{GHZ_09} and \eqref{eq:claim_T_estimate_strong} already imply \eqref{eq:norm_1_T_estimate_strong_strong}. 
\end{proof}

\section{Inhomogeneous viscosity and conductivity}
\label{s:variable_viscosity}
In this section we show how the results of the Sections \ref{s:strong_weak} and \ref{s:strong_strong} can be extended to the case where
the Laplacians  
$
\Delta 
$
appearing in the first two equations in \eqref{eq:primitive_intro_v} are replaced by  elliptic second order differential operators with $(t,\om,x)$-dependent coefficients which will be denoted by $\Lv$ and $\LT$, respectively. In the strong-weak setting, to accommodate the weak setting for the $\T$-equation, $\LT$ is chosen to be a differential operator in divergence form.

Differential operators with $(t,\om,x)$-dependent coefficients can be useful 
to model inhomogeneous viscosity of the fluid and/or thermal conductivity. Moreover, 
if one considers the stochastic primitive equations with transport noise in Stratonovich form (see Section \ref{s:Stratonovich} below), then differential operators with $(t,\om,x)$-dependent coefficients appears naturally, and the principal part of such operators have coefficients 
\begin{equation}
\label{eq:coefficients_Stratonovich_example}
a_{\phi}^{i,j}:=\delta^{i,j}+\frac{1}{2}\sum_{n\geq 1} \phi^j_n\phi^i_n, 
\quad \text{ and }\quad
a_{\psi}^{i,j}:=\delta^{i,j}+\frac{1}{2}\sum_{n\geq 1} \psi^j_n\psi^i_n,
\end{equation}
respectively. Here $\delta^{i,j}$ is Kronecker's delta and $i,j\in \{1,2,3\}$. 
Let us stress that the Stratonovich formulation is often used in the physical literature (see e.g.\ \cite{HL84,Franzke14,W_thesis}) and can be studied using It\^{o}'s calculus by translating the Strotonivch integration into an It\^{o}'s ones plus additional correction terms. Such correction terms lead to consider the system \eqref{eq:primitive_intro_v} modified to include  variable viscosity and conductivity given by \eqref{eq:coefficients_Stratonovich_example}. 

This section is organized as follows. In Subsection \ref{ss:variable_viscosity_strong_weak} we state our main results on the stochastic primitive equations in the strong-weak setting and in Subsection \ref{ss:proofs_variable_viscosity_strong_weak} we provide the corresponding proofs. For brevity, we do not state any result in the strong-strong setting. However, the latter situation can be handled by extending the estimates for the strong-weak setting as we shown in Section \ref{s:strong_strong} for the case of homogeneous viscosity and/or conductivity, see Remark \ref{r:variable_viscosity_strong_strong} below for more comments.

\subsection{The strong-weak setting}
\label{ss:variable_viscosity_strong_weak}
Here we extend the results of Section \ref{s:strong_weak} to the case of variable viscosity and/or conductivity. More precisely, here we consider the primitive equations with transport noise in the strong-weak setting:
\begin{equation}
\label{eq:primitive_weak_strong_2}
\begin{cases}
\displaystyle{d v -\p[\Lv v]\, dt=\p\Big[ -(v\cdot \nabla_{\h}) v- \w(v)\partial_3 v+ \Lp (\cdot,v)} \\
\qquad\qquad \qquad \qquad\ \ \ \quad \displaystyle{
+\nabla_{\h}\int_{-h}^{\cdot} ( \kone (\cdot,\zeta)\T(\cdot,\zeta))\,d\zeta + \fv (\cdot,v,\T,\nabla v) \Big]dt }\\ 
\ \ \qquad \qquad\qquad \qquad \qquad\ \ \ \qquad 
\displaystyle{+\sum_{n\geq 1}\p \Big[(\phi_{n}\cdot\nabla) v  +\gvn(\cdot,v)\Big] d\beta_t^n}, \\
\displaystyle{d \Temp -\LT \Temp\, dt=\Big[ -(v\cdot\nabla_{\h}) \T -\w(v) \partial_3 \T+ \ft(\cdot,v,\T,\nabla v ) \Big]dt}
\\
\ \ \qquad \qquad\qquad \qquad \qquad\ \ 
\displaystyle{+\sum_{n\geq 1}\Big[(\psi_n\cdot \nabla) \Temp+\gtn(\cdot,v,\T,\nabla v)\Big] d\beta_t^n}, 
\\
v(\cdot,0)=v_0,\ \ \T(\cdot,0)=\T_0, 
\end{cases}
\end{equation}
on $\Dom=\Tor^2\times(-h,0)$ complemented with the following boundary conditions 
\begin{equation}
\label{eq:boundary_conditions_strong_weak_2}
\begin{aligned}
\partial_3 v (\cdot,-h)=\partial_3 v(\cdot,0)=0 \ \  \text{ on }\Tor^2,&\\
\partial_3 \T(\cdot,-h)=\partial_3 \T(\cdot,0)+\alpha \T(\cdot,0)=0\ \  \text{ on }\Tor^2.&
\end{aligned}
\end{equation}
Here $\Lp(\cdot,v)$ and $w(v)$ are as in \eqref{eq:def_P_gamma} and \eqref{eq:def_w_v}, respectively, and
\begin{align}
\label{eq:def_Lv_LT}
\Lv v= \sum_{i,j=1}^3 a_v^{i,j} \partial_{i,j}^2 v  +\sum_{j=1}^3 b_v^j \partial_j v,  \qquad
\LT \T= \sum_{i,j=1}^3 \partial_i (a_{\T}^{i,j} \partial_{j} \T )+\sum_{j=1}^3 b_{\T}^j \partial_j \T.
\end{align}
Here, $0$-th order terms in \eqref{eq:def_Lv_LT} can be added as well under suitable integrability conditions on the coefficients. 
Since it turns out  not to be useful when dealing with the primitive equations with Stratonovich noise (see Section \ref{s:Stratonovich}), we will not consider such terms.

To show local existence for \eqref{eq:primitive_weak_strong_2}-\eqref{eq:boundary_conditions_strong_weak_2} we employ the following 

\begin{assumption} 
\label{ass:local_strong_weak_2}
Let Assumption \ref{ass:well_posedness_primitive}\eqref{it:well_posedness_measurability}-\eqref{it:well_posedness_primitive_kone_smoothness} and 
\eqref{it:nonlinearities_measurability}-\eqref{it:nonlinearities_strong_weak} be satisfied. 
Suppose that there exist $K,\eta>0$ for which the following hold.
\begin{enumerate}[{\rm(1)}]
\item For all $i,j\in \{1,2,3\}$,
\begin{align*}
a^{i,j}_v,a^{i,j}_{\T}:\R_+\times \O\times \Dom\to \R
\end{align*}
are $\Progress\otimes \Borel$-measurable. Moreover, $a_v^{i,j}=a_v^{j,i}$ for all $i,j\in \{1,2,3\}$;
\item\label{it:a_ij_regularity_strong_weak_2} a.s.\ for all $t\in \R$ and $i,j\in \{1,2,3\}$,
\begin{align*}
\|a^{i,j}_v(t,\cdot)\|_{H^{1,3+\eta}(\Dom)}+
\|a_{\T}^{i,j}(t,\cdot)\|_{C(\overline{\Dom})}
\leq K,&\\
\|b_{v}^j(t,\cdot)\|_{L^{3+\eta}(\Dom)}
+
\|b_{\T}^j(t,\cdot)\|_{L^{3+ \eta}(\Dom)}\leq K;&
\end{align*}
\item\label{it:local_strong_weak_2_a_3} a.s.\ for all $t\in \R_+$, $x_{\h}\in \Tor^2$ and $j\in \{1,2\}$,
\begin{equation*}
 \|a^{3,j}_v(t,\cdot,0)\|_{H^{\frac{1}{2}+\eta}(\Tor^2)}\leq K, 
\ \ \text{ and } \ \ 
 a_{\T}^{3,j}(t,x_{\h},0)=a_{\T}^{3,j}(t,x_{\h},-h)=0;
\end{equation*}
\item\label{it:ellipticity_strong_weak_2} there exists $\ellip>0$ such that, a.s.\ for all $t\in \R$, $x\in \Dom$ and $\xi\in \R^d$,
\begin{align*}
\sum_{i,j=1}^3 \Big(a^{i,j}_{v}(t,x)-\frac{1}{2}\sum_{n\geq 1} \phi^j_n(t,x) \phi^i_n(t,x) \Big) \xi_i\xi_j&\geq \ellip |\xi|^2, \\
 \sum_{i,j=1}^3 \Big(a^{i,j}_{\T}(t,x)-\frac{1}{2}\sum_{n\geq 1} \psi^j_n(t,x) \psi^i_n (t,x)\Big) \xi_i\xi_j&\geq \ellip |\xi|^2.
\end{align*}
\end{enumerate}
\end{assumption}
Comments on Assumption \ref{ass:well_posedness_primitive} can be found in Remark \ref{r:assump_local_existence_strong_weak}. 
Let us collect some remarks on 
Assumption \ref{ass:local_strong_weak_2} in the following
\begin{remark}\
\label{r:well_posedness_assumption_variable_viscosity}
\begin{enumerate}[{\rm(a)}]
\item By \eqref{it:a_ij_regularity_strong_weak_2} and Sobolev embeddings, for all $i,j\in \{1,2,3\}$ we have 
$$
\|a^{i,j}_v(t,\cdot)\|_{ C(\overline{\Dom})}\lesssim_{\eta} K, \ \  \text{ a.s.\ for all $t\in \R_+;$}
$$
\item The regularity assumption on the trace $a_v^{i,j}(t,\cdot,0)$ in \eqref{it:local_strong_weak_2_a_3} is 
motivated by Lemma \ref{l:Kadlec_formula_II} which will be needed in the proofs of the results below;
\item 
\eqref{it:ellipticity_strong_weak_2} is the usual stochastic parabolicity condition, cf.\ Remark \ref{r:assump_local_existence_strong_weak}\eqref{it:remark_parabolicity}.
\end{enumerate}
\end{remark}

The notion of $L^2$-maximal strong-weak solution to \eqref{eq:primitive_weak_strong_2}-\eqref{eq:boundary_conditions_strong_weak_2} can be defined as in Definition \ref{def:sol_strong_weak} where the weak Robin Laplacian $\Dr$ has to be substitute by the weak formulation of $\LT$, which will be denoted by $\LTw$ and it is defined as
\begin{equation}
\label{eq:def_LTw}
\begin{aligned}
\l \LTw(t) \T, \varphi\r 
:= 
&- \int_{\Dom} \Big(\sum_{i,j=1}^3
a^{i,j}(t,\cdot)\partial_j \T \partial_i \varphi\,+ \sum_{j=1}^3 b^j_{\T}(t,\cdot) (\partial_j \T)\, \varphi\Big) dx
\\
&-\alpha \int_{\Tor^2} a^{3,3}(t,\cdot,0) \T(\cdot,0)\varphi(\cdot,0)\,dx_{\h},
\end{aligned}
\end{equation}
for all $\T,\varphi\in H^1(\Dom)$ and $t\in \R_+$. Here $\l \cdot,\cdot\r$ denotes the pairing in the duality $(H^1(\Dom))^*\times H^1(\Dom)$.
Note that the above formula is consistent with a formal integration by part using \eqref{eq:boundary_conditions_strong_weak_2} and the second statement in Assumption \ref{ass:local_strong_weak_2}\eqref{it:local_strong_weak_2_a_3}.

The following is an extension of Theorem \ref{t:local_primitive}. 
Recall that $\Hs_{\n}^2$ and $\norm_k$ are as in \eqref{eq:def_H_2_N_strong_weak} and \eqref{eq:def_norm}, respectively.

\begin{theorem}[Local existence - Inhomogeneous viscosity and/or conductivity]
\label{t:local_primitive_strong_weak_2}
Let Assumption \ref{ass:local_strong_weak_2} be satisfied.
 Then for each 
\begin{equation}
\label{eq:data_strong_weak_2}
v_0\in L^0_{\F_0}(\O;\Hs^1(\Dom)), \ \ \text{ and }\ \ \T_0\in L^0_{\F_0}(\O;L^2(\Dom)),
\end{equation}
there exists an \emph{$L^2$-maximal strong-weak solution} $((v,\theta) ,\tau)$ to \eqref{eq:primitive_weak_strong_2}-\eqref{eq:boundary_conditions_strong_weak_2} such that $\tau>0$ a.s. Moreover:
\begin{enumerate}[{\rm(1)}]
\item $\displaystyle{
(v,\T)\in L^2_{\loc}([0,\tau);\Hs_{\n}^2(\Dom)\times H^1(\Dom))\cap C([0,\tau);\Hs^1(\Dom)\times L^2(\Dom))}$ a.s.;
\vspace{0.2cm}
\item\label{it:blow_up_criteria_variable_viscosity}
$
\displaystyle{
\P\big(\tau<T,\,  \norm_{1}(\tau;v)+\norm_{0}(\tau;\T) <\infty\big)=0} 
$ for all $T\in (0,\infty)$.
\end{enumerate}
\end{theorem}

As in Subsection \ref{ss:main_results_strong_weak}, for global existence we need additional assumptions.

\begin{assumption}\
\label{ass:global_primitive_inhomogeneous}
Let Assumption \ref{ass:global_primitive} be satisfied. Suppose that 
\label{ass:global_primitive_strong_weak_2}
\begin{enumerate}[{\rm(1)}]
\item\label{it:independence_z_variable_a_ij} a.s.\ for all $n\geq 1$, $x=(x_{\h},x_3)\in \Tor^2\times (-h,0)=\Dom$ and $t\in \R_+$, 
\begin{align*}
a^{i,j}_v(t,x), \ b^j_v(t,x) \text{ are independent of $x_3$ where }i,j\in \{1,2\}.
\end{align*}
\item\label{it:a_3j_null_boundary} 
a.s.\ for all $n\geq 1$, $x_{\h}\in \Tor^2$ and $t\in \R_+$,
$$
a_v^{3,j}(t,x_{\h},0)=
a_v^{3,j}(t,x_{\h},-h)=0, \ \text{ where } j\in \{1,2\}. 
$$
\end{enumerate}
\end{assumption}

Remarks on Assumption \ref{ass:global_primitive} can be found in Remark \ref{r:global_primitive}. Note that 
Assumption \ref{ass:global_primitive_inhomogeneous}\eqref{it:independence_z_variable_a_ij} is the analogue of Assumption \ref{ass:global_primitive}\eqref{it:independence_z_variable}. Note that Assumption \ref{ass:global_primitive_inhomogeneous}\eqref{it:a_3j_null_boundary} implies the first condition in Assumption \ref{ass:local_strong_weak_2}\eqref{it:local_strong_weak_2_a_3}, and it will be needed to extend the $L^2$-estimate of Lemma \ref{l:L_2_estimate_I} to the case of inhomogeneous viscosity and/or conductivity.

The following is an extension of Theorem \ref{t:global_primitive}.

\begin{theorem}[Global existence - Inhomogeneous viscosity and/or conductivity]
\label{t:global_primitive_strong_weak_2}
Let Assumptions \ref{ass:local_strong_weak_2} and \ref{ass:global_primitive_strong_weak_2} be satisfied.  
Let $(v_0,\T_0)$ be as in \eqref{eq:data_strong_weak_2}. 
Then the $L^2$-maximal strong-weak solution $((v,\T),\tau)$ to  \eqref{eq:primitive_weak_strong_2}-\eqref{eq:boundary_conditions_strong_weak_2} provided by Theorem \ref{t:local_primitive_strong_weak_2} is \emph{global in time}, i.e.\ $\tau=\infty$ a.s. 
\end{theorem}

The proofs of Theorems \ref{t:local_primitive_strong_weak_2} and \ref{t:global_primitive_strong_weak_2} are given in Subsection \ref{ss:proofs_variable_viscosity_strong_weak} below and consist of a variation of the one in Section \ref{s:proofs_strong_weak} given for Theorems \ref{t:local_primitive} and \ref{t:global_primitive}. The major difference is that we will use the Kaldec's formula in Lemma \ref{l:Kadlec_formula_II} instead of the one in Lemma \ref{l:Kadlec_formula}.

\begin{remark}[Inhomogeneous viscosity and/or conductivity - Strong-strong setting]
\label{r:variable_viscosity_strong_strong}
As in Section \ref{s:strong_strong}, one sees that the results of Theorems \ref{t:local_primitive_strong_weak_2} and \ref{t:global_primitive_strong_weak_2} extend to the strong-strong setting under the following minor modifications: 
\begin{enumerate}[{\rm(a)}]
\item The local existence result of Theorem \ref{t:local_primitive_strong_weak_2} holds provided we take into account the following modifications.
\begin{itemize}
\item The symmetry assumption $a^{i,j}_{\T}=a^{j,i}_{\T}$ for all $i,j\in \{1,2,3\}$ is added.
 \item The regularity assumptions on $\psi_n$ and $a^{3,j}_{\T}$ in Assumptions \ref{ass:well_posedness_primitive}\eqref{it:well_posedness_primitive_phi_smoothness} and \ref{ass:local_strong_weak_2}\eqref{it:a_ij_regularity_strong_weak_2} are strengthened to, 
a.s.\ for all $t\in \R_+$ and $j,k\in \{1,2,3\}$,
\begin{equation}
\label{eq:additional_condition_psi_strong_strong}
\|a^{i,j}_{\T}(t,\cdot)\|_{H^{1,3+\eta}(\Dom)}+\|(\psi^j_n(t,\cdot))_{n\geq 1}\|_{H^{1,3+\eta}(\Dom;\ell^2)}
 \leq K.
\end{equation}
 \item The second condition in Assumption \ref{ass:local_strong_weak_2}\eqref{it:local_strong_weak_2_a_3} is replaced by, 
 a.s.\ for all $t\in \R_+$ and $j\in \{1,2\}$,
 \begin{equation}
 \label{eq:boundary_conditions_a_T_strong_strong}
\|a_{\T}^{3,j}(t,\cdot,0)\|_{H^{\frac{1}{2}+\eta}(\Tor^2)}
+
\|a_{\T}^{3,j}(t,\cdot,-h)\|_{H^{\frac{1}{2}+\eta}(\Tor^2)}
\leq K.
 \end{equation}
\item Assumption \ref{ass:well_posedness_primitive}\eqref{it:nonlinearities_measurability}-\eqref{it:nonlinearities_strong_weak} has to be replaced by Assumption \ref{ass:well_posedness_primitive_double_strong}\eqref{it:nonlinearities_measurability_strong_strong}-\eqref{it:nonlinearities_strong_strong}.
\end{itemize}
Note that \eqref{eq:additional_condition_psi_strong_strong}
  (resp.\ \eqref{eq:boundary_conditions_a_T_strong_strong}) is stronger (resp.\ weaker) than the corresponding assumption in the strong-weak setting. Let us stress that \eqref{eq:boundary_conditions_a_T_strong_strong} is enough in the strong setting for the local existence to hold due to the modified Kadlec's formula of Lemma \ref{l:Kadlec_formula_II} below, cf.\ the proof of Theorem \ref{t:local_primitive_strong_weak_2};
\item 
The global existence results of Theorem \ref{t:global_primitive_strong_weak_2} 
 also holds for \eqref{eq:primitive_weak_strong_2}-\eqref{eq:boundary_conditions_strong_weak_2} provided Assumptions \ref{ass:local_strong_weak_2} and \ref{ass:global_primitive_inhomogeneous} are satisfied, Assumption \ref{ass:global_primitive}\eqref{it:sublinearity_Gforce} is replaced by Assumption \ref{ass:global_primitive_strong_strong}\eqref{it:sublinearity_Gforce_strong_strong}, and also \eqref{eq:additional_condition_psi_strong_strong} holds. 
\end{enumerate}
\end{remark}

\subsection{Proof of Theorems \ref{t:local_primitive_strong_weak_2} and \ref{t:global_primitive_strong_weak_2}}
\label{ss:proofs_variable_viscosity_strong_weak}
In this subsection we collect the proofs of Theorems \ref{t:local_primitive_strong_weak_2} and \ref{t:global_primitive_strong_weak_2}. As the arguments are similar to the one used for Theorems \ref{t:local_primitive} and \ref{t:global_primitive}, respectively, we only give a sketch.

\begin{proof}[Proof of Theorem \ref{t:local_primitive_strong_weak_2} -- Sketch]
To prove Theorem \ref{t:local_primitive_strong_weak_2} one can follow verbatim the one of Theorem \ref{t:local_primitive} using the results in \cite{AV19_QSEE_1,AV19_QSEE_2} once the stochastic maximal $L^2$-regularity result of Proposition \ref{prop:SMR_2} holds in the present case, namely where $\Dr$ and $\Delta$ are replaced by $\LTw$ and $\Lv$ respectively. Let us note that, in the proof of Proposition \ref{prop:SMR_2}, the structure of the operators plays a role in the integration by parts arguments used. Below, we provide some comments which allows one to extend the argument given in Proposition \ref{prop:SMR_2} for the case of homogeneous viscosity and/or conductivity. 
In particular, as in Proposition \ref{prop:SMR_2}, we will assume that 
$$
(v,\T)
\in L^2((0,\tau)\times \O;\Hs_{\n}^2\times H^1)\cap  L^2(\O;C([0,\tau];\Hs^1\times L^2)),
$$ 
where $\tau:\O\to [0,T]$ is a stopping time and $T\in (0,\infty)$ is fixed. 
To begin we comment how to extend Step 1 in Proposition \ref{prop:SMR_2}. 
Let $\e_k$ be as in \eqref{eq:def_e_k}.
Note that $\e_k$ is symmetric since $\Dr$ is self-adjoint. 
Thus, a.e.\ on $[0,\tau]\times \O$,
\begin{equation}
\label{eq:convergence_mollified_operator_L_T}
\begin{aligned}
\int_{\Dom}\e_k [\LTw(t) \T(t,\cdot)] \T(t,x)\,dx
&= \l \LTw(t) \T, \e_k \T(t,\cdot)\r\\
& 
= -\alpha \int_{\Tor^2} a^{3,3}(t,x_{\h},0) \T(x_{\h},0)(\e_k\T)(x_{\h},0)\,dx_{\h}\\
&-\sum_{i,j=1}^3 \int_{\Dom} a^{i,j}(t,x)\partial_j \T(t,x) \partial_i (\e_k\T)(t,x)\,dx\\
&+ \sum_{j=1}^3 \int_{\Dom} b^j_{\T}(t,x) \partial_j \T(t,x)\, (\e_k \T)(t,x) \,dx \\
&\stackrel{k\to \infty}{\to}\l \LTw(t)\T, \T(t,\cdot)\r.
\end{aligned}
\end{equation}
The above estimate and the Lebesgue dominated convergence theorem yield an identity similar to the one in \eqref{eq:Ito_formula_theta_L2}. Note that, in all the remaining terms in \eqref{eq:Ito_formula_theta_L2}, the Lebesgue dominated convergence theorem can be applied since $\T\in L^2((0,\tau)\times \O;H^1)\cap L^2(\O; C([0,\tau];L^2))$ by assumption.

Considering the estimate $II_t$ in \eqref{eq:Ito_formula_theta_L2} (see the estimate before \eqref{eq:first_step_estimate_temperature_L_2}),
note that,  by Assumption \ref{ass:well_posedness_primitive}\eqref{it:well_posedness_primitive_L_infty_bound}, $\big|\sum_{n\geq 1}\sum_{i,j=1}^3 \psi^j_n(t,x) \psi ^i_n(t,x)  \xi_i\xi_j\big|\leq 9 M^2 |\xi|^2 $ a.s.\ for all $t\in \R_+$, $x\in \Dom$ and $\xi\in \R^d$. Let $R_{\T}:=1+\frac{\ellip}{18 M^2}>1$. Thus the parabolicity condition of Assumption 
\ref{ass:local_strong_weak_2}\eqref{it:ellipticity_strong_weak_2} yields, a.s.\ for all $t\in \R_+$ and $x\in \Dom$,
\begin{align}
\label{eq:ellipticity_a_T_R}
\frac{R_{\T}}{2}\sum_{n\geq 1} \psi^j_n(t,x) \psi^i_n(t,x)  \xi_i\xi_j&\leq 
\sum_{i,j=1}^3 a^{i,j}_{\T}(t,x) \xi_i\xi_j 
-\frac{\ellip}{2} |\xi|^2.
\end{align}
By \eqref{eq:convergence_mollified_operator_L_T}-\eqref{eq:ellipticity_a_T_R}, one can check that \eqref{eq:first_step_estimate_temperature_L_2} also holds in the present case where  in the estimates leading to \eqref{eq:first_step_estimate_temperature_L_2} one uses \eqref{eq:ellipticity_a_T_R} and $ R_{\T}$ instead of Assumption \ref{ass:well_posedness_primitive}\eqref{it:well_posedness_primitive_parabolicity} and $\frac{\ellip'}{\ellip}$, respectively. 

Next we give some comments which are useful to extend Step 2 of Proposition \ref{prop:SMR_2} to the present case. 
Let us begin by checking that the argument in \eqref{eq:convergence_mollified_operator_L_T} can be performed also for the $v$-equation. Let $\mathcal{R}_{k}:=k(k+1+\Delta_{\n})^{-1}$ where $\Delta_{\n}$ is the strong Neumann Laplacian on $\Ls^2$, see \eqref{eq:def_strong_Neumann}.
To this end, note that, integrating by parts, we have
\begin{equation}
\label{eq:convergence_mollified_operator_L_v}
\begin{aligned}
\int_{\Dom}(\nabla \mathcal{R}_k [\p (\Lv v)])\cdot \nabla v\,dx  
&\stackrel{(i)}{=} -\int_{\Dom}\mathcal{R}_k [\p(\Lv v)] \cdot \Delta v\,dx\\
&\stackrel{k\to \infty}{\to}-\int_{\Dom}[\p(\Lv v)] \cdot \Delta v\,dx\\
&\stackrel{(ii)}{=}- \int_{\Dom}\Lv v \cdot \Delta v\,dx
\end{aligned}
\end{equation}
a.e.\ on $[0,\tau]\times \O$. Here in $(i)$ we used the boundary conditions $\partial_3 v(\cdot,0)=\partial_3 v(\cdot,-h)=0$ on $\Tor^2$ and in $(ii)$ we have used that $\p$ is self-adjoint and $\Delta$ commutes with $\p$ since $v\in \Hs^2(\Dom)$ on $[0,\tau]\times \O$ and satisfies the above mentioned homogeneous Neumann boundary conditions. 
As for \eqref{eq:convergence_mollified_operator_L_T}, the limit in \eqref{eq:convergence_mollified_operator_L_v} as $k\to \infty$ holds a.e.\ on $[0,\tau)\times \O$ by Lebesgue's dominated convergence theorem and the fact that $v\in L^2((0,\tau)\times \O;H^2)\cap L^2(\O; C([0,\tau];H^1))$ by assumption.

To extend the estimate in \eqref{eq:estimate_Ito_correction_v_smr} similarly as for the $\T$-equation, note that, by Assumption \ref{ass:well_posedness_primitive}\eqref{it:well_posedness_primitive_phi_smoothness} (see  Remark \ref{r:assump_local_existence_strong_weak}\eqref{it:Holder_continuity_phi}), one sees that \eqref{eq:ellipticity_a_T_R} holds for $\psi,a_{\T}$ replaced by $\phi,a_{v}$ with $R_v>1$ depending only on $K,\delta$. 
Thus, one obtains the estimate,
\begin{equation}
\label{eq:Ito_correction_bound_variable_viscosity}
\sum_{n\geq 1}\| \nabla \p[(\phi_n\cdot \nabla )v]\|_{L^2}^2	\leq 
\frac{2}{R_v} \sum_{i,j,k=1}^3 \int_{\Dom}a_v^{i,j}\partial_{i,k}^2 v\cdot \partial_{j,k}^2 v\,dx
+c_{R_v,M}\| v\|_{L^{2}}^2.
\end{equation}
By the modified Kadlec's formula of Lemma \ref{l:Kadlec_formula_II}, one can estimate the first term on the RHS of \eqref{eq:Ito_correction_bound_variable_viscosity} by $\frac{2}{R_v'}\int_{\Dom}\Lv v \cdot \Delta v\,dx+ C_{K,\eta} \int_{\Dom}|v|^2\,dx$ for any $1<R'_v<R_v$ and it can be adsorbed in the LHS of the modified It\^{o}'s identity for $\nabla v^{\tau}$ due to \eqref{eq:convergence_mollified_operator_L_v} and that $R_v'>1$.
\end{proof}

\begin{proof}[Proof of Theorem \ref{t:global_primitive_strong_weak_2} -- Sketch]
To prove Theorem \ref{t:global_primitive_strong_weak_2} one can follow the argument used in Subsection \ref{ss:proof_global_strong_weak}. Indeed, since the blow-up criteria in Theorem \ref{t:local_primitive_strong_weak_2}\eqref{it:blow_up_criteria_variable_viscosity} holds, it remains to show that the energy estimate of Proposition \ref{prop:energy_estimate_primitive} also holds in the case of inhomogeneous viscosity and/or conductivity. 

Note that the $L^2$-estimate of Lemma \ref{l:L_2_estimate_I} follows similarly where one also need to use Assumption \ref{ass:global_primitive_inhomogeneous}\eqref{it:a_3j_null_boundary} to integrate by parts in the $v$-equation, cf.\ \eqref{eq:def_LTw}.

Also the content of Lemma \ref{l:intermediate_estimate} holds in this case. For simplicity, we consider the case $b_v^j\equiv 0$, the general case is analogous. To see that Lemma \ref{l:intermediate_estimate} also holds in this case,  note that, Assumption \ref{ass:global_primitive_inhomogeneous}\eqref{it:independence_z_variable_a_ij} is needed to obtain an equation for $\overline{v}$ similar to \eqref{eq:primitive_bar} and also in Step 3 where we estimate $v_3=\partial_3 v$ to get
\begin{equation}
\begin{aligned}
\label{eq:integration_by_parts_inhomogeneous_v_3}
\int_{\Dom} \p [\Lv v] \cdot \partial_{3,3}^2 v\,dx
&\stackrel{(i)}{=}
\int_{\Dom}\Lv v \cdot \partial_{3,3}^2 v\\
&\stackrel{(ii)}{=}
\sum_{i,j=1}^3 \int_{\Dom}a^{i,j} \partial_j v_3\cdot \partial_{i} v_3- 
\sum_{i,j=1}^2 \int_{\Dom} (\partial_j a^{i,j}) \partial_j v_3 \cdot v_3\, dx,
\end{aligned}
\end{equation}
where in $(i)$ we used that $\p [\Lv v]=\Lv v -\q [\Lv v]$ and $\int_{\Dom}\q [\Lv v]\partial_{3,3}^2 v\,dx=0$ since $\q[\Lv v]$ is independent of $x_3$ and $\partial_3v (\cdot,-h)=\partial_3 v(\cdot,0)=0$ on $\Tor^2$.  While $(ii)$ follows by the symmetry of the matrix $a=(a^{i,j})_{i,j=1}^3$ and the fact that, for all $f\in C^{3}(\overline{\Dom})$ such that $\partial_3 f (\cdot,-h)=\partial_3 f(\cdot,0)=0$ on $\Tor^2$,
\begin{align*}
\int_{\Dom} a^{i,j} \partial_{i,j}^2 f\cdot \partial_{3,3}^2 f\,dx
&=
-\int_{\Dom} a^{i,j} \partial_{i,j,3}^2 f\cdot f_3 \,dx\\
&=
\int_{\Dom} a^{i,j} \partial_{i,j}^2 f_3 \cdot f_3\,dx\\
&=
-\int_{\Dom} a^{i,j} \partial_{i} f_3 \cdot \partial_j f_3- \int_{\Dom} (\partial_j a^{i,j}) \partial_i f_3 \cdot f_3\,dx
\end{align*}
where $f_3:=\partial_3 f$ and we have used  Assumption \ref{ass:global_primitive_inhomogeneous}\eqref{it:independence_z_variable_a_ij} in the first integration by part. As a concluding remark, let us stress that the last  terms in \eqref{eq:integration_by_parts_inhomogeneous_v_3} are of lower-order type.
Indeed, $\partial_j a^{i,j}\in H^{1,3+\eta}(\Dom)$ by Assumption \ref{ass:local_strong_weak_2}\eqref{it:a_ij_regularity_strong_weak_2}, and one can reason as in the proof of \eqref{eq:estimate_Ito_correction_v_smr}  to get, for each $\varepsilon>0$,
$$
\Big|\int_{\Dom} (\partial_j a^{i,j}) \partial_j v_3 \cdot v_3\, dx\Big|
\leq \varepsilon\|\nabla v_3\|_{L^2}^2+ C_{\varepsilon,K,\eta}  \|v_3\|^2_{L^2}, 
$$
where $K$ is as in Assumption \ref{ass:global_primitive_inhomogeneous}\eqref{it:a_ij_regularity_strong_weak_2}.

Having extended Lemmas \ref{l:L_2_estimate_I} and \ref{l:intermediate_estimate}, the extension of the Proposition \ref{prop:energy_estimate_primitive} in the present case follows verbatim from the one  given in Subsection \ref{ss:proof_global_strong_weak} since Proposition \ref{prop:SMR_2} holds also in the case of inhomogeneous viscosity and/or conductivity (cf.\  the proof of Theorem \ref{t:local_primitive_strong_weak_2}).
\end{proof}

\section{Transport noise of Stratonovich type}
\label{s:Stratonovich}
In this section we show how the results of the previous section also covers the situation where the primitive equations are perturbed by a transport noise in Stratonovich form. More precisely, here we consider 
\begin{equation}
\label{eq:primitive_Stratonovich}
\begin{cases}
\displaystyle{d v -\Delta v\, dt=\p\Big[  -(v\cdot \nabla_{\h})v- w(v)\partial_3 v}\\
\qquad \qquad\quad
 \displaystyle{+\nabla_{\h} \int_{-h}^{\cdot}  (\kone(\cdot,\zeta)\T(\cdot,\zeta))\,d\zeta+ \fv(\cdot,v,\T) \Big]dt }\\
\qquad \qquad \ \ \ \ \ \qquad  \qquad   \qquad   \ \ 
 \displaystyle{+\sum_{n\geq 1}\p [(\phi_{n}\cdot\nabla) v ]\circ d\beta_t^n}, 
& \ \ 
\text{on }\Dom,\\
\displaystyle{d \T -\Delta \T\, dt=\Big[ -(v\cdot \nabla_{\h})\T- w(v)\partial_3 \theta + \ft(\cdot,v,\T ) \Big]dt} \\
\qquad \qquad\qquad\qquad\ \ \ \ \  \qquad  \qquad  
\displaystyle{+\sum_{n\geq 1}(\psi_{n}\cdot\nabla) \T \circ d\beta_t^n}, 
&\ \ 
\text{on }\Dom,\\
v(\cdot,0)=v_0,\ \ \T(\cdot,0)=\T_0, &\ \ 
\text{on }\Dom,
\end{cases}
\end{equation}
where  $\circ$ denotes the Stratonovich integration (see e.g.\ \cite{Gar09}) and,  as above, \eqref{eq:primitive_Stratonovich} is complemented with the following boundary conditions
\begin{equation}
\label{eq:primitive_Stratonovich_boundary_conditions}
\begin{aligned}
\partial_3 v (\cdot,-h)=\partial_3 v(\cdot,0)=0 \ \  \text{ on }\Tor^2,&\\
\partial_3 \T(\cdot,-h)= \partial_3 \T(\cdot,0)+\alpha \T(\cdot,0)=0\ \  \text{ on }\Tor^2.&
\end{aligned}
\end{equation}
In order to make this section as clear as possible, in contrast to Sections \ref{s:strong_weak}, \ref{s:strong_strong} and \ref{s:variable_viscosity}, in \eqref{eq:primitive_Stratonovich} we do not consider lower order terms in the stochastic part keeping only the transport terms $(\phi_n\cdot \nabla) v \circ d\beta^n_t$ and $(\psi_n\cdot \nabla )\T \circ d\beta^n_t$ which are the most relevant from a physical point of view (see \cite{BiFla20, MR01, MR04} and the references therein) and mathematically the lower order terms are easier to handle. However, our methods can be also extended to the case of lower-order stochastic perturbations. 

As common in SPDEs, the Stratonovich integration in $\p [(\phi_n\cdot \nabla) v]\circ d\beta^n_t$ and $(\psi_n\cdot \nabla )\T\circ d\beta^n_t$ will be understood as an It\^{o}'s ones plus some correction terms. As remarked at the beginning of Section \ref{s:variable_viscosity} latter terms yield (in general) non-constant viscosity and/or conductivity and therefore \eqref{eq:primitive_Stratonovich} fits into the scheme of such section.

This section is organized as follows. In Subsection \ref{ss:main_results_Stratonovich_strong_weak} we study the equations 
\eqref{eq:primitive_Stratonovich}-\eqref{eq:primitive_Stratonovich_boundary_conditions} in the strong-weak setting and in Subsection \ref{ss:proofs_Stratonovich} we provide the corresponding proofs.
For brevity, we do not give the explicit statements in the strong-strong setting. The modifications needed in the latter situation is commented on in Remark \ref{r:strong_strong_Stratonovich} below.

\subsection{The strong-weak setting in the Stratonovich case}
\label{ss:main_results_Stratonovich_strong_weak}
Here we analyze \eqref{eq:primitive_Stratonovich}-\eqref{eq:primitive_Stratonovich_boundary_conditions}
 in the strong-weak setting under the following 
\begin{assumption}
\label{ass:local_Stratonovich_strong_weak}
There exists $M,\delta>0$ for which the following hold.
\begin{enumerate}[{\rm(1)}]
\item\label{it:well_posedness_measurability_Stratonovich} For all $n\geq 1$ and $j\in \{1,2,3\}$, 
$\phi^j_n,\psi^j_n :\O\times \Dom \to \R$ and $\kone: \R_+\times \O\times \Dom\to \R$ 
are $\F_0\otimes \Borel$-measurable and $\Progress\otimes \Borel$-measurable, respectively;
\item\label{it:well_posedness_primitive_phi_psi_smoothness_Stratonovich} 
 a.s.\ for all $t\in \R_+$, $x\in \Dom$ and $j,k\in \{1,2,3\}$,
\begin{align*}
\Big\|\Big(\sum_{n\geq 1}| \phi^j_n|^2\Big)^{1/2} \Big\|_{L^{3+\delta}(\Dom)}+
\Big\|\Big(\sum_{n\geq 1}|\partial_k \phi^j_n|^2\Big)^{1/2} \Big\|_{L^{3+\delta}(\Dom)} 
\leq M,&\\
\Big(\sum_{n\geq 1} | \psi^j_n(x) |^2\Big)^{1/2} + 
\Big\|\Big(\sum_{n\geq 1}|\div\,\psi_n|^2 \Big)^{1/2}\Big\|_{L^{3+\delta}(\Dom)}
\leq M,&
\end{align*}
where $\div\,\psi_n=\sum_{j=1}^3 \partial_j \psi^j_n$;
\item\label{it:well_posedness_primitive_kone_smoothness_Stratonovich}
a.s.\ for all $t\in \R_+$, $x_{\h}\in \Tor^2$, $j\in \{1,2,3\}$ and $i\in \{1,2\}$,
$$
\| \kone(t,x_{\h},\cdot) \|_{L^2(-h,0)} +\|\partial_i \kone(t,\cdot) \|_{L^{2+\delta}(\Tor^2;L^2(-h,0))} \leq M;
$$
\item\label{it:independence_z_variable_Stratonovich}
a.s.\ for all $n\geq 1$, $x=(x_{\h},x_3)\in \Tor^2\times (-h,0)=\Dom$ and $j\in \{1,2\}$,
\begin{align*}
\phi_n^j(x) \text{ is independent of $x_3$}.
\end{align*}
\item a.s. for all $n\geq 1$ and $j\in\{1,2\}$,
$$
\Big\|\sum_{n\geq 1} \phi_n^j(\cdot,0)\phi_n^3(\cdot,0)\Big\|_{H^{\frac{1}{2}+\delta}(\Tor^2)}
+
\Big\|\sum_{n\geq 1} \phi_n^j(\cdot,-h)\phi_n^3(\cdot,-h)\Big\|_{H^{\frac{1}{2}+\delta}(\Tor^2)}
\leq M;
$$
\item\label{it:psi_3_null_regularity_assumption_phi_boundary} 
a.s.\ for all $n\geq 1$ and  $x_{\h}\in \Tor^2$,
\begin{align*}
\psi_n^3(x_{\h},0)=
\psi_n^3(x_{\h},-h)=0;
\end{align*}
\item 
the maps 
\begin{align*}
&\fv\colon \R_+\times \O\times \R^2\times \R^6\times \R \to \R^2, \\
&\ft\colon \R_+\times \O\times \R^2\times \R^6\times \R \to \R
\end{align*}
are $\Progress\otimes \Borel$-measurable. 
%
For all $T\in (0,\infty)$ and $i\in\{1,2\}$, 
\begin{align*}
\fv^i (\cdot,0),\ft(\cdot,0)\in L^2((0,T)\times\O\times \Dom).
\end{align*}
%
Finally, for all $n\geq 1$, $t\in \R_+$,  $x\in \Dom$, $y,y'\in \R^2$, $Y,Y'\in \R^6$ and $z,z'\in \R$,
\begin{align*}
&|\fv(t,x,y,z,Y)-\fv(t,x,y',z',Y')|+
|\ft(t,x,y,z,Y)-\ft(t,x,y',z',Y')|\\
&\qquad\qquad
\lesssim (1+|y|^4+ |y'|^4)|y-y'|+
(1+|z|^{2/3}+|z'|^{2/3})|z-z'|\\
&\qquad\qquad
+(1+|Y|^{2/3}+|Y'|^{2/3})|Y-Y'|.
\end{align*}
\end{enumerate}
\end{assumption}

\begin{remark} Some remark on Assumption \ref{ass:local_Stratonovich_strong_weak} may be in order.
\begin{enumerate}[{\rm(a)}]
\item \eqref{it:well_posedness_measurability_Stratonovich}-\eqref{it:well_posedness_primitive_kone_smoothness_Stratonovich} should be compared with Assumption \ref{ass:well_posedness_primitive}\eqref{it:well_posedness_measurability}-\eqref{it:well_posedness_primitive_kone_smoothness}. Note that in \eqref{it:well_posedness_primitive_phi_psi_smoothness_Stratonovich} a regularity assumption on $\div\,\psi_n$ has been added;
\item \eqref{it:independence_z_variable_Stratonovich} coincide with Assumption \ref{ass:global_primitive}\eqref{it:independence_z_variable} (see also \eqref{eq:coefficients_Stratonovich} below for the choice of $\hp_n^{\ell,m}$).
\end{enumerate}
\end{remark}

As explained at the beginning of Section \ref{s:variable_viscosity}, the Stratonovich problem \eqref{eq:primitive_Stratonovich}-\eqref{eq:primitive_Stratonovich_boundary_conditions} will be viewed in the form  \eqref{eq:primitive_weak_strong_2}-\eqref{eq:boundary_conditions_strong_weak_2} where inhomogeneous viscosity and/or conductivity is considered. To this end, set $a_{\phi}=(a^{i,j}_{\phi})_{i,j=1}^3$ and $a_{\psi}=(a^{i,j}_{\psi})_{i,j=1}^3$, where a.s.\ for all $x\in \Dom$,
\begin{equation}
\label{eq:def_a_phi_psi_strong_weak}
a_{\phi}^{i,j}(x):=\delta^{i,j}+\frac{1}{2}\sum_{n\geq 1} \phi^i_n (x)\phi^j_n (x),
\ \text{ and } \ 
a_{\psi}^{i,j} (x):=\delta^{i,j}+\frac{1}{2}\sum_{n\geq 1} \psi^i_n (x)\psi^j_n (x).
\end{equation}
If Assumption \ref{ass:local_Stratonovich_strong_weak}\eqref{it:independence_z_variable_Stratonovich} holds, then at least formally one has,
\begin{align}
\label{eq:Ito_corrections_strong_weak_T}
(\psi_n \cdot \nabla ) \T \circ d\beta_t^n 
&=
\LTphi \T\,dt
+
(\psi_n \cdot \nabla ) \T \, d\beta_t^n,
\\
\label{eq:Ito_corrections_strong_weak_v}
\p[(\phi_n \cdot \nabla) v]\circ d\beta^n_t
&=
\p\big[\Lvphi v+ \Lvp v \big]\,dt+ \p[(\phi_n \cdot \nabla) v]\, d\beta^n_t,
\end{align}
where 
\begin{align}
\label{eq:def_LTp}
\LTphi \T
&:=
\div(a_{\psi} \cdot\nabla \T)\,dt-\frac{1}{2}\sum_{n\geq 1} (\div \,\psi_n )[ (\psi_n\cdot \nabla)\T],
\\
\label{eq:def_Lvp}
\Lvphi v
&:=
 \sum_{i,j=1}^3 \Big(a_{\phi}^{i,j} \partial_{i,j}^2 v + \frac{1}{2}\sum_{n\geq 1}(\partial_i \phi^j_n) \phi^i_n \partial_j v\Big),\\
\label{eq:def_Lvp_two}
\Lvp v
&:=\Big(\sum_{n\geq 1}\sum_{i=1}^2\partial_j \phi_n^i  (\q[(\phi_n\cdot\nabla) v])^i \Big)_{j=1}^2,
\end{align}
where $(\q[(\phi_n\cdot\nabla) v])^i$ denotes the $i$-th component of $\q[(\phi_n\cdot\nabla) v]$.
We postpone the proof of \eqref{eq:Ito_corrections_strong_weak_T}-\eqref{eq:Ito_corrections_strong_weak_v} to the beginning of Subsection \ref{ss:proofs_Stratonovich} below.

By \eqref{eq:Ito_corrections_strong_weak_T}-\eqref{eq:def_Lvp_two}, the problem \eqref{eq:primitive_Stratonovich}-\eqref{eq:primitive_Stratonovich_boundary_conditions} is (formally) equivalent to \eqref{eq:primitive_weak_strong_2}-\eqref{eq:boundary_conditions_strong_weak_2}  with 
\begin{equation}
\label{eq:coefficients_Stratonovich}
\begin{aligned}
\gvn\equiv 0, \ \ \ \gtn\equiv 0, \ \ \ a=a_{\phi},\ \ \ a=a_{\psi}, \ \ \ \hp^{\ell,m}_n=\partial_{\ell} \phi^m_n,& \\
b_{v}^j =\sum_{n\geq 1}\sum_{i=1}^3\frac{1}{2}(\partial_i \phi^j_n) \phi^i_n, \ \ \ \text{ and }\ \ \
b_{\T}^j= -\frac{1}{2}\sum_{n\geq 1} (\div \,\psi_n )\psi^j_n.&
\end{aligned}
\end{equation}

Since \eqref{eq:primitive_Stratonovich}-\eqref{eq:primitive_Stratonovich_boundary_conditions} is in the form \eqref{eq:primitive_weak_strong_2}-\eqref{eq:boundary_conditions_strong_weak_2} with the above choice, the notion of $L^2$-maximal strong-weak solution corresponds to the one of \eqref{eq:primitive_weak_strong_2}-\eqref{eq:boundary_conditions_strong_weak_2}.
In the following $\Hs_{\n}^2$ and $\norm_k$ are as in \eqref{eq:def_H_2_N_strong_weak} and \eqref{eq:def_norm}, respectively.

\begin{theorem}[Local existence - Stratonovich case]
\label{t:local_primitive_Stratonovich_strong_weak}
Let Assumption \ref{ass:local_Stratonovich_strong_weak} be satisfied.
 Then for each 
\begin{equation}
\label{eq:data_Stratonovich_strong_weak}
v_0\in L^0_{\F_0}(\O;\Hs^1(\Dom)), \ \ \text{ and }\ \ \T_0\in L^0_{\F_0}(\O;L^2(\Dom)),
\end{equation}
there exists an \emph{$L^2$-maximal strong-weak solution} $((v,\theta) ,\tau)$ to \eqref{eq:primitive_Stratonovich}-\eqref{eq:primitive_Stratonovich_boundary_conditions} such that $\tau>0$ a.s. Moreover:
\begin{enumerate}[{\rm(1)}]
\item $\displaystyle{
(v,\T)\in L^2_{\loc}([0,\tau);\Hs_{\n}^2(\Dom)\times H^1(\Dom))\cap C([0,\tau);\Hs^1(\Dom)\times L^2(\Dom))}$ a.s.;
\vspace{0.1cm}
\item 
$
\displaystyle{
\P\Big(\tau<T,\,  \norm_{1}(\tau;v)+\norm_{0}(\tau;\T) <\infty\Big)=0} 
$ for all $T\in (0,\infty)$.
\end{enumerate}
\end{theorem}

As usual, under additional assumptions we obtain a global existence result. 
 
\begin{assumption}\ 
\label{ass:global_primitive_Stratonovich_strong_weak}
\begin{enumerate}[{\rm(1)}]
\item\label{it:null_phi_3_Stratonovich_global} 
a.s.\ for all $n\geq 1$ and $x_{\h}\in \Tor^2$,
$$
\phi^3_n(x_{\h},0)=
\phi^3_n(x_{\h},-h)=0;
$$
\item there exist $C>0$ and $\y\in L^0_{\Progress}(\O;L^2_{\loc}([0,\infty);L^2(\Dom)))$ such that, a.s.\ for all $t\in \R_+$, $j\in \{1,2,3\}$, $x\in \Dom$, $y\in \R^2$, $z\in \R$ and $Y\in \R^{6}$, 
\begin{align*}
|\fv(t,x,y,z,Y)|+|\ft(t,x,y,z,Y)|
\leq C(\y(t,x)+|y|+|z|+|Y|).
\end{align*}
\end{enumerate}
\end{assumption}

\begin{theorem}[Global existence - Stratonovich case]
\label{t:global_primitive_strong_strong_2}
Let Assumptions \ref{ass:local_Stratonovich_strong_weak} and \ref{ass:global_primitive_Stratonovich_strong_weak} be satisfied.
Let $(v_0,\T_0)$ be as in \eqref{eq:data_Stratonovich_strong_weak}.
Then the $L^2$-maximal strong-weak solution $((v,\T),\tau)$ to  \eqref{eq:primitive_Stratonovich}-\eqref{eq:primitive_Stratonovich_boundary_conditions} provided by Theorem \ref{t:local_primitive_Stratonovich_strong_weak} is \emph{global in time}, i.e.\ $\tau=\infty$ a.s. 
\end{theorem}

The proofs of Theorems \ref{t:local_primitive_Stratonovich_strong_weak} and \ref{t:global_primitive_strong_strong_2} as well as \eqref{eq:Ito_corrections_strong_weak_T}-\eqref{eq:Ito_corrections_strong_weak_v} will be given in Subsection \ref{ss:proofs_Stratonovich} below. We conclude this subsection with a few remarks.

\begin{remark}
Assumption \ref{ass:global_primitive_Stratonovich_strong_weak}\eqref{it:null_phi_3_Stratonovich_global}
ensures that $\phi^3_n$ does \emph{not} interact with the boundary. Such assumption will be used to ensure that $a_{\phi}^{3,j}(x_{\h},0)=a_{\phi}^{3,j}(x_{\h},-h)=0$ for all $j\in \{1,2\}$ and $x_{\h}\in \Tor^2$, cf.\ Assumption \ref{ass:global_primitive_inhomogeneous}\eqref{it:a_3j_null_boundary}.
\end{remark}

\begin{remark}[The strong-strong setting for \eqref{eq:primitive_Stratonovich}]\
\label{r:strong_strong_Stratonovich}
By Remark \ref{r:variable_viscosity_strong_strong} in the strong-strong setting (cf.\ Section \ref{s:strong_strong} for the It\^{o}'s case) the following hold for \eqref{eq:primitive_Stratonovich}-\eqref{eq:primitive_Stratonovich_boundary_conditions}:
\begin{enumerate}[{\rm(a)}]
\item 
The local existence results of Theorem \ref{t:local_primitive_Stratonovich_strong_weak} 
 also holds for \eqref{eq:primitive_Stratonovich}-\eqref{eq:primitive_Stratonovich_boundary_conditions} also holds in the strong-strong setting  
 provided Assumptions \ref{ass:local_Stratonovich_strong_weak} and \ref{ass:global_primitive_Stratonovich_strong_weak}
hold and we add the following modifications: 
\begin{itemize}
 \item The regularity assumptions on $\psi_n$ in Assumption \ref{ass:local_Stratonovich_strong_weak}\eqref{it:well_posedness_primitive_phi_psi_smoothness_Stratonovich} are replaced by 
\begin{equation}
\label{eq:additional_condition_psi_strong_strong_Stratonovich}
\| (\psi^j_n)_{n\geq 1} \|_{H^{1,3+\delta}(\Dom;\ell^2)} \leq M, \ \ \text{ for all }j\in \{1,2,3\}; 
\end{equation}
 \item Assumption \ref{ass:local_Stratonovich_strong_weak}\eqref{it:psi_3_null_regularity_assumption_phi_boundary} is replaced by, for all $j\in \{1,2\}$,
 \begin{equation}
 \label{eq:boundary_conditions_psi_strong_strong_Stratonovich}
\Big\|\sum_{n\geq 1} \psi_n^j(\cdot,0)\psi_n^3(\cdot,0)\Big\|_{H^{\frac{1}{2}+\delta}(\Tor^2)}
+
\Big\|\sum_{n\geq 1} \psi_n^j(\cdot,-h)\psi_n^3(\cdot,-h)\Big\|_{H^{\frac{1}{2}+\delta}(\Tor^2)}
\leq M.
 \end{equation}
\end{itemize}
Note that  \eqref{eq:additional_condition_psi_strong_strong_Stratonovich} (resp.\ 
 \eqref{eq:boundary_conditions_psi_strong_strong_Stratonovich}) is stronger (resp.\ weaker) than the corresponding assumption in the strong-weak setting;
\item 
The global existence results of Theorem \ref{t:global_primitive_strong_strong_2} 
 also holds for \eqref{eq:primitive_Stratonovich}-\eqref{eq:primitive_Stratonovich_boundary_conditions} provided Assumptions \ref{ass:local_Stratonovich_strong_weak} and \ref{ass:global_primitive_Stratonovich_strong_weak} are satisfied and also \eqref{eq:additional_condition_psi_strong_strong_Stratonovich} holds. 
\end{enumerate}
\end{remark}

\subsection{Proofs of \eqref{eq:Ito_corrections_strong_weak_T}-\eqref{eq:Ito_corrections_strong_weak_v} and Theorems \ref{t:local_primitive_Stratonovich_strong_weak} and \ref{t:global_primitive_strong_strong_2}}
\label{ss:proofs_Stratonovich}

\begin{proof}[Formal proof of \eqref{eq:Ito_corrections_strong_weak_T}-\eqref{eq:Ito_corrections_strong_weak_v}]
To motivate \eqref{eq:Ito_corrections_strong_weak_T}, recall that
\begin{equation}
\label{eq:stratonovich_correction_motivation_T}
(\psi_n \cdot \nabla) \T\circ d\beta^n_t
=
\frac{1}{2}(\psi_n \cdot \nabla)[(\psi_n \cdot \nabla) \T]\,dt+
(\psi_n \cdot \nabla) \T\,d\beta^n_t.
\end{equation}
Thus \eqref{eq:Ito_corrections_strong_weak_v} follows by 
noticing that $(\psi_n \cdot\nabla) f= \div(\psi_n f)-(\div\, \psi) f$ for all sufficiently smooth function $f$.

To motivate \eqref{eq:Ito_corrections_strong_weak_v}, reasoning as in \eqref{eq:stratonovich_correction_motivation_T} (cf.\ \cite[Chapter 3]{F_lectures_Waseda}), we have
\begin{equation}
\label{eq:correction_v_step_1}
\p[(\phi_n \cdot \nabla) v]\circ d\beta^n_t
=
\p[(\phi_n \cdot \nabla) v ]\,d\beta^n_t
+
\frac{1}{2}\p\Big[(\phi_n \cdot \nabla)\big(\p[(\phi_n \cdot \nabla) v] \big)\Big] \,dt.
\end{equation}
Next, we rewrite the last term appearing in the RHS of \eqref{eq:correction_v_step_1}. To this end, recall that $\nabla_{\h} \wt{P}_n=\q [(\phi_n \cdot \nabla) v]$ and $\p=I- \q $ (here $I$ denotes the identity operator), cf.\ Subsection \ref{ss:set_up}. Thus
\begin{equation}
\label{eq:correction_v_step_2}
\p\Big[(\phi_n \cdot \nabla)\p[(\phi_n \cdot \nabla) v ]\Big]
=
\p\Big[(\phi_n \cdot \nabla)[(\phi_n \cdot \nabla) v]\Big]
- \p\Big[(\phi_n \cdot \nabla)\nabla_{\h} \wt{P}_n\Big].
\end{equation}
By the product rule the first term on the right hand side of \eqref{eq:correction_v_step_2} is equivalent to $\p[\Lvphi v]\,dt$. It remains to show that the second term on the right hand side of \eqref{eq:correction_v_step_1} is equivalent to $\p[\Lvp v]\,dt$. 
For $j\in \{1,2\}$, note that
$$
(\phi_n \cdot \nabla)\partial_j  \wt{P}_n\stackrel{(i)}{=}
(\phi_{n,\h} \cdot \nabla_{\h})\partial_j  \wt{P}_n
=
\partial_j \big[ (\phi_{n,\h} \cdot \nabla_{\h}) \wt{P}_n \big]- \sum_{i=1}^2 ( \partial_j \phi^i_n  ) \partial_i \wt{P}_n,
$$
where $\phi_{n,\h}=(\phi^1_n,\phi^2_n)$ and in $(i)$ we used that $\wt{P}_n$ is independent of $x_3$. 
By Assumption \ref{ass:local_Stratonovich_strong_weak}\eqref{it:independence_z_variable_Stratonovich}, one has $\p \Big[\nabla_{\h} [ (\phi_{n,\h} \cdot \nabla_{\h}) \wt{P}_n] \Big]=\p_{\h} \Big[\nabla_{\h} [ (\phi_{n,\h} \cdot \nabla_{\h}) \wt{P}_n] \Big]=0$ and therefore
$$
\p\Big[
(\phi_n \cdot \nabla)\big[\nabla_{\h}  \wt{P}_n\big]\Big]=-\sum_{i=1}^2 \p\Big[
 ( \nabla_{\h} \phi^i_n  ) \partial_i \wt{P}_n\Big].
$$
Since $\partial_i\wt{P}_n=(\q[(\phi_n\cdot \nabla) v])^i$, the previous identity shows that the 
second term on the RHS of \eqref{eq:correction_v_step_1} is equivalent to $\p[\Lvp v]\,dt$ as desired.
\end{proof}

It remains to prove 
Theorems \ref{t:local_primitive_Stratonovich_strong_weak} and \ref{t:global_primitive_strong_strong_2}.

\begin{proof}[Proof of Theorem \ref{t:local_primitive_Stratonovich_strong_weak}]
The claim follows from Theorem \ref{t:local_primitive_strong_weak_2} noticing that Assumption \ref{ass:local_Stratonovich_strong_weak} yield Assumption \ref{ass:local_strong_weak_2} with the choice \eqref{eq:coefficients_Stratonovich}. Among others, note that $(\hp_{n}^{\ell,m})_{n\geq 1}=(\partial_{\ell} \phi^m_n)_{n\geq 1}\in L^{3+\delta}(\Dom;\ell^2) $ by 
Assumption \ref{ass:local_Stratonovich_strong_weak}\eqref{it:well_posedness_primitive_phi_psi_smoothness_Stratonovich}.
\end{proof}

\begin{proof}[Proof of Theorem \ref{t:global_primitive_strong_strong_2}]
The claim follows  from Theorem \ref{t:global_primitive_strong_weak_2} noticing that Assumption \ref{ass:global_primitive_inhomogeneous} are satisfied with the choice \eqref{eq:coefficients_Stratonovich} due to Assumption \ref{ass:global_primitive_Stratonovich_strong_weak}. In particular, note that Assumption \ref{ass:local_Stratonovich_strong_weak}\eqref{it:independence_z_variable_Stratonovich} and \ref{ass:global_primitive_Stratonovich_strong_weak}\eqref{it:null_phi_3_Stratonovich_global} ensure that 
Assumption \ref{ass:global_primitive_inhomogeneous}\eqref{it:independence_z_variable_a_ij} and \eqref{it:a_3j_null_boundary} hold, respectively.
\end{proof}

\appendix

\section{Kadlec's formulas}
The following was used to prove Propositions \ref{prop:SMR_2} and \ref{prop:SMR_2_strong}. Below $\nabla_{\h}=(\partial_1,\partial_2)$.

\begin{lemma}[Kadlec's formula]
\label{l:Kadlec_formula}
Let $\Dom=\Tor^2\times (-h,0)$ for some $h>0$. Let $\beta\in \R$ and 
$
f\in H^{2}(\Dom) 
$
be such that 
$\partial_3 f(\cdot,-h)=\partial_3 f(\cdot,0)+\beta f(\cdot,0)=0$ on $\Tor^2$. Then
$$
\sum_{i,j=1}^3 \int_{\Dom}|\partial_{i,j}^2 f |^2 \,dx = \int_{\Dom} |\Delta f|^2 \,dx - 
2\beta \int_{\Tor^2 } |\nabla_{\h} f(\cdot,0)|^2\,dx.
$$
In particular, for all $\varepsilon \in (0,1)$ there exists $C(\varepsilon)>0$ independent of $f$ such that 
$$
\sum_{i,j=1}^3 \int_{\Dom}|\partial_{i,j}^2 f |^2 \,dx\leq (1+\varepsilon)  \int_{\Dom} |\Delta f|^2 \,dx + C_{\varepsilon}
 \int_{\Dom} | f|^2 \,dx.
$$
\end{lemma}

Lemma \ref{l:Kadlec_formula} is well-known to experts and actually holds under more general assumption on the domain. For the sake of completeness we provide a proof.

\begin{proof}[Proof of Lemma \ref{l:Kadlec_formula}]
Since $\Dom$ is a smooth manifold with boundary, 
by a standard density argument we may assume $f\in C^{3}(\overline{\Dom})$. Note that, for $i,j\in \{1,2\}$, integrating by parts we have  due to the periodicity in the horizontal directions that
$$
\int_{\Dom} \partial_{i,j}^2 f \, \partial_{i,j}^2 f\, dx=
-\int_{\Dom} \partial_{i,j,j}^3 f \, \partial_{i} f\, dx=
\int_{\Dom} \partial_{i,i}^2 f \, \partial_{j,j}^2 f\, dx,
$$
and using that 
$\partial_3 f(\cdot,-h)=\partial_3 f(\cdot,0)+\beta f(\cdot,0)=0$ on $\Tor^2$,
\begin{align*}
\int_{\Dom} \partial_{3, j}^2 f\,  \partial_{3, j}^2 f \,dx 
&=- \int_{\Dom} \partial_{3, j ,j}^3 f \, \partial_{3} f \,dx\\ 
&=- \int_{\Tor^2} \Big[\partial_{ j, j}^2 f \,\partial_{3} f\Big]_{x_3=-h}^{x_3=0} \,dx_{\h}+
\int_{\Dom} \partial_{ j,j}^2 f \,\partial_{3,3}^2f\,dx\\
&= \beta \int_{\Tor^2} \partial_{ j j}^2 f(\cdot,0) \, f(\cdot,0) \,dx_{\h}+
\int_{\Dom} \partial_{ j,j}^2 f \,\partial_{3,3}^2f\,dx\\
&= -\beta \int_{\Tor^2} |\partial_{ j} f(\cdot,0)|^2 \,dx_{\h}+
\int_{\Dom} \partial_{ j,j}^2 f \,\partial_{3,3}^2f\,dx.
\end{align*}
Thus
\begin{align*}
&\sum_{i,j=1}^3\|\partial_{i,j}^2 f  \|_{L^2(\Dom)}^2\\
&=
\sum_{i,j=1}^2\|\partial_{i,j}^2 f\|_{L^2(\Dom)}^2+ 2 \sum_{j=1}^2 \|\partial_{3,j}^2 f \|_{L^2(\Dom)}^2
+ \|\partial_{3,3}^2 f \|_{L^2(\Dom)}^2\\
&=\int_{\Dom}\Big( \sum_{i,j=1}^2\partial_{i,i}^2 f\, \partial_{j,j}^2 f 
+ 2\sum_{j=1}^2 \partial_{3,3}^2 f\, \partial_{j,j}^2 f +|\partial_{3,3}^2 f|^2\Big) \,dx-2 \beta \int_{\Tor^2} |\nabla_{\h} f(\cdot,0)|^2\,dx_{\h}\\
&=
\int_{\Dom}|\Delta f|^2\,dx -2 \beta \int_{\Tor^2} |\nabla_{\h} f(\cdot,0)|^2\,dx_{\h}.
\end{align*}
Recalling that $\|\nabla_{\h} f(\cdot,0)\|_{L^2(\Tor^2)}\lesssim \|f\|_{H^{1+s}(\Dom)}$ for any $s>\frac{1}{2}$, the last inequality follows from a standard interpolation argument.
\end{proof}

The above argument can be easily extended to prove the following

\begin{lemma}[Kadlec's formula II]
\label{l:Kadlec_formula_II}
Let $\Dom=\Tor^2\times (-h,0)$ for some $h>0$.  
Assume that there are $a^{i,j}\colon \Dom \rightarrow \R$ for $i,j\in \{1,2,3\}$ and  some $\delta,\ellip\in (0,1)$ and $M\geq 1$ such that
\begin{align}
\label{eq:a_ij_modified_Kadlec}
a^{i,j}=a^{j,i}, \ \ \ \qquad 
\|a^{i,j}\|_{ H^{1,3+\delta}(\Dom)} \leq M, & 
&\text{ for all }i,j\in \{1,2,3\},&
\\
\label{eq:a_3j_trace_regularity_Kadlec}
\|a^{3,j}(\cdot,0)\|_{H^{\frac{1}{2}+\delta}(\Tor^2)} \leq M, & 
&\text{ for all }j\in \{1,2\},&
\\
\label{eq:ellipticity_Kadlec}
\sum_{i,j=1}^3 a^{i,j} \xi_i\xi_j \geq \ellip |\xi|^2,&   &\text{ for all }\xi\in \R^d.&
\end{align}
Let $\beta\in \R$ and 
$
f\in H^{2}(\Dom) 
$
be such that 
$\partial_3 f(\cdot,-h)=\partial_3 f(\cdot,0)+\beta f(\cdot,0)=0$ on $\Tor^2$. 
Then for all $\varepsilon\in (0,1)$ there exists 
$
C(\varepsilon,\ellip,\delta,M)>0
$ 
independent of $f$ such that 
\begin{equation}
\label{eq:claim_Kadlec_2}
\sum_{i,j,k=1}^3 \int_{\Dom}a^{i,j}\partial_{i,k}^2 f \partial_{j,k}^2 f\,dx
\leq (1+\varepsilon) 
\sum_{i,j=1}^3 \int_{\Dom}a^{i,j}\partial_{i,j}^2 f \Delta f \,dx
 + C
 \int_{\Dom} | f|^2 \,dx.
\end{equation}
\end{lemma}

The integrals in \eqref{eq:claim_Kadlec_2} are well-defined since, 
by \eqref{eq:a_ij_modified_Kadlec} and Sobolev embeddings, 
\begin{equation}
\label{eq:a_ij_boundedness_Kadlec_2}
\|a^{i,j}\|_{ L^{\infty}(\Dom)}\lesssim_{\delta} M.
\end{equation}  
The trace $a^{i,j}(\cdot,0)$ is well-defined since $a^{i,j}\in H^{1,3+\delta}(\Dom)$ but the latter only implies $a^{i,j}(\cdot,0)\in H^{1/2,3+\delta}(\Tor^2)$ which is not enough for \eqref{eq:claim_Kadlec_2} to hold.

\begin{proof}
The proof follows from the argument used in Lemma \ref{l:Kadlec_formula}. As above it is enough to consider the case $f\in C^3(\overline{\Dom})$. Without loss of generality we assume that $\delta\in (0,\frac{1}{4})$.

\emph{Step 1: For all $\eta>0$ there exists $C_1(\delta, M,\eta)>0$ such that for all $i,j,k\in \{1,2,3\}$}
$$
\int_{\Dom}a^{i,j}\partial_{i,k}^2 f \partial_{j,k}^2 f\,dx
\leq 
 \Big(\int_{\Dom}a^{i,j}\partial_{i,j}^2 f \partial_{k,k}^2 f \,dx
 +\eta  \int_{\Dom}| D^2 f|^2\,dx\Big)
 +C_1
 \int_{\Dom} | f|^2 \,dx,
$$
\emph{where $D^2 f:=(\partial_{i,j}^2 f)_{i,j=1}^3$ and $|D^2 f|^2:=\sum_{i,j=1}^3|\partial_{i,j}^2 f |^2 $.}

Let us divide the proof of this step into three sub-steps. Note that in case $i=j=k$ the claim of Step 1 follows. Hence we will assume either $i\neq j$ or $j\neq k$.

\emph{Sub-step 1a: Step 1 holds in case $k=3$}. Since either $i\neq 3$ or $j\neq 3$, without loss of generality we may assume $i\neq 3$. Integrating by parts and using that $\partial_3 f(\cdot,-h)=0$, 
\begin{align*}
\int_{\Dom}a^{i,j}\partial_{i,j}^2 f \partial_{3,3}^2 f\,dx
&=
\int_{\Tor^2}a^{i,j}(\cdot,0)\partial_{i,j}^2 f(\cdot,0)\partial_{3} f(\cdot,0)\,dx_{\h}\\
&- 
\int_{\Dom}\partial_3 a^{i,j}\partial_{i,j}^2 f \partial_{3} f\,dx
-\int_{\Dom}a^{i,j}\partial_{i,j,3}^3 f \partial_{3} f\,dx.
\end{align*}
Since $\partial_3 f(\cdot,0)=-\beta f(\cdot,0)$,
$$
\int_{\Tor^2}a^{i,j}(\cdot,0)\partial_{i,j}^2 f(\cdot,0)\partial_{3} f(\cdot,0)\,dx_{\h}
=-\beta \int_{\Tor^2}a^{i,j}(\cdot,0)\partial_{i,j}^2 f(\cdot,0) f(\cdot,0)\,dx_{\h},
$$
and by $i\neq 3$, 
$$
-\int_{\Dom}a^{i,j}\partial_{i,j,3}^3 f \partial_{3} f\,dx=
\int_{\Dom}\partial_i a^{i,j} \partial_{j,3}^2 f \partial_{3} f
+
\int_{\Dom}a^{i,j}\partial_{j,3}^2 f \partial_{i,3}^2 f\,dx.
$$
Thus the above equalities yield
\begin{align*}
&\int_{\Dom}a^{i,j}\partial_{i,j}^2 f \partial_{3,3}^2 f\,dx
=-\beta \int_{\Tor^2}a^{i,j}(\cdot,0)\partial_{i,j}^2 f(\cdot,0) f(\cdot,0)\,dx_{\h}\\
&- 
\int_{\Dom}\partial_3 a^{i,j}\partial_{i,j}^2 f \partial_{3} f\,dx
+\int_{\Dom}\partial_i a^{i,j} \partial_{j,3}^2 f \partial_{3} f
+
\int_{\Dom}a^{i,j}\partial_{j,3}^2 f \partial_{i,3}^2 f\,dx\\
&=: I_1+I_2+I_3+
\int_{\Dom}a^{i,j}\partial_{j,3}^2 f \partial_{i,3}^2 f\,dx.
\end{align*}

Let us show that the additional term in the previous estimate are of lower order. Recall that the trace operator $H^{s+\frac{1}{2}}(\Dom)\ni f\mapsto f(\cdot,0)\in H^{s}(\Tor^2)$ is bounded for all $s>0$.
Thus, the boundary term is lower order since 
\begin{align*}
|I_1|
&\lesssim_{\beta} 
\|a^{i,j}(\cdot,0) f(\cdot,0)\|_{H^{\frac{1}{2}+\delta}(\Tor^2)}\|\partial_{i,j}^2 f(\cdot,0)\|_{H^{-\frac{1}{2}-\delta}(\Tor^2)}\\
&\lesssim_{\delta}  
\|a^{i,j}(\cdot,0) f(\cdot,0)\|_{H^{\frac{1}{2}+\delta}(\Tor^2)}\| f(\cdot,0)\|_{H^{\frac{3}{2}-\delta}(\Tor^2)}\\
&\lesssim_{\delta}  
\Big( 
\|a^{i,j}(\cdot,0)\|_{H^{\frac{1}{2}+\delta}(\Tor^2)}\| f(\cdot,0)\|_{L^{\infty}(\Tor^2)}\\
&\ \ \  +
\|a^{i,j}(\cdot,0)\|_{L^{\infty}(\Tor^2)}\| f(\cdot,0)\|_{H^{\frac{1}{2}+\delta}(\Tor^2)}
\Big) \| f(\cdot,0)\|_{H^{\frac{3}{2}-\delta}(\Tor^2)}
\lesssim_{\delta,M} \|f\|_{H^{2-\delta}}^2,
\end{align*} 
where in the last inequality we used \eqref{eq:a_3j_trace_regularity_Kadlec}, \eqref{eq:a_ij_boundedness_Kadlec_2} and that  
$$
\|f(\cdot,0)\|_{L^{\infty}(\Tor^2)}\lesssim_{\delta}\|f(\cdot,0)\|_{H^{1+\delta}(\Tor^2)} \lesssim \|f\|_{H^{\frac{3}{2}+\delta}}
\stackrel{(\delta<\frac{1}{4})}{\lesssim} \|f\|_{H^{2-\delta}}.
$$
Let $\varrho\in (2,6)$ be such that $\frac{1}{3+\delta} + \frac{1}{\varrho}=\frac{1}{2}$. Then there exists $\delta'>0$ such that $H^{1-\delta'}(\Dom)\embed L^{\varrho}(\Dom)$, and by H\"{o}lder inequality,
\begin{equation}
\label{eq:I_2_3_appendix_estimate}
|I_2|+|I_3|\lesssim  
\big\| | \nabla a^{i,j}| |\nabla f|\big\|_{L^2}\|D^2 f\|_{L^2} 
\lesssim \|\nabla a^{i,j} \|_{L^{3+\delta}}\| f\|_{H^{2-\delta'}} \|D^2 f\|_{L^2}.
\end{equation}

The previous estimates show that $I_1,I_2$ and $I_3$ are lower order terms w.r.t.\ $\|f\|_{H^2}$. Thus the claim of Step 1a follows from interpolation and Young inequalities. 

\emph{Sub-step 1b: Step 1 holds in case $i=j=3$ and $k\in \{1,2\}$}. Integrating by parts we have
\begin{align*}
\int_{\Dom}a^{3,3}\partial_{3,3}^2 f \partial_{k,k}^2 f\,dx
&= 
-\int_{\Dom}a^{3,3}\partial_{3,3,k}^3 f \partial_{k} f\,dx
-\int_{\Dom}\partial_k a^{3,3}\partial_{3,3}^2 f \partial_{k} f\,dx\\
&
\stackrel{(i)}{=}-
\int_{\Tor^2}a^{3,3}(\cdot,0)\partial_{3,k}^2 f(\cdot,0) \partial_{k} f(\cdot,0)\,dx_{\h}
+\int_{\Dom}\partial_3 a^{3,3}\partial_{3,k}^2 f \partial_{k} f\,dx\\
&\qquad \qquad \qquad
+\int_{\Dom} a^{3,3}|\partial_{3,k}^2 f|^2\,dx
-\int_{\Dom}\partial_k a^{3,3}\partial_{3,3}^2 f \partial_{k} f\,dx\\
&
\stackrel{(ii)}{=}
-\beta \int_{\Dom}a^{3,3}(\cdot,0)|\partial_{k} f(\cdot,0) |^2\,dx_{\h}
+\int_{\Dom}\partial_3 a^{3,3}\partial_{3,k}^2 f \partial_{k} f\,dx\\
&\qquad \qquad \qquad
+\int_{\Dom} a^{3,3}|\partial_{3,k}^2 f|^2\,dx
-\int_{\Dom}\partial_k a^{3,3}\partial_{3,3}^2 f \partial_{k} f\,dx
\end{align*}
where in $(i)$ and $(ii)$ we used that  $\partial_{3,k}f (\cdot,-h)=0$ and $\partial_{3,k}^2 f(\cdot,0)= -\beta \partial_k f(\cdot,0)$ on $\Tor^2$, respectively.  

Note that, by \eqref{eq:a_ij_boundedness_Kadlec_2}, for all $s\in (0,\frac{1}{2})$, 
$$
\Big|\int_{\Dom}a^{3,3}(\cdot,0)|\partial_{k} f(\cdot,0) |^2\,dx_{\h}\Big|
\leq \|a^{i,j}\|_{L^{\infty}} \|\nabla f(\cdot,0)\|_{L^2(\Tor^2)}^2\lesssim_{\delta}  M \|f\|_{H^{3/2+s}}^2
$$
Now, the claim follows from the previous inequalities by repeating the argument in \eqref{eq:I_2_3_appendix_estimate} to estimate the remaining terms.

\emph{Sub-step 1c: Step 1 holds in case $k\in \{1,2\}$ and either $i\neq 3$ or $j\neq 3$}. Without loss of generality we assume $i\neq 3$. 
In the latter case, one can integrate by parts on the $k$- and $i$-coordinates and therefore no boundary terms appear. The claim of Substep 1c follows by repeating the estimates in \eqref{eq:I_2_3_appendix_estimate}.

\emph{Step 2: Proof of \eqref{eq:claim_Kadlec_2}}. Fix $\varepsilon\in (0,1)$ and choose $\eta>0$ such that $1+\varepsilon=
\big(1-\frac{27\eta}{\ellip}\big)^{-1}$. 
By ellipticity \eqref{eq:ellipticity_Kadlec},
$$
\eta 
\int_{\Dom}|D f|^2\,dx
\leq \frac{\eta}{\ellip}
\sum_{i,j,k=1}^3 \int_{\Dom}a^{i,j}\partial_{i,k}^2 f \partial_{j,k}^2 f\,dx.
$$
Thus, by Step 1 and summing over $i,j,k\in \{1,2,3\}$, we have
$$
\Big(1-\frac{27\eta}{\ellip}\Big)
\sum_{i,j,k=1}^3 \int_{\Dom}a^{i,j}\partial_{i,k}^2 f \partial_{j,k}^2 f\,dx
\leq 
\sum_{i,j=1}^3 \int_{\Dom}a^{i,j}\partial_{i,j}^2 f \Delta f \,dx
 + C_{2}
 \int_{\Dom} | f|^2 \,dx
$$
where $C_2:=27C_1$.
The above choice of $\eta$ yields \eqref{eq:claim_Kadlec_2} with $C:=C_2\big(1-\frac{27\eta}{\ellip}\big)^{-1}$.
\end{proof}


\bibliographystyle{alpha-sort}
\bibliography{literature}

\end{document}